\documentclass{aims-mod}
\usepackage{amsmath}
\usepackage{paralist}
\usepackage{graphics}
\usepackage{epsfig}
\usepackage{graphicx}
\usepackage{epstopdf}
\usepackage[colorlinks=true]{hyperref}
\hypersetup{breaklinks, linktocpage, colorlinks, bookmarksnumbered, pdfstartview={XYZ null null 1.00}, urlcolor=blue, citecolor=red}

\usepackage{stmaryrd}
\usepackage{enumitem}
\usepackage{tikz}
\usepackage{cite}

\newtheorem{theorem}{Theorem}[section]
\newtheorem{corollary}[theorem]{Corollary}

\newtheorem{lemma}[theorem]{Lemma}
\newtheorem{proposition}[theorem]{Proposition}

\theoremstyle{definition}
\newtheorem{definition}[theorem]{Definition}
\newtheorem{remark}[theorem]{Remark}

\DeclareMathOperator{\spr}{\rho} 
\DeclareMathOperator{\id}{Id} 
\DeclareMathOperator{\Ker}{Ker} 
\DeclareMathOperator{\rank}{rk} 
\DeclareMathOperator{\range}{Ran} 
\DeclareMathOperator{\Span}{Span} 
\DeclareMathOperator*{\esssup}{ess\,sup} 
\DeclareMathOperator{\Vol}{Vol} 
\DeclareMathOperator{\diag}{diag} 
\newcommand{\abs}[1]{\left\lvert #1\right \rvert} 
\newcommand{\norm}[1]{\left\lVert #1\right\lVert} 
\newcommand{\pad}[2]{\frac{\partial #1}{\partial #2}} 

\title[Stability of difference equations and applications]
      {Stability of non-autonomous difference equations with applications to transport and wave propagation on networks}

\author[Y. Chitour, G. Mazanti, and M. Sigalotti]{}


\hypersetup{pdftitle={Stability of non-autonomous difference equations with applications to transport and wave propagation on networks}, pdfauthor={Yacine Chitour, Guilherme Mazanti, Mario Sigalotti}, pdfkeywords={difference equations, switching systems, transport equation, wave propagation, networks, exponential stability}}

\email{yacine.chitour@lss.supelec.fr}
\email{guilherme.mazanti@polytechnique.edu}
\email{mario.sigalotti@inria.fr}

\thanks{This research was partially supported by the iCODE Institute, research project of the IDEX Paris-Saclay, and by the Hadamard Mathematics LabEx (LMH) through the grant number ANR-11-LABX-0056-LMH in the ``Programme des Investissements d'Avenir''.}

\begin{document}
\maketitle

\centerline{\scshape Yacine Chitour }
\medskip
{\footnotesize
 \centerline{Laboratoire des Signaux et Syst\`emes}
   \centerline{Universit\'e Paris Sud, CNRS, CentraleSup\'elec}
   \centerline{91192 Gif-sur-Yvette, France}
}

\medskip

\centerline{\scshape Guilherme Mazanti}
\medskip
{\footnotesize
 \centerline{CMAP, \'Ecole Polytechnique}
   \centerline{Inria Saclay, GECO Team}
   \centerline{91120 Palaiseau, France}
}

\medskip

\centerline{\scshape Mario Sigalotti}
\medskip
{\footnotesize
 \centerline{Inria Saclay, GECO Team}
   \centerline{CMAP, \'Ecole Polytechnique}
   \centerline{91128 Palaiseau, France}
}

\bigskip

\begin{abstract}
In this paper, we address the stability of transport systems and wave propagation on networks with time-varying parameters. We do so by reformulating these systems as non-autonomous difference equations and by providing a suitable representation of their solutions in terms of their initial conditions and some time-dependent matrix coefficients. This enables us to characterize the asymptotic behavior of solutions in terms of such coefficients. In the case of difference equations with arbitrary switching, we obtain a delay-independent generalization of the well-known criterion for autonomous systems due to Hale and Silkowski. As a consequence, we show that exponential stability of transport systems and wave propagation on networks is robust with respect to variations of the lengths of the edges of the network preserving their rational dependence structure. This leads to our main result: the wave equation on a network with arbitrarily switching damping at external vertices is exponentially stable if and only if the network is a tree and the damping is bounded away from zero at all external vertices but at most one.
\end{abstract}

\tableofcontents


\section{Introduction}

Dynamics on networks has generated in the past decades an intense research activity within the PDE control community \cite{Dager2006Wave, Ali2001Partial, Bressan2014Flows, Hante2009Modeling, Gugat2012Well}. In particular, stability and stabilization of transport and wave propagation on networks raise challenging questions on the relationships between the asymptotic-in-time behavior of solutions on the one hand and, on the other hand, the topology of the network, its interconnection and damping laws at the vertices, and the rational dependence of the transit times on the network edges \cite{Valein2009Stabilization, Chitour2015Persistently, Bastin2007Lyapunov, Coron2008Dissipative, Zuazua2013Control, Alabau2015Finite}. A case of special interest is when some coefficients of the system are time-dependent and switch arbitrarily within a given set \cite{Gugat2010Stars, Prieur2014Stability, Amin2012Exponential}.

In this paper, we address stability issues first for transport systems with time-dependent transmission conditions and then for wave propagation on networks with time-dependent damping terms. When the time-dependent coefficients switch arbitrarily in a given bounded set, we prove that the stability is robust with respect to variations of the lengths of the edges of the network preserving their rational dependence structure (see Corollary \ref{CoroSilkTransport} for transport and Corollary \ref{CoroWaveEquivLambdaL} for wave propagation). Such robustness enables us to get the main result of the paper, namely a necessary and sufficient criterion for exponential stability of wave propagation on networks: exponential stability holds for a network if and only if it is a tree and the damping is bounded away from zero at all external vertices but at most one (Theorem \ref{TheoWaveStabIffTopo}).

We address these issues by formulating them within the framework of non-autonomous linear difference equations
\begin{equation}
\label{IntroDelay}
u(t) = \sum_{j=1}^N A_j(t) u(t - \Lambda_j), \qquad u(t) \in \mathbb C^d,\; (\Lambda_1, \dotsc, \Lambda_N) \in (\mathbb R_+^\ast)^N.
\end{equation}
This standard approach relies on the d'Alembert decomposition and classical transformations of hyperbolic systems of PDEs into delay differential-difference equations \cite{Cooke1968Differential, Slemrod1971Nonexistence, Miranker1961Periodic, Brayton1966Bifurcation, Kloss2012Flow, Fridman2010Bounds}. Here, stability is meant uniformly with respect to the matrix-valued function $A(\cdot) = (A_1(\cdot), \dotsc, A_N(\cdot))$ belonging to a given class $\mathcal A$.

In the autonomous case, Equation \eqref{IntroDelay} has a long history and its stability is completely characterized using Laplace transform techniques by the celebrated Hale--Silkowski criterion (see e.g. \cite[Theorem 5.2]{Avellar1980Zeros}, \cite[Chapter 9, Theorem 6.1]{Hale1993Introduction}). The latter can be formulated as follows: if $\Lambda_1, \dotsc, \Lambda_N$ are rationally independent, then all solutions of $u(t) = \sum_{j=1}^N A_j u(t - \Lambda_j)$ tend exponentially to zero as $t$ tends to infinity if and only if $\rho_{\mathrm{HS}}(A) < 1$, where
\begin{equation}
\label{RhoHSIntro}
\rho_{\mathrm{HS}}(A) = \max_{(\theta_1, \dotsc, \theta_N) \in [0, 2\pi]^N} \rho\left(\sum_{j=1}^N A_j e^{i \theta_j}\right)
\end{equation}
and $\rho(\cdot)$ denotes the spectral radius. Notice that the latter condition does not depend on $\Lambda_1, \dotsc, \Lambda_N$. The Hale--Silkowski criterion actually says more, namely the striking fact that exponential stability for some $\Lambda_1, \dotsc, \Lambda_N$ rationally independent is equivalent to exponential stability for any choice of positive delays $L_1, \dotsc, L_N$. This criterion can be used to evaluate the maximal Lyapunov exponent associated with $u(t) = \sum_{j=1}^N A_j u(t - \Lambda_j)$, i.e., the infimum over the exponential bounds for the corresponding semigroup. A remarkable feature of the Hale--Silkowski criterion is that, contrarily to the maximal Lyapunov exponent, it does not involve taking limits as time tends to infinity. An extension of these results has been obtained in \cite{Michiels2009Strong} for the case where $\Lambda_1, \dotsc, \Lambda_N$ are not assumed to be rationally independent. 

The non-autonomous case turns out to be more difficult since one does not have a general characterization of the exponential stability of \eqref{IntroDelay} not involving limits as time tends to infinity. To illustrate that, consider the simple case $N = 1$ of a single delay and $\mathcal A = L^\infty(\mathbb R, \mathfrak B)$ where $\mathfrak B$ is a bounded set of $d \times d$ matrices. Then the stability of \eqref{IntroDelay} is equivalent to that of the discrete-time switched system $u_{n+1} = A_n u_n$ where $A_n \in \mathfrak B$, and it is characterized by the joint spectral radius of $\mathfrak B$ (see for instance \cite[Section 2.2]{Jungers2009Joint} and references therein) for which there is not yet a general characterization not involving limits as $n$ tends to infinity.

Up to our knowledge, the only results regarding the stability of non-autonomous difference equations were obtained in \cite{Ngoc2014Exponential}, where sufficient conditions for stability are deduced from Perron--Frobenius Theorem. Our approach is rather based on a trajectory analysis relying on a suitable representation for solutions of \eqref{IntroDelay}, which expresses the solution $u(t)$ at time $t$ as a linear combination of the initial condition $u_0$ evaluated at finitely many points identified explicitly. The matrix coefficients, denoted by $\Theta$, are obtained in terms of the functions $A_1(\cdot), \dotsc, A_N(\cdot)$ and take into account the rational dependence structure of $\Lambda_1, \dotsc, \Lambda_N$ (Proposition \ref{PropSolExpliciteAdpBis}). This representation provides a correspondence between the asymptotic behavior of solutions of \eqref{IntroDelay}, uniformly with respect to the initial condition $u_0$ and $A(\cdot) \in \mathcal A$, and that of the matrix coefficients $\Theta$ uniformly with respect to $A(\cdot) \in \mathcal A$ (Theorem \ref{TheoP0EquivP1}). In the case where $\mathcal A = L^\infty(\mathbb R, \mathfrak B)$ for some bounded set $\mathfrak B$ of $N$-tuples of $d \times d$ matrices, we extend the results of \cite{Michiels2009Strong}, replacing $\rho_{\mathrm{HS}}$ in the Hale--Silkowski criterion by a generalization $\mu$ of the joint spectral radius. As a consequence of our analysis, and despite the lack of a closed and delay-independent formula for $\mu$ analogous to \eqref{RhoHSIntro}, we are able to show that stability for some $N$-tuple $\Lambda = (\Lambda_1, \dotsc, \Lambda_N)$ is equivalent to stability for any choice of $N$-tuple $(L_1, \dotsc, L_N)$ having the same rational dependence structure as $\Lambda$ (Corollaries \ref{CoroSilkRho0} and \ref{CoroSilkRhoHS}).

The structure of the paper goes as follows. Section \ref{SecNotations} provides the main notations and definitions used in this paper. Difference equations of the form \eqref{IntroDelay} are discussed in Section \ref{SecDiffEqns}. We start by establishing the well-posedness of the Cauchy problem and a representation formula for solutions in Sections \ref{SecCauchy} and \ref{SecExplicit}. Stability criteria are given in Section \ref{SecStability} in terms of convergence of the coefficients and specified to the cases of shift-invariant classes $\mathcal A$ and arbitrary switching. In the latter case, we provide the above discussed generalization of the Hale--Silkowski criterion. Applications to transport equations are developed in Section \ref{SecTransport} by exhibiting a correspondence with difference equations of the type \eqref{IntroDelay}. Thanks to the d'Alembert decomposition, results for transport equations are transposed to wave propagation on networks in Section \ref{SecWave}. The topological characterization of exponential stability is given in Section \ref{SecStabWave}.


\section{Notations and definitions}
\label{SecNotations}

In this paper, we denote by $\mathbb N$ and $\mathbb N^\ast$ the set of nonnegative and positive integers respectively, $\mathbb R_+ = \left[0, +\infty\right)$ and $\mathbb R_+^\ast = (0, +\infty)$. For $a, b \in \mathbb R$, let $\llbracket a, b \rrbracket = [a, b] \cap \mathbb Z$, with the convention that $[a, b] = \emptyset$ if $a > b$. The closure of a set $F$ is denoted by $\overline F$. If $x \in \mathbb R$ and $F \subset \mathbb R$, we use $x + F$ to denote the set $\{x + y \;|\: y \in F\}$.

We use $\# F$ and $\delta_{ij}$ to denote, respectively, the cardinality of a set $F$ and the Kronecker symbol of $i, j$. For $x \in \mathbb R$, we use $x_\pm$ to denote $\max(\pm x, 0)$ and we extend this notation componentwise to vectors. For $x \in \mathbb R^d$, we use $x_{\min}$ and $x_{\max}$ to denote the smallest and the largest components of $x$, respectively.

If $K$ is a subset of $\mathbb C$ and $d, m \in \mathbb N$, the set of $d \times m$ matrices with coefficients in $K$ is denoted by $\mathcal M_{d, m}(K)$, or simply by $\mathcal M_d(K)$ when $d = m$. The identity matrix in $\mathcal M_d(\mathbb C)$ is denoted by $\id_d$. We use $e_1, \dotsc, e_d$ to denote the canonical basis of $\mathbb C^d$, i.e., $e_i = (\delta_{ij})_{j \in \llbracket 1, d\rrbracket}$ for $i \in \llbracket 1, d\rrbracket$. For $p \in [1, +\infty]$, $\abs{\cdot}_{p}$ indicates both the $\ell^p$-norm in $\mathbb C^d$ and the induced matrix norm in $\mathcal M_d(\mathbb C)$. We use $\spr(A)$ to denote the spectral radius of a matrix $A \in \mathcal M_d(\mathbb C)$, i.e., the maximum of $\abs{\lambda}$ with $\lambda$ eigenvalue of $A$. The range and kernel of a matrix $A$ are denoted by $\range A$ and $\Ker A$ respectively, and $\rank(A)$ denotes the dimension of $\range A$. Given $A_1, \dotsc, A_N \in \mathcal M_d(\mathbb C)$, we denote by $\prod_{i=1}^N A_i$ the ordered product $A_1 \dotsm A_N$.

All Banach and Hilbert spaces are supposed to be complex. For $p \in [1, +\infty]$, we use $L^p$ to denote the usual Lebesgue spaces of $p$-integrable functions and $W^{k, p}$ the Sobolev spaces of $k$-times weakly differentiable functions with derivatives in $L^p$.

A subset $\mathcal A$ of $L^\infty_{\mathrm{loc}}(\mathbb R, \mathcal M_d(\mathbb C)^N)$ is said to be \emph{uniformly locally bounded} if, for every compact time interval $I$, $\sup_{A \in \mathcal A} \norm{A}_{L^\infty(I, \mathcal M_d(\mathbb C)^N)}$ is finite. We say that $\mathcal A$ is \emph{shift-invariant} if $A(\cdot + t) \in \mathcal A$ for every $A \in \mathcal A$ and $t \in \mathbb R$.

Throughout the paper, we will use the indices $\delta$, $\tau$ and $\omega$ in the notations of systems and functional spaces when dealing, respectively, with difference equations, transport systems and wave propagation.


\section{Difference equations}
\label{SecDiffEqns}

Let $N, d \in \mathbb N^\ast$, $\Lambda = (\Lambda_1, \dotsc, \Lambda_N) \in (\mathbb R_+^\ast)^N$, and consider the system of time-dependent difference equations
\begin{equation}
\label{SystDelay}
\Sigma_{\delta}(\Lambda, A): \qquad
u(t) = \sum_{j=1}^N A_j(t) u(t - \Lambda_j).
\end{equation}
Here, $u(t) \in \mathbb C^d$ and $A = (A_1, \dotsc, A_N): \mathbb R \to \mathcal M_d(\mathbb C)^N$.

\subsection{Well-posedness of the Cauchy problem}
\label{SecCauchy}

In this section, we show existence and uniqueness of solutions of the Cauchy problem associated with \eqref{SystDelay}. We also consider the regularity of these solutions in terms of the initial condition and $A(\cdot)$.

\begin{definition}
Let $u_0: \left[-\Lambda_{\max}, 0\right) \to \mathbb C^d$ and $A = (A_1, \dotsc, A_N): \mathbb R \to \mathcal M_d(\mathbb C)^N$. We say that $u: \left[-\Lambda_{\max}, +\infty\right) \to \mathbb C^d$ is a \emph{solution} of $\Sigma_{\delta}(\Lambda, A)$ with initial condition $u_0$ if it satisfies \eqref{SystDelay} for every $t \in \mathbb R_+$ and $u(t) = u_0(t)$ for $t \in \left[-\Lambda_{\max}, 0\right)$. In this case, we set, for $t \geq 0$, $u_t = u(\cdot + t)|_{\left[-\Lambda_{\max}, 0\right)}$.
\end{definition}

We have the following result.

\begin{proposition}
\label{PropExistUnique}
Let $u_0: \left[-\Lambda_{\max}, 0\right) \to \mathbb C^d$ and $A = (A_1, \dotsc, A_N): \mathbb R \to \mathcal M_d(\mathbb C)^N$. Then $\Sigma_{\delta}(\Lambda, A)$ admits a unique solution $u: \left[-\Lambda_{\max}, +\infty\right) \to \mathbb C^d$ with initial condition $u_0$.
\end{proposition}

\begin{proof}
It suffices to build the solution $u$ on $\left[-\Lambda_{\max}, \Lambda_{\min}\right)$ and then complete its construction on $\left[\Lambda_{\min}, +\infty\right)$ by a standard inductive argument.

Suppose that $u: \left[-\Lambda_{\max}, \Lambda_{\min}\right) \to \mathbb C^d$ is a solution of $\Sigma_{\delta}(\Lambda, A)$ with initial condition $u_0$. Then, by \eqref{SystDelay}, we have
\begin{equation}
\label{SolSmallTime}
u(t) = 
\left\{
\begin{aligned}
& \sum_{j=1}^N A_j(t) u_0(t - \Lambda_j), & & \text{ if } 0 \leq t < \Lambda_{\min}, \\
& u_0(t), & & \text{ if } -\Lambda_{\max} \leq t < 0.
\end{aligned}
\right.
\end{equation}
Since the right-hand side is uniquely defined in terms of $u_0$ and $A$, we obtain the uniqueness of the solution. Conversely, if $u: \left[-\Lambda_{\max}, \Lambda_{\min}\right) \to \mathbb C^d$ is defined by \eqref{SolSmallTime}, then \eqref{SystDelay} clearly holds for $t \in \left[-\Lambda_{\max}, \Lambda_{\min}\right)$ and thus $u$ is a solution of $\Sigma_{\delta}(\Lambda, A)$.
\end{proof}

\begin{definition}
For $p \in \left[1, +\infty\right]$, we use $\mathsf X_p^\delta$ to denote the Banach space $\mathsf X_p^\delta = L^p([-\Lambda_{\max}, 0], \mathbb C^d)$ endowed with the usual $L^p$-norm denoted by $\norm{\cdot}_{p}$.
\end{definition}

\begin{remark}
\label{RemkRegular}
If $u_0, v_0: \left[-\Lambda_{\max}, 0\right) \to \mathbb C^d$ are such that $u_0 = v_0$ almost everywhere on $\left[-\Lambda_{\max}, 0\right)$ and $A, B: \mathbb R \to \mathcal M_d(\mathbb C)^N$ are such that $A = B$ almost everywhere on $\mathbb R_+$, then it follows from \eqref{SolSmallTime} that the solutions $u, v: \left[-\Lambda_{\max}, +\infty\right) \to \mathbb C^d$ associated respectively with $A$, $u_0$ and $B$, $v_0$ satisfy $u = v$ almost everywhere on $\left[-\Lambda_{\max}, +\infty\right)$. In particular, for initial conditions in $\mathsf X_p^\delta$, $p \in [1, +\infty]$, we still have existence and uniqueness of solutions, now in the sense of functions defined almost everywhere. If moreover $A \in L^\infty_{\mathrm{loc}}(\mathbb R, \mathcal M_d(\mathbb C)^N)$, it easily follows from \eqref{SolSmallTime} that the corresponding solution $u$ of $\Sigma_{\delta}(\Lambda, A)$ satisfies $u \in L^p_{\mathrm{loc}}(\left[-\Lambda_{\max}, +\infty\right), \mathbb C^d)$.
\end{remark}

\begin{proposition}
\label{PropRegular}
Let $p \in \left[1, +\infty\right)$, $u_0 \in \mathsf X_p^\delta$, $A \in L^\infty_{\mathrm{loc}}(\mathbb R, \mathcal M_d(\mathbb C)^N)$, and $u$ be the solution of $\Sigma_{\delta}(\Lambda, A)$ with initial condition $u_0$. Then the $\mathsf X_p^\delta$-valued mapping $t \mapsto u_t$ defined on $\mathbb R_+$ is continuous.
\end{proposition}

\begin{proof}
By Remark \ref{RemkRegular}, $u_t \in \mathsf X_p^\delta$ for every $t \in \mathbb R_+$. Since $u_t(s) = u(s + t)$ for $s \in \left[-\Lambda_{\max}, 0\right)$, the continuity of $t \mapsto u_t$ follows from the continuity of translations in $L^p$ (see, for instance, \cite[Theorem 9.5]{Rudin1987Real}).
\end{proof}

\begin{remark}
The conclusion of Proposition \ref{PropRegular} does not hold for $p = +\infty$ since translations in $L^\infty$ are not continuous.
\end{remark}

\subsection{Representation formula for the solution}
\label{SecExplicit}

When $t \in \left[0, \Lambda_{\min}\right)$, Equation \eqref{SolSmallTime} yields $u(t)$ in terms of the initial condition $u_0$. If $t \geq \Lambda_{\min}$, we use \eqref{SystDelay} to express the solution $u$ at time $t$ in terms of its values on previous times $t - \Lambda_j$, and, for each $j$ such that $t > \Lambda_j$, we can reapply \eqref{SystDelay} at the time $t - \Lambda_j$ to obtain the expression of $u(t - \Lambda_j)$ in terms of $u$ evaluated at previous times. By proceeding inductively, we can obtain an explicit expression for $u$ in terms of $u_0$. For that purpose, let us introduce some notations.

\begin{definition}
\begin{enumerate}[label={\bf \roman*.}, ref={(\roman*)}]
\item An \emph{increasing path} (in $\mathbb N^N$) is a finite sequence of points $(\mathbf q_k)_{k=1}^n$ in $\mathbb N^N$ such that, for $k \in \llbracket 1, n-1\rrbracket$, $\mathbf q_{k+1}$ is obtained from $\mathbf q_k$ by adding $1$ to exactly one of the coordinates of $\mathbf q_k$. For $n \in \mathbb N^\ast$ and $v = (v_1, \dotsc, v_n) \in \llbracket 1, N\rrbracket^n$, we use $(\mathbf p_v(k))_{k=1}^{n+1}$ to denote the increasing path defined by
\[
\mathbf p_v(k) = \sum_{j=1}^{k-1} e_{v_j}.
\]

\item For $\mathbf n \in \mathbb N^N \setminus \{0\}$, we use $V_{\mathbf n}$ to denote the set
\begin{equation*}
V_{\mathbf n} = \left\{(v_1, \dotsc, v_{\abs{\mathbf n}_{1}}) \in \llbracket 1, N\rrbracket^{\abs{\mathbf n}_{1}} \;\middle|\: \mathbf p_v(\abs{\mathbf n}_1 + 1) = \mathbf n\right\},
\end{equation*}
i.e., $V_{\mathbf n}$ can be seen as the set of all increasing paths from $0$ to $\mathbf n$.

\item For $A = (A_1, \dotsc, A_N): \mathbb R \to \mathcal M_d(\mathbb C)^N$, $\Lambda = (\Lambda_1, \dotsc, \Lambda_N) \in (\mathbb R_+^\ast)^N$, $\mathbf n \in \mathbb Z^N$ and $t \in \mathbb R$, we define the matrix $\Xi^{\Lambda, A}_{\mathbf n, t} \in \mathcal M_d(\mathbb C)$ inductively by
\begin{equation}
\label{DefiXi}
\Xi^{\Lambda, A}_{\mathbf n, t} = 
\left\{
\begin{aligned}
& 0, & & \text{ if } \mathbf n \in \mathbb Z^N \setminus \mathbb N^N, \\
& \id_d, & & \text{ if } \mathbf n = 0, \\
& \sum_{k=1}^N A_k(t) \Xi_{\mathbf n - e_k, t - \Lambda_k}^{\Lambda, A}, & & \text{ if } \mathbf n \in \mathbb N^N \setminus \{0\}. \\
\end{aligned}
\right.
\end{equation}
We omit $\Lambda$, $A$ or both from the notation $\Xi_{\mathbf n, t}^{\Lambda, A}$ when they are clear from the context.
\end{enumerate}
\end{definition}

The following result provides a way to write $\Xi_{\mathbf n, t}$ as a sum over $V_{\mathbf n}$ and as an alternative recursion formula.

\begin{proposition}
For every $\mathbf n \in \mathbb N^N \setminus \{0\}$ and $t \in \mathbb R$, we have
\begin{equation}
\label{XiExplicit}
\Xi_{\mathbf n, t}^{\Lambda, A} = \sum_{v \in V_{\mathbf n}} \prod_{k=1}^{\abs{\mathbf n}_{1}} A_{v_k}\left(t - \Lambda \cdot \mathbf p_v(k)\right)
\end{equation}
and
\begin{equation}
\label{XiRecurrence}
\Xi_{\mathbf n, t}^{\Lambda, A} = \sum_{k=1}^N \Xi_{\mathbf n - e_k, t}^{\Lambda, A} A_k(t - \Lambda \cdot \mathbf n + \Lambda_k).
\end{equation}
\end{proposition}

\begin{proof}
We prove \eqref{XiExplicit} by induction over $\abs{\mathbf n}_{1}$. If $\mathbf n = e_i$ for some $i \in \llbracket 1, N\rrbracket$, we have
\[
\sum_{v \in V_{e_i}} \prod_{k=1}^{1} A_{v_k}(t) = A_i(t) = \Xi_{e_i, t}.
\]
Let $R \in \mathbb N^\ast$ be such that \eqref{XiExplicit} holds for every $\mathbf n \in \mathbb N^N$ with $\abs{\mathbf n}_{1} = R$ and $t \in \mathbb R$. If $\mathbf n \in \mathbb N^N$ is such that $\abs{\mathbf n}_{1} = R + 1$ and $t \in \mathbb R$, we have, by \eqref{DefiXi} and the induction hypothesis, that
\[
\begin{split}
\Xi_{\mathbf n, t} & = \sum_{\substack{k=1 \\ n_k \geq 1}}^N A_k(t) \Xi_{\mathbf n - e_k, t - \Lambda_k} = \sum_{\substack{k=1 \\ n_k \geq 1}}^N \sum_{v \in V_{\mathbf n - e_k}} A_k(t) \prod_{r=1}^{\abs{\mathbf n}_{1} - 1} A_{v_r}\left(t - \Lambda_k - \Lambda \cdot \mathbf p_v(r)\right) \\
 & = \sum_{v \in V_{\mathbf n}} \prod_{r=1}^{\abs{\mathbf n}_{1}} A_{v_r}(t - \Lambda \cdot \mathbf p_v(r)),
\end{split}
\]
where we use that $V_{\mathbf n} = \{(k, v) \;|\: k \in \llbracket 1, N\rrbracket,\; n_k \geq 1,\; v \in V_{\mathbf n - e_k}\}$ and that $e_k + \mathbf p_v(r) = \mathbf p_{(k, v)}(r+1)$. This establishes \eqref{XiExplicit}.

We now turn to the proof of \eqref{XiRecurrence}. Since $\Xi_{e_j, t} = A_j(t)$, \eqref{XiRecurrence} is satisfied for $\mathbf n = e_j$, $j \in \llbracket 1, N\rrbracket$. For $\mathbf n \in \mathbb N^N$ with $\abs{\mathbf n}_1 \geq 2$, the set $V_{\mathbf n}$ can be written as
\[V_{\mathbf n} = \{(v, k) \;|\: k \in \llbracket 1, N\rrbracket,\; n_k \geq 1,\; v \in V_{\mathbf n - e_k}\},\]
and thus, by \eqref{XiExplicit}, we have
\[
\begin{split}
\Xi_{\mathbf n, t} & = \sum_{\substack{k = 1 \\ n_k \geq 1}}^N \sum_{v \in V_{\mathbf n - e_k}} \left[\prod_{r=1}^{\abs{\mathbf n}_{1} - 1} A_{v_r}\left(t - \Lambda \cdot \mathbf p_v(r)\right)\right] A_k\left(t - \Lambda \cdot \mathbf p_v(\abs{\mathbf n}_1)\right) \\
 & = \sum_{\substack{k = 1 \\ n_k \geq 1}}^N \sum_{v \in V_{\mathbf n - e_k}} \left[\prod_{r=1}^{\abs{\mathbf n}_{1} - 1} A_{v_r}\left(t - \Lambda \cdot \mathbf p_v(r)\right)\right] A_k(t - \Lambda \cdot \mathbf n + \Lambda_k) \\
 & = \sum_{k = 1}^N \Xi_{\mathbf n - e_k, t} A_k(t - \Lambda \cdot \mathbf n + \Lambda_k) \qedhere.
\end{split}
\]
\end{proof}

In order to take into account the relations of rational dependence of $\Lambda_1, \dotsc, \Lambda_N \in \mathbb R_+^\ast$ in the representation formula for the solution of $\Sigma_\delta(\Lambda, A)$, we set
\begin{gather*}
Z(\Lambda) = \{\mathbf n \in \mathbb Z^N \;|\: \Lambda \cdot \mathbf n = 0\}, \\
\begin{aligned}
V(\Lambda) & = \{L \in \mathbb R^N \;|\: Z(\Lambda)\subset Z(L)\}, & \quad V_+(\Lambda) & = V(\Lambda) \cap (\mathbb R_+^\ast)^N, \\
W(\Lambda) & = \{L \in \mathbb R^N \;|\: Z(\Lambda)=Z(L)\}, & W_+(\Lambda) & = W(\Lambda) \cap (\mathbb R_+^\ast)^N.
\end{aligned}
\end{gather*}
Notice that $W(\Lambda) \subset V(\Lambda)$ and $W(\Lambda) = \{L \in V(\Lambda) \;|\: V(L) = V(\Lambda)\}$.

The point of view of this paper is to prescribe $\Lambda = (\Lambda_1, \dotsc, \Lambda_N) \in (\mathbb R_+^\ast)^N$ and to describe the rational dependence structure of its components through the sets $Z(\Lambda)$, $V(\Lambda)$, and $W(\Lambda)$. Another possible viewpoint, which is the one used for instance in \cite{Michiels2009Strong}, is to fix $B \in \mathcal M_{N, h}(\mathbb N)$ and consider the delays $\Lambda = (\Lambda_1, \dotsc, \Lambda_N) \in \range B \cap (\mathbb R_+^\ast)^N$. We show in the next proposition that the two points of view are equivalent.

\begin{proposition}
\label{PropRatDep}
Let $\Lambda = (\Lambda_1, \dotsc, \Lambda_N) \in (\mathbb R_+^\ast)^N$. There exist $h \in \llbracket 1, N\rrbracket$, $\ell = (\ell_1, \dotsc, \ell_h) \in (\mathbb R_+^\ast)^h$ with rationally independent components, and $B \in \mathcal M_{N, h}(\mathbb N)$ with $\rank(B) = h$ such that $\Lambda = B \ell$. Moreover, for every $B$ as before, one has
\begin{equation}
\label{VRanB}
\begin{aligned}
V(\Lambda) & = \range B, \\
W(\Lambda) & = \{B(\ell_1^\prime, \dotsc, \ell_h^\prime) \;|\: \ell_1^\prime, \dotsc, \ell_h^\prime \text{ are rationally independent} \}.
\end{aligned}
\end{equation}
In particular, $W(\Lambda)$ is dense and of full measure in $V(\Lambda)$.
\end{proposition}

\begin{proof}
Let $\mathsf V = \Span_{\mathbb Q} \{\Lambda_1, \dotsc, \Lambda_N\}$, $h = \dim_{\mathbb Q} \mathsf V$, and $\{\lambda_1, \dotsc, \lambda_h\}$ be a basis of $\mathsf V$ with positive elements, so that $\Lambda = A u$ for some $A = (a_{ij}) \in \mathcal M_{N, h}(\mathbb Q)$ and $u = (\lambda_1, \dotsc, \lambda_h) \in (\mathbb R_+^\ast)^h$. For $v \in \mathbb R^h \setminus \{0\}$, we denote by $P_v$ the orthogonal projection in the direction of $v$, i.e., $P_v = {v v^\mathrm{T}}/{\abs{v}_2^2}$.

Since $\mathbb Q^h$ is dense in $\mathbb R^h$, there exists a sequence of vectors $u_n = (r_{1, n}, \dotsc, r_{h, n})$ in $(\mathbb Q_+^\ast)^h$ converging to $u$ as $n \to +\infty$, and we can further assume that the sequence is chosen in such a way that $\abs{P_{u_n} - P_u}_2 \leq 1/n^2$ for every $n \in \mathbb N^\ast$.

For $n \in \mathbb N^\ast$, we define $T_n = P_{u_n} + \frac{1}{n} \left(\id_h - P_{u_n}\right)$. This operator is invertible, with inverse $T_n^{-1} = P_{u_n} + n \left(\id_h - P_{u_n}\right)$. Furthermore, both $T_n$ and $T_n^{-1}$ belong to $\mathcal M_h(\mathbb Q)$. For $i \in \llbracket 1, h\rrbracket$, we have
\[
(T_n^{-1} e_i)^\mathrm{T} u = e_i^{\mathrm{T}} P_{u_n} u + n e_i^{\mathrm{T}} (\id_h - P_{u_n}) u = e_i^{\mathrm{T}} P_{u_n} u + n e_i^{\mathrm{T}} (P_u - P_{u_n}) u
\]
and thus $(T_n^{-1} e_i)^{\mathrm{T}} u \to e_i^{\mathrm{T}} u = \lambda_i$ as $n \to +\infty$. Since $\lambda_1, \dotsc, \lambda_h > 0$, there exists $n_0 \in \mathbb N^\ast$ such that
\begin{equation}
\label{TnEiPositif}
(T_n^{-1} e_i)^{\mathrm{T}} u > 0, \qquad \forall i \in \llbracket 1, h\rrbracket,\; \forall n \geq n_0.
\end{equation}

For $i \in \llbracket 1, N\rrbracket$, let $\alpha_i = (a_{ij})_{j \in \llbracket 1, h\rrbracket} \in \mathbb Q^h$. For each $i \in \llbracket 1, N\rrbracket$, we construct the sequence $\left(\alpha_{i, n}\right)_{n \in \mathbb N^\ast}$ in $\mathbb Q^h$ by setting $\alpha_{i, n} = T_n \alpha_i$. It follows from the definition of $T_n$ that $\alpha_{i, n}$ converges to $P_u \alpha_i = \frac{u u^{\mathrm{T}} \alpha_i}{\abs{u}_2^2}$ as $n \to +\infty$. Since $u^{\mathrm{T}} \alpha_i = \sum_{j=1}^h a_{ij} \lambda_j = \Lambda_i > 0$ and the components of $u$ are positive, we conclude that there exists $n_1 \geq n_0$ such that $\alpha_{i, n_1} \in (\mathbb Q_+)^h$ for every $i \in \llbracket 1, N\rrbracket$.

Let $\ell = (T_{n_1}^{-1})^{\mathrm{T}} u$. By \eqref{TnEiPositif}, $\ell_i = (T_{n_1}^{-1} e_i)^{\mathrm{T}} u > 0$ for every $i \in \llbracket 1, h\rrbracket$. Since the components of $u$ are rationally independent, $\ell_1, \dotsc, \ell_h$ are also rationally independent. Let $b_{ij} \in \mathbb Q_+$, $i \in \llbracket 1, N\rrbracket$, $j \in \llbracket 1, h\rrbracket$, be such that $\alpha_{i, n_1} = (b_{ij})_{j \in \llbracket 1, h\rrbracket}$. Hence, for $i \in \llbracket 1, N\rrbracket$,
\[\Lambda_i = u^{\mathrm{T}} \alpha_i = u^{\mathrm{T}} T_{n_1}^{-1} \alpha_{i, n_1} = \sum_{j=1}^h b_{ij} u^{\mathrm{T}} T_{n_1}^{-1} e_j = \sum_{j=1}^h b_{ij} \ell_j.\]
We then get the required result up to multiplying $B = (b_{ij})$ by a large integer and modifying $\ell$ in accordance.

We next prove that \eqref{VRanB} holds for every $B$ as before. (Our proof actually holds for every $B \in \mathcal M_{N, h}(\mathbb Q)$ with $\rank(B) = h$ such that $\Lambda = B \ell$ for some $\ell \in (\mathbb R_+^\ast)^h$ with rationally independent components.) First notice that $Z(\Lambda) = \{\mathbf n \in \mathbb Z^N \;|\: \mathbf n \in \Ker B^{\mathrm{T}}\}$. Indeed, $\mathbf n \in Z(\Lambda)$ if and only if $\mathbf n$ is perpendicular in $\mathbb R^N$ to $B \ell$, which is equivalent to $\mathbf n^{\mathrm{T}} B = 0$ since $\ell_1, \dotsc, \ell_h$ are rationally independent. Moreover, remark that $\Ker B^{\mathrm{T}} = (\range B)^\perp$ admits a basis with integer coefficients since $\range B$ admits such a basis. To see that, it is enough to complete any basis of $\range B$ in $\mathbb Q^N$ by $N - h$ vectors in $\mathbb Q^N$ and find a basis of $(\range B)^\perp$ by Gram--Schmidt orthogonalization. We finally deduce that $\Span_{\mathbb R} (Z(\Lambda)) = (\range B)^\perp$. Since by definition $V(\Lambda) = Z(\Lambda)^\perp$, we conclude that $V(\Lambda) = \range B$. As regards the characterization of $W(\Lambda)$, an argument goes as follows. Let $L \in V(\Lambda)$, so that $L = B \ell^\prime$ for a certain $\ell^\prime \in \mathbb R^h$. The components of $\ell^\prime$ are rationally dependent if and only if $\dim_{\mathbb Q}\Span_{\mathbb Q}\{L_1, \dotsc, L_N\} < h$, i.e., $\dim_{\mathbb R} V(L) < \dim_{\mathbb R} V(\Lambda)$, which holds if and only if $L \notin W(\Lambda)$.
\end{proof}

We introduce the following additional definitions.

\begin{definition}
\label{DefiEquivalence}
Let $\Lambda = (\Lambda_1, \dotsc, \Lambda_N) \in (\mathbb R_+^\ast)^N$. We partition $\llbracket 1, N\rrbracket$ and $\mathbb Z^N$ according to the equivalence relations $\sim$ and $\approx$ defined as follows: $i \sim j$ if $\Lambda_i = \Lambda_j$ and $\mathbf n \approx \mathbf n^\prime$ if $\Lambda \cdot \mathbf n = \Lambda \cdot \mathbf n^\prime$. We use $[\cdot]$ to denote equivalence classes of both $\sim$ and $\approx$ and we set $\mathcal J = \llbracket 1, N\rrbracket / \sim$ and $\mathcal Z = \mathbb Z^N / \approx$.

For $A: \mathbb R \to \mathcal M_d(\mathbb C)^N$, $L \in V_+(\Lambda)$, $[\mathbf n] \in \mathcal Z$, $[i] \in \mathcal J$, and $t \in \mathbb R$, we define
\begin{equation}
\label{EqDefiXiHat}
\widehat\Xi_{[\mathbf n], t}^{L, \Lambda, A} = \sum_{\mathbf n^\prime \in [\mathbf n]} \Xi_{\mathbf n^\prime, t}^{L, A}, \qquad
\widehat A_{[i]}^\Lambda(t) = \sum_{j \in [i]} A_j(t),
\end{equation}
and
\begin{equation}
\label{EqDefiTheta}
\Theta_{[\mathbf n], t}^{L, \Lambda, A} = \sum_{\substack{[j] \in \mathcal J \\ L \cdot \mathbf n - L_j \leq t}} \widehat\Xi_{[\mathbf n - e_j], t}^{L, \Lambda, A} \widehat A_{[j]}^\Lambda(t - L \cdot \mathbf n + L_j).
\end{equation}
\end{definition}

\begin{remark}
\label{RemkLambdaPrime}
The expression for $\widehat\Xi_{[\mathbf n], t}^{L, \Lambda, A}$ given in \eqref{EqDefiXiHat} is well-defined since, thanks to \eqref{DefiXi}, all terms in the sum are equal to zero except finitely many. The expression for $\Theta_{[\mathbf n], t}^{L, \Lambda, A}$ given in \eqref{EqDefiTheta} is also well-defined since, for every $L \in V_+(\Lambda)$, if $i \sim j$ and $\mathbf n \approx \mathbf n^\prime$, one has $L_i = L_j$ and $L \cdot \mathbf n = L \cdot \mathbf n^\prime$. In addition, notice that $\widehat\Xi_{[\mathbf n], t}^{L, \Lambda, A} \not = 0$ only if $[\mathbf n] \cap \mathbb N^N$ is nonempty, and, similarly, $\Theta_{[\mathbf n], t}^{L, \Lambda, A} \not = 0$ only if $[\mathbf n] \cap (\mathbb N^N \setminus \{0\})$ is nonempty. Another consequence of the above fact and \eqref{EqDefiTheta} is that $\Theta_{[\mathbf n], t}^{L, \Lambda, A} \not = 0$ only if $t \geq 0$, since $[\mathbf n - e_j] \cap \mathbb N^N = \emptyset$ whenever $[\mathbf n] \in \mathcal Z$ and $[j] \in \mathcal J$ are such that $L \cdot \mathbf n - L_j < 0$.

Notice, moreover, that the matrices $\widehat\Xi$, $\widehat A$ and $\Theta$ depend on $\Lambda$ only through $Z(\Lambda)$. Hence, if $\Lambda^\prime \in W_+(\Lambda)$ (i.e., $Z(\Lambda^\prime) = Z(\Lambda)$), then
\[\widehat\Xi_{[\mathbf n], t}^{L, \Lambda, A} = \widehat\Xi_{[\mathbf n], t}^{L, \Lambda^\prime, A}, \qquad \widehat A_{[i]}^\Lambda(t) = \widehat A_{[i]}^{\Lambda^\prime}(t), \qquad \Theta_{[\mathbf n], t}^{L, \Lambda, A} = \Theta_{[\mathbf n], t}^{L, \Lambda^\prime, A}.\]
\end{remark}

From now on, we fix $\Lambda = (\Lambda_1, \dotsc, \Lambda_N) \in (\mathbb R_+^\ast)^N$ and our goal consists of deriving a suitable representation for the solutions of $\Sigma_{\delta}(L, A)$ for every $L \in V_+(\Lambda)$. Even though the above definitions depend on $\Lambda$, $L \in V_+(\Lambda)$ and $A$, we will sometimes omit (part of) this dependence from the notations when there is no risk of confusion. 

Let us now provide further expressions for $\widehat\Xi_{[\mathbf n], t}^{L, \Lambda, A}$.

\begin{proposition}
\label{PropXiHatRecurrence}
For every $L \in V_+(\Lambda)$, $A: \mathbb R \to \mathcal M_d(\mathbb C)^N$, $\mathbf n \in \mathbb N^N \setminus \{0\}$, and $t \in \mathbb R$, we have
\[
\widehat\Xi_{[\mathbf n], t}^{L, \Lambda, A} = \sum_{[j] \in \mathcal J} \widehat A_{[j]}^\Lambda(t) \widehat\Xi_{[\mathbf n - e_j], t - L_j}^{L, \Lambda, A}, \qquad \widehat\Xi_{[\mathbf n], t}^{L, \Lambda, A} = \sum_{[j] \in \mathcal J} \widehat\Xi_{[\mathbf n - e_j], t}^{L, \Lambda, A} \widehat A_{[j]}^\Lambda(t - L \cdot \mathbf n + L_j),
\]
and
\begin{equation}
\label{EqExplicitXiHat}
\widehat\Xi_{[\mathbf n], t}^{L, \Lambda, A} = \sum_{\mathbf n^\prime \in [\mathbf n] \cap \mathbb N^N} \sum_{v \in V_{\mathbf n^\prime}} \prod_{k=1}^{\abs{\mathbf n^\prime}_{1}} A_{v_k}\left(t - L \cdot \mathbf p_v(k)\right).
\end{equation}
\end{proposition}

\begin{proof}
We have, by Definition \ref{DefiEquivalence} and Equation \eqref{DefiXi}, that
\[
\begin{split}
\widehat\Xi_{[\mathbf n], t} & = \sum_{\mathbf n^\prime \in [\mathbf n]} \Xi_{\mathbf n^\prime, t} = \sum_{\mathbf n^\prime \in [\mathbf n]} \sum_{j=1}^N A_j(t) \Xi_{\mathbf n^\prime - e_j, t - L_j} = \sum_{j=1}^N A_j(t) \sum_{\mathbf n^\prime \in [\mathbf n]} \Xi_{\mathbf n^\prime - e_j, t - L_j} \\
 & = \sum_{j=1}^N A_j(t) \sum_{\mathbf m \in [\mathbf n - e_j]} \Xi_{\mathbf m, t - L_j} = \sum_{j=1}^N A_j(t) \widehat\Xi_{[\mathbf n - e_j], t - L_j} \\
 & = \sum_{[j] \in \mathcal J} \sum_{i \in [j]} A_i(t) \widehat\Xi_{[\mathbf n - e_i], t - L_i} = \sum_{[j] \in \mathcal J} \left(\sum_{i \in [j]} A_i(t)\right) \widehat\Xi_{[\mathbf n - e_j], t - L_j} \\
 & = \sum_{[j] \in \mathcal J} \widehat A_{[i]}(t) \widehat\Xi_{[\mathbf n - e_j], t - L_j}.
\end{split}
\]
The second expression is obtained similarly from Definition \ref{DefiEquivalence} and Equation \eqref{XiRecurrence} and the last one follows immediately from \eqref{XiExplicit} and \eqref{EqDefiXiHat}.
\end{proof}

Let us give a first representation for solutions of $\Sigma_\delta(L, A)$.

\begin{lemma}
Let $L \in (\mathbb R_+^\ast)^N$, $A = (A_1, \dotsc, A_N): \mathbb R \to \mathcal M_d(\mathbb C)^N$, and $u_0: \left[-L_{\max}, 0\right) \to \mathbb C^d$. The corresponding solution $u: \left[-L_{\max}, +\infty\right) \to \mathbb C^d$ of $\Sigma_{\delta}(L, A)$ is given for $t \geq 0$ by
\begin{equation}
\label{ExplicitU}
u(t) = \sum_{\substack{(\mathbf n, j) \in \mathbb N^N \times \llbracket 1, N\rrbracket \\ -L_j \leq t - L \cdot \mathbf n < 0}} \Xi_{\mathbf n - e_j, t}^{L, A} A_j(t - L \cdot \mathbf n + L_j) u_0(t - L \cdot \mathbf n).
\end{equation}
\end{lemma}

\begin{proof}
Let $u: \left[-L_{\max}, +\infty\right) \to \mathbb C^d$ be given for $t \geq 0$ by \eqref{ExplicitU} and $u(t) = u_0(t)$ for $t \in \left[-L_{\max}, 0\right)$. Fix $t \geq 0$ and notice that
\begin{align}
 & \sum_{j=1}^N A_j(t) u(t - L_j) \notag \\
 = & \sum_{\substack{j = 1 \\ t \geq L_j}}^N\; \sum_{\substack{(\mathbf n, k) \in \mathbb N^N \times \llbracket 1, N\rrbracket \\ -L_k \leq t - L_j - L \cdot \mathbf n < 0 \\ n_k \geq 1}} A_j(t) \Xi_{\mathbf n - e_k, t - L_j}^{L, A} A_k(t - L_j - L \cdot \mathbf n + L_k) u_0(t - L_j - L \cdot \mathbf n) \notag \\
 & {} + \sum_{\substack{j = 1 \\ t < L_j}}^N A_j(t) u_0(t - L_j). \label{ExplicitExpansion1}
\end{align}
Consider the sets
\[B_1(t) = \{(\mathbf n, k, j) \in \mathbb N^N \times \llbracket 1, N\rrbracket^2 \;|\: t \geq L_j,\; -L_k \leq t - L_j - L \cdot \mathbf n < 0,\; n_k \geq 1\},\]
\[B_2(t) = \{j \in \llbracket 1, N\rrbracket \;|\: t < L_j\},\]
\[C_1(t) = \{(\mathbf n, k, j) \in \mathbb N^N \times \llbracket 1, N\rrbracket^2 \;|\: -L_k \leq t - L \cdot \mathbf n < 0,\; n_k \geq 1,\; n_j \geq 1 + \delta_{jk},\; \mathbf n \not = e_k\},\]
\[C_2(t) = \{(\mathbf n, k) \in \mathbb N^N \times \llbracket 1, N\rrbracket \;|\: -L_k \leq t - L \cdot \mathbf n < 0,\; \mathbf n = e_k\},\]
and the functions $\varphi_i: B_i(t) \to C_i(t)$, $i \in \{1, 2\}$, given by
\[\varphi_1(\mathbf n, k, j) = (\mathbf n + e_j, k, j), \qquad \varphi_2(j) = (e_j, j).\]
One can check that $\varphi_1$ and $\varphi_2$ are well-defined and injective. We claim that they are also bijective. For the surjectivity of $\varphi_1$, we take $(\mathbf n, k, j) \in C_1(t)$ and set $\mathbf m = \mathbf n - e_j$. Since $n_j \geq 1$, one has $\mathbf m \in \mathbb N^N$. Since $n_k \geq 1$, $n_j \geq 1 + \delta_{jk}$, one has $t \geq L \cdot \mathbf n - L_k \geq L_j + L_k - L_k = L_j$. The inequalities $-L_k \leq t - L_j - L \cdot \mathbf m < 0$ and $n_k \geq 1$ are trivially satisfied, and thus $(\mathbf m, k, j) \in B_1(t)$, which shows the surjectivity of $\varphi_1$ since one clearly has $\varphi_1(\mathbf m, k, j) = (\mathbf n, k, j)$. For the surjectivity of $\varphi_2$, we take $(\mathbf n, k) \in C_2(t)$, which then satisfies $\mathbf n = e_k$ and $t < L \cdot \mathbf n = L_k$. This shows that $k \in B_2(t)$ and, since $\varphi_2(k) = (\mathbf n, k)$, we obtain that $\varphi_2$ is surjective.

Thanks to the bijections $\varphi_1$, $\varphi_2$, and \eqref{DefiXi}, \eqref{ExplicitExpansion1} becomes
\begin{align*}
& \sum_{j=1}^N A_j(t) u(t - L_j) \displaybreak[0] \\
 = & \sum_{\substack{(\mathbf n, k) \in \mathbb N^N \times \llbracket 1, N\rrbracket \\ -L_k \leq t - L \cdot \mathbf n < 0 \\ n_k \geq 1,\; \mathbf n \not = e_k}} \sum_{\substack{j = 1 \\ n_j \geq 1 + \delta_{jk}}}^N A_j(t) \Xi_{\mathbf n - e_k - e_j, t - L_j}^{L, A} A_k(t - L \cdot \mathbf n + L_k) u_0(t - L \cdot \mathbf n) \\
& {} + \sum_{\substack{(\mathbf n, k) \in \mathbb N^N \times \llbracket 1, N\rrbracket\\ -L_k \leq t - L \cdot \mathbf n < 0,\\ \mathbf n = e_k}} A_k(t - L \cdot \mathbf n + L_k) u_0(t - L \cdot \mathbf n) \displaybreak[0] \\
 = & \sum_{\substack{(\mathbf n, k) \in \mathbb N^N \times \llbracket 1, N\rrbracket \\ -L_k \leq t - L \cdot \mathbf n < 0 \\ n_k \geq 1,\; \mathbf n \not = e_k}} \Xi_{\mathbf n - e_k, t}^{L, A} A_k(t - L \cdot \mathbf n + L_k) u_0(t - L \cdot \mathbf n) \\
& {} + \sum_{\substack{(\mathbf n, k) \in \mathbb N^N \times \llbracket 1, N\rrbracket\\ -L_k \leq t - L \cdot \mathbf n < 0,\\ \mathbf n = e_k}} \Xi_{0, t}^{L, A} A_k(t - L \cdot \mathbf n + L_k) u_0(t - L \cdot \mathbf n) \displaybreak[0] \\
 = & \sum_{\substack{(\mathbf n, k) \in \mathbb N^N \times \llbracket 1, N\rrbracket \\ -L_k \leq t - L \cdot \mathbf n < 0}} \Xi_{\mathbf n - e_k, t}^{L, A} A_k(t - L \cdot \mathbf n + L_k) u_0(t - L \cdot \mathbf n) = u(t),
\end{align*}
which shows that $u$ satisfies \eqref{SystDelay}.
\end{proof}

We can now give the main result of this section.

\begin{proposition}
\label{PropSolExpliciteAdpBis}
Let $\Lambda \in (\mathbb R_+^\ast)^N$, $L \in V_+(\Lambda)$, $A: \mathbb R \to \mathcal M_d(\mathbb C)^N$, and $u_0: \left[-L_{\max}, 0\right) \to \mathbb C^d$. The corresponding solution $u: \left[-L_{\max}, +\infty\right) \to \mathbb C^d$ of $\Sigma_{\delta}(L, A)$ is given for $t \geq 0$ by
\begin{equation}
\label{ExplicitUAdpBis}
u(t) = \sum_{\substack{[\mathbf n] \in \mathcal Z \\ t < L \cdot \mathbf n \leq t + L_{\max}}} \Theta_{[\mathbf n], t}^{L, \Lambda, A} u_0(t - L \cdot \mathbf n),
\end{equation}
where the coefficients $\Theta$ are defined in \eqref{EqDefiTheta}.
\end{proposition}

\begin{proof}
Equation \eqref{ExplicitUAdpBis} follows immediately from \eqref{ExplicitU} and from the fact that the function $\varphi: \mathbb N^N \times \llbracket 1, N\rrbracket \to \mathcal Z \times \mathbb N^N \times \mathcal J \times \llbracket 1, N\rrbracket$ given by $\varphi(\mathbf n, j) = ([\mathbf n], \mathbf n, [j], j)$ is a bijective map from $\{(\mathbf n, j) \in \mathbb N^N \times \llbracket 1, N\rrbracket \;|\: -L_j \leq t - L \cdot \mathbf n < 0\}$ to $\{([\mathbf m], \mathbf n, [i], j) \in \mathcal Z \times \mathbb N^N \times \mathcal J \times \llbracket 1, N\rrbracket \;|\: \mathbf n \in [\mathbf m],\; j \in [i],\; t < L \cdot \mathbf n \leq t + L_{\max},\; L \cdot \mathbf n - L_j \leq t\}$ for every $t \in \mathbb R$.
\end{proof}

\subsection{Asymptotic behavior of solutions in terms of the coefficients}
\label{SecStability}

Let us fix a matrix norm $\abs{\cdot}$ on $\mathcal M_d(\mathbb C)$ in the whole section. Let $C_1, C_2 > 0$ be such that
\begin{equation}
\label{EquivNormAst}
C_1 \abs{A}_{p} \leq \abs{A} \leq C_2 \abs{A}_{p}, \qquad \forall A \in \mathcal M_{d}(\mathbb C),\; \forall p \in [1, +\infty].
\end{equation}

Let $\mathcal A$ be a uniformly locally bounded subset of $L^\infty_{\mathrm{loc}}(\mathbb R, \mathcal M_d(\mathbb C)^N)$. The family of all systems $\Sigma_\delta(L, A)$ for $A \in \mathcal A$ is denoted by $\Sigma_\delta(L, \mathcal A)$. We wish to characterize the asymptotic behavior of $\Sigma_\delta(L, \mathcal A)$ (i.e., uniformly with respect to $A \in \mathcal A$) in terms of the behavior of the coefficients $\widehat\Xi_{[\mathbf n], t}$ and $\Theta_{[\mathbf n], t}$. For that purpose, we introduce the following definitions.

\begin{definition}
Let $L \in (\mathbb R_+^\ast)^N$.

\begin{enumerate}[label={\bf \roman*.}, ref={(\roman*)}]
\item For $p \in [1, +\infty]$, we say that $\Sigma_\delta(L, \mathcal A)$ is of \emph{exponential type} $\gamma \in \mathbb R$ in $\mathsf X_p^\delta$ if, for every $\varepsilon > 0$, there exists $K > 0$ such that, for every $A \in \mathcal A$ and $u_0 \in \mathsf X_p^\delta$, the corresponding solution $u$ of $\Sigma_\delta(L, A)$ satisfies, for every $t \geq 0$,
\[
\norm{u_t}_{p} \leq K e^{(\gamma + \varepsilon) t} \norm{u_0}_{p}.
\]
We say that $\Sigma_\delta(L, \mathcal A)$ is \emph{exponentially stable} in $\mathsf X_p^\delta$ if it is of negative exponential type.

\item Let $\Lambda \in (\mathbb R_+^\ast)^N$ be such that $L \in V_+(\Lambda)$. We say that $\Sigma_\delta(L, \mathcal A)$ is of \emph{$(\Theta, \Lambda)$-exponential type} $\gamma \in \mathbb R$ if, for every $\varepsilon > 0$, there exists $K > 0$ such that, for every $A \in \mathcal A$, $\mathbf n \in \mathbb N^N$, and almost every $t \in \left(L \cdot \mathbf n - L_{\max}, L \cdot \mathbf n\right)$, we have
\[
\abs{\Theta_{[\mathbf n], t}^{L, \Lambda, A}} \leq K e^{(\gamma + \varepsilon) t}.
\]

\item Let $\Lambda \in (\mathbb R_+^\ast)^N$ be such that $L \in V_+(\Lambda)$. We say that $\Sigma_\delta(L, \mathcal A)$ is of \emph{$(\widehat\Xi, \Lambda)$-exponential type} $\gamma \in \mathbb R$ if, for every $\varepsilon > 0$, there exists $K > 0$ such that, for every $A \in \mathcal A$, $\mathbf n \in \mathbb N^N$, and almost every $t \in \mathbb R$, we have
\[
\abs{\widehat\Xi_{[\mathbf n], t}^{L, \Lambda, A}} \leq K e^{(\gamma + \varepsilon) L \cdot \mathbf n}.
\]

\item For $p \in [1, +\infty]$, the \emph{maximal Lyapunov exponent of $\Sigma_\delta(L, \mathcal A)$} in $\mathsf X_p^\delta$ is defined as
\[
\lambda_p(L, \mathcal A) = \limsup_{t \to +\infty} \sup_{A \in \mathcal A} \sup_{\substack{u_0 \in \mathsf X_p^\delta\\ \norm{u_0}_{p} = 1}} \frac{\ln \norm{u_t}_{p}}{t},
\]
where $u$ denotes the solution of $\Sigma_\delta(L, A)$ with initial condition $u_0$.
\end{enumerate}
\end{definition}

\begin{remark}
\label{RemkTransformation}
Let $L \in (\mathbb R_+^\ast)^N$ and $\mu \in \mathbb R$. For every $A: \mathbb R \to \mathcal M_d(\mathbb C)^N$ and $u$ solution of $\Sigma_\delta(L, A)$, it follows from \eqref{SystDelay} that $t \mapsto e^{\mu t} u(t)$ is a solution of the system $\Sigma_\delta(L, (e^{\mu L_1} A_1, \dotsc, e^{\mu L_N} A_N))$. As a consequence, if $\mathcal A \subset L^\infty_{\mathrm{loc}}(\mathbb R, \mathcal M_d(\mathbb C)^N)$ and
\[\mathcal A_\mu = \{(e^{\mu L_1} A_1, \dotsc, e^{\mu L_N} A_N) \;|\: A = (A_1, \dotsc, A_N) \in \mathcal A\},\]
one has $\lambda_p(L, \mathcal A_\mu) = \lambda_p(L, \mathcal A) + \mu$.
\end{remark}

The link between exponential type and maximal Lyapunov exponent of $\Sigma_\delta(L, \mathcal A)$ is provided by the following proposition.

\begin{proposition}
\label{PropLyapExpo}
Let $L \in (\mathbb R_+^\ast)^N$, $\mathcal A$ be uniformly locally bounded, and $p \in [1, +\infty]$. Then
\[\lambda_p(L, \mathcal A) = \inf\{\gamma \in \mathbb R \;|\: \Sigma_\delta(L, \mathcal A) \text{ is of exponential type }\gamma\text{ in }\mathsf X_p^\delta\}.\]
In particular, $\Sigma_\delta(L, \mathcal A)$ is exponentially stable if and only if $\lambda_p(L, \mathcal A) < 0$.
\end{proposition}

\begin{proof}
Let $\gamma \in \mathbb R$ be such that $\Sigma_\delta(L, \mathcal A)$ is of exponential type $\gamma$ in $\mathsf X_p^\delta$. It is clear from the definition that $\lambda_p(L, \mathcal A) \leq \gamma$. We are left to prove that $\Sigma_\delta(L, \mathcal A)$ is of exponential type $\lambda_p(L, \mathcal A)$ when the latter is finite. Let $\varepsilon > 0$. From the definition of $\lambda_p(L, \mathcal A)$, there exists $t_0 > 0$ such that, for every $t \geq t_0$, $A \in \mathcal A$, and $u_0 \in \mathsf X_p^\delta$, one has
\[
\norm{u_t}_{p} \leq e^{(\lambda_p(L, \mathcal A) + \varepsilon) t} \norm{u_0}_{p}.
\]
Since $\mathcal A$ is uniformly locally bounded, by using \eqref{ExplicitU} and \eqref{XiExplicit}, one deduces that there exists $K > 0$ such that, for every $t \in [0, t_0]$, $A \in \mathcal A$, and $u_0 \in \mathsf X_p^\delta$, one has $\norm{u_t}_{p} \leq K \norm{u_0}_{p}$. Hence the conclusion.
\end{proof}

\begin{remark}
\label{RemkLimsupThetaXi}
Similarly, one proves that, for $\Lambda \in (\mathbb R_+^\ast)^N$ and $L \in V_+(\Lambda)$,
\begin{multline*}
\limsup_{L \cdot \mathbf n \to +\infty} \sup_{A \in \mathcal A} \esssup_{t \in \left(L \cdot \mathbf n - L_{\max}, L \cdot \mathbf n\right)} \frac{\ln \abs{\Theta_{[\mathbf n], t}^{L, \Lambda, A}}}{t} \\
= \inf\{\gamma \in \mathbb R \;|\: \Sigma_\delta(L, \mathcal A) \text{ is of $(\Theta, \Lambda)$-exponential type }\gamma\}
\end{multline*}
and
\begin{multline*}
\limsup_{L \cdot \mathbf n \to +\infty} \sup_{A \in \mathcal A} \esssup_{t \in \mathbb R} \frac{\ln \abs{\widehat\Xi_{[\mathbf n], t}^{L, \Lambda, A}}}{L \cdot \mathbf n} \\
 = \inf\{\gamma \in \mathbb R \;|\: \Sigma_\delta(L, \mathcal A) \text{ is of $(\widehat\Xi, \Lambda)$-exponential type }\gamma\}.
\end{multline*}
\end{remark}

\subsubsection{General case}

The following result, which is a generalization of \cite[Proposition 4.1]{Chitour2015Persistently}, uses the representation formula \eqref{ExplicitUAdpBis} for the solutions of $\Sigma_{\delta}(L, A)$ in order to provide upper bounds on their growth.

\begin{proposition}
\label{PropConvCoeffConvSolGeneral}
Let $L \in V_+(\Lambda)$. Suppose that there exists a continuous function $f: \mathbb R \to \mathbb R_+^\ast$ such that, for every $A \in \mathcal A$, $\mathbf n \in \mathbb N^N$, and almost every $t \in \left(L \cdot \mathbf n - L_{\max}, L \cdot \mathbf n\right)$, one has
\begin{equation}
\label{HypoConvCoeff}
\abs{\Theta_{[\mathbf n], t}^{L, \Lambda, A}} \leq f(t).
\end{equation}
Then there exists $C > 0$ such that, for every $A \in \mathcal A$, $p \in \left[1, +\infty\right]$, and $u_0 \in \mathsf X_p^\delta$, the corresponding solution $u$ of $\Sigma_{\delta}(L, A)$ satisfies, for every $t \geq 0$,
\begin{equation}
\norm{u_t}_{p} \leq C (t + 1)^{N-1} \max_{s \in [t - L_{\max}, t]} f(s) \norm{u_0}_{p}.
\label{SolConvExpo}
\end{equation}
\end{proposition}

\begin{proof}
Let $A \in \mathcal A$, $p \in \left[1, +\infty\right)$, $u_0 \in \mathsf X_p^\delta$, and $u$ be the solution of $\Sigma_{\delta}(L, A)$ with initial condition $u_0$. For $t \in \mathbb R_+$, we write $\mathcal Y_{t} = \{[\mathbf n] \in \mathcal Z \;|\: t < L \cdot \mathbf n \leq t + L_{\max},\; [\mathbf n] \cap \mathbb N^N \not = \emptyset\}$ and $Y_t = \# \mathcal Y_t$. Thanks to Proposition \ref{PropSolExpliciteAdpBis}, Remark \ref{RemkLambdaPrime}, and \eqref{HypoConvCoeff}, we have, for $t \geq L_{\max}$,
\[
\begin{split}
\norm{u_t}_{p}^p & = \int_{t - L_{\max}}^t \abs{\sum_{[\mathbf n] \in \mathcal Y_s} \Theta_{[\mathbf n], s} u_0(s - L \cdot \mathbf n)}_{p}^p ds \\
 & \leq \int_{t - L_{\max}}^t Y_{s}^{p-1} \sum_{[\mathbf n] \in \mathcal Y_s} \abs{\Theta_{[\mathbf n], s} u_0(s - L \cdot \mathbf n)}_{p}^p ds \\
 & \leq C_1^{-p} \int_{t - L_{\max}}^t Y_{s}^{p-1} f(s)^p \sum_{[\mathbf n] \in \mathcal Y_s} \abs{u_0(s - L \cdot \mathbf n)}_{p}^p ds \\
 & \leq C_1^{-p} \max_{s \in [t - L_{\max}, t]} f(s)^p \int_{t - L_{\max}}^t Y_{s}^{p-1} \sum_{[\mathbf n] \in \mathcal Y_s} \abs{u_0(s - L \cdot \mathbf n)}_{p}^p ds.
\end{split}
\]

We clearly have $Y_t \leq \# \{\mathbf n \in \mathbb N^N \;|\: t < L \cdot \mathbf n \leq t + L_{\max}\}$. For $\mathbf n \in \mathbb N^N$, we denote $\mathfrak C_{\mathbf n} = \{x \in \mathbb R^N \;|\: n_i < x_i < n_i + 1 \text{ for every } i \in \llbracket 1, N\rrbracket\}$. This defines a family of pairwise disjoint open hypercubes of unit volume. Thus
\begin{equation*}
\begin{split}
Y_t & \leq \sum_{\substack{\mathbf n \in \mathbb N^N \\ t < L \cdot \mathbf n \leq t + L_{\max}}} \Vol \mathfrak C_{\mathbf n} = \Vol\left(\bigcup_{\substack{\mathbf n \in \mathbb N^N \\ t < L \cdot \mathbf n \leq t + L_{\max}}} \mathfrak C_{\mathbf n}\right) \\
 & \leq \Vol \{x \in (\mathbb R_+)^N \;|\: t < L \cdot x < t + \abs{L}_1 + L_{\max}\}.
\end{split}
\end{equation*}
Then there exists $C_3 > 0$ only depending on $L$ and $N$ such that $Y_t \leq C_3 (t + 1)^{N - 1}$. Thus,
\begin{align*}
\norm{u_t}_{p}^p & \leq C_1^{-p} C_3^{p-1} (t + 1)^{(N-1)(p-1)} \\
 & \hphantom{\leq} {} \cdot \max_{s \in [t - L_{\max}, t]} f(s)^p \int_{t - L_{\max}}^t \sum_{[\mathbf n] \in \mathcal Y_s} \abs{u_0(s - L \cdot \mathbf n)}_{p}^p ds \displaybreak[0] \\
 & = C_1^{-p} C_3^{p-1} (t + 1)^{(N-1)(p-1)} \\
 & \hphantom{=} {} \cdot \max_{s \in [t - L_{\max}, t]} f(s)^p \int_{-L_{\max}}^0 \sum_{[\mathbf n] \in \mathcal Y_{t - L_{\max} - s}} \abs{u_0(s)}_{p}^p ds.
\end{align*}
Similarly, there exists $C_4 > 0$ only depending on $L$ and $N$ such that, for every $t \in \mathbb R_+$ and $s \in [-L_{\max}, 0]$, $Y_{t - L_{\max} - s} \leq C_4 (t + 1)^{N-1}$, yielding \eqref{SolConvExpo} for $t \geq L_{\max}$. One can easily show that, for $0 \leq t \leq L_{\max}$, we have $\norm{u_t}_{p} \leq C^\prime \norm{u_0}_{p}$ for some constant $C^\prime$ independent of $p$ and $u_0$, and so \eqref{SolConvExpo} holds for every $t \geq 0$. The case $p = +\infty$ is treated by similar arguments.
\end{proof}

When $L \in W_+(\Lambda)$, we also have the following lower bound for solutions of $\Sigma_\delta(L, \mathcal A)$.

\begin{proposition}
\label{PropNConvCoeffNConvSolGeneral}
Let $L \in W_+(\Lambda)$ and $f: \mathbb R \to \mathbb R_+^\ast$ be a continuous function. Suppose that there exist $A \in \mathcal A$, $\mathbf n_0 \in \mathbb N^N$, and a set of positive measure $S \subset \left(L \cdot \mathbf n_0 - L_{\max}, L \cdot \mathbf n_0\right)$ such that, for every $s \in S$,
\begin{equation}
\label{HypoNConvCoeff}
\abs{\Theta_{[\mathbf n_0], s}^{L, \Lambda, A}} > f(s).
\end{equation}
Then there exist a constant $C > 0$ independent of $f$, an initial condition $u_0 \in L^\infty([-L_{\max}, 0], \mathbb C^d)$, and $t > 0$, such that, for every $p \in [1, +\infty]$, the solution $u$ of $\Sigma_{\delta}(L, A)$ with initial condition $u_0$ satisfies
\begin{equation*}
\norm{u_t}_{p} > C \min_{s \in [t - L_{\max}, t]} f(s) \norm{u_0}_{p}.
\end{equation*}
\end{proposition}

\begin{proof}
According to Remark \ref{RemkLambdaPrime}, one has $\Theta_{[\mathbf n], s}^{L, \Lambda, A} = \Theta_{[\mathbf n], s}^{L, L, A}$ for every $[\mathbf n] \in \mathcal Z$ and $s \in \mathbb R$, and therefore we assume for the rest of the argument that $\Lambda = L$ and we drop the upper index $L, L, A$.

For $s \in S$, one has $\abs{\Theta_{[\mathbf n_0], s}}_{\infty} > C_2^{-1} f(s)$, where $C_2$ is defined in \eqref{EquivNormAst}. Using \eqref{HypoNConvCoeff} and Remark \ref{RemkLambdaPrime}, one derives that $S \subset \left[0, +\infty\right)$.

For every $s \in S$, one has
\[C_2^{-1} f(s) < \abs{\Theta_{[\mathbf n_0], s}}_{\infty} \leq \sum_{j=1}^d \abs{\Theta_{[\mathbf n_0], s} e_j}_{\infty},\]
and thus there exist $j_0 \in \llbracket 1, d\rrbracket$ and a subset $\tilde S \subset S$ of positive measure such that, for every $s \in \tilde S$ and $p \in [1, +\infty]$, one has
\begin{equation}
\label{NConvCoeffLp}
C_2^{-1} d^{-1} f(s) < \abs{\Theta_{[\mathbf n_0], s} e_{j_0}}_{\infty} \leq \abs{\Theta_{[\mathbf n_0], s} e_{j_0}}_{p}.
\end{equation}
In order to simplify the notations in the sequel, we write $S$ instead of $\tilde S$.

Let $t_0 \in S \setminus \{0\}$ be such that, for every $\varepsilon > 0$, $(t_0 - \varepsilon, t_0 + \varepsilon) \cap S$ has positive measure. Let $\delta > 0$ be such that
\[
2 \delta < \min\left\{2 t_0, L \cdot \mathbf n_0 - t_0, t_0 - L \cdot \mathbf n_0 + L_{\max}, \min_{\substack{\mathbf n \in \mathbb N^N \\ L \cdot (\mathbf n - \mathbf n_0) \not = 0}} \abs{L \cdot (\mathbf n - \mathbf n_0)}\right\}.
\]
Such a choice is possible since $t_0 \in (L \cdot \mathbf n_0 - L_{\max}, L \cdot \mathbf n_0)$, $t_0 \in S \setminus \{0\} \subset \mathbb R_+^\ast$, and $\{L \cdot \mathbf n \;|\: \mathbf n \in \mathbb N^N\}$ is locally finite.

Let $S_1 = (S - t_0) \cap (-\delta, \delta)$, which is, by construction, of positive measure, and $\mu: \mathbb R \to \mathbb R$ be any non-zero bounded measurable function which is zero outside $S_1$. Define $u_0: \left[-L_{\max}, 0\right) \to \mathbb C^d$ by
\begin{equation*}
u_0(s) = \mu (s - t_0 + L \cdot \mathbf n_0) e_{j_0}
\end{equation*}
and let $u$ be the solution of $\Sigma_{\delta}(L, A)$ with initial condition $u_0$. For $s \in (-\delta, \delta)$, we have $t_0 + s > 0$ since $t_0 > \delta$. By Proposition \ref{PropSolExpliciteAdpBis}, one has
\begin{equation}
\label{uT0pS}
u(t_0 + s) = \sum_{\substack{[\mathbf n] \in \mathcal Z \\ t_0 + s < L \cdot \mathbf n \leq t_0 + s + L_{\max}}} \Theta_{[\mathbf n], t_0 + s} \mu (s + L \cdot (\mathbf n_0 - \mathbf n)) e_{j_0}.
\end{equation}
If $L \cdot \mathbf n \not = L \cdot \mathbf n_0$, we have $\abs{L \cdot (\mathbf n - \mathbf n_0)} > 2 \delta$, and so $\abs{s + L \cdot (\mathbf n_0 - \mathbf n)} > \delta$, which shows that $\mu(s + L \cdot (\mathbf n_0 - \mathbf n)) = 0$. Hence, Equation \eqref{uT0pS} reduces to $u(t_0 + s) = \Theta_{[\mathbf n_0], t_0 + s} \mu(s) e_{j_0}$. We finally obtain, using \eqref{NConvCoeffLp} and letting $t = t_0 + \delta$, that, for $p \in \left[1, +\infty\right)$,
\begin{equation}
\label{EstimNConv}
\begin{split}
\norm{u_{t}}_{p}^p & \geq \norm{u_{t_0}}_{L^p([-\delta, \delta], \mathbb C^d)}^p \geq \int_{S_1} \abs{u(t_0 + s)}_{p}^p ds = \int_{S_1} \abs{\Theta_{[\mathbf n_0], t_0 + s} e_{j_0}}_{p}^p \abs{\mu(s)}^p ds\\
 & > C_2^{-p} d^{-p} \int_{S_1} f(t_0 + s)^p \abs{\mu(s)}^p ds \geq C_2^{-p} d^{-p} \min_{s \in [t - L_{\max}, t]} f(s)^p \norm{u_0}_{p}^p. 
\end{split}
\end{equation}
A similar estimate holds in the case $p = +\infty$, which concludes the proof of the proposition.
\end{proof}

As a corollary of Propositions \ref{PropConvCoeffConvSolGeneral} and \ref{PropNConvCoeffNConvSolGeneral}, by taking $f$ of the type $f(t) = K e^{(\gamma + \varepsilon) t}$, one obtains the following theorem. The last equality follows from Proposition \ref{PropLyapExpo} and Remark \ref{RemkLimsupThetaXi}.

\begin{theorem}
\label{TheoP0EquivP1}
Let $\Lambda \in (\mathbb R_+^\ast)^N$ and $\mathcal A$ be uniformly locally bounded. For every $L \in V_+(\Lambda)$, if $\Sigma_\delta(L, \mathcal A)$ is of $(\Theta, \Lambda)$-exponential type $\gamma$ then, for every $p \in [1, +\infty]$, it is of exponential type $\gamma$ in $\mathsf X_p^\delta$. Conversely, for every $L \in W_+(\Lambda)$, if there exists $p \in [1, +\infty]$ such that $\Sigma_\delta(L, \mathcal A)$ is of exponential type $\gamma$ in $\mathsf X_p^\delta$, then it is of $(\Theta, \Lambda)$-exponential type $\gamma$. Finally, for every $L \in W_+(\Lambda)$ and $p \in [1, +\infty]$,
\begin{equation}
\label{LyapTheta}
\lambda_p(L, \mathcal A) = \limsup_{L \cdot \mathbf n \to +\infty} \sup_{A \in \mathcal A} \esssup_{t \in \left(L \cdot \mathbf n - L_{\max}, L \cdot \mathbf n\right)} \frac{\ln \abs{\Theta_{[\mathbf n], t}^{L, \Lambda, A}}}{t}.
\end{equation}
\end{theorem}

\begin{remark}
It also follows from Proposition \ref{PropConvCoeffConvSolGeneral} that, in the first part of the theorem, the constant $K > 0$ in the definition of exponential type of $\Sigma_\delta(L, \mathcal A)$ can be chosen independently of $p \in [1, +\infty]$. Moreover, the left-hand side of \eqref{LyapTheta} does not depend on $p$ and its right-hand side does not depend on $\Lambda$.
\end{remark}

\subsubsection{Shift-invariant classes}

We start this section by the following technical result.

\begin{lemma}
\label{LemmXiTimeTransl}
For every $\Lambda \in (\mathbb R_+^\ast)^N$, $L \in V_+(\Lambda)$, $A: \mathbb R \to \mathcal M_d(\mathbb C)^N$, $\mathbf n \in \mathbb Z^N$, and $t, \tau \in \mathbb R$, we have
\begin{equation*}
\Xi_{\mathbf n, t + \tau}^{L, A} = \Xi_{\mathbf n, t}^{L, A(\cdot + \tau)} \qquad \text{ and } \qquad \widehat\Xi_{[\mathbf n], t + \tau}^{L, \Lambda, A} = \widehat\Xi_{[\mathbf n], t}^{L, \Lambda, A(\cdot + \tau)}.
\end{equation*}
\end{lemma}

\begin{proof}
The first part holds trivially if $\mathbf n \in \mathbb Z^N \setminus \mathbb N^N$ or if $\mathbf n = 0$, for, in these cases, it follows from \eqref{DefiXi} that $\Xi_{\mathbf n, t}^{L, A}$ does not depend on $t$ and $A$. If $\mathbf n \in \mathbb N^N \setminus \{0\}$, the conclusion follows as a consequence of the explicit formula \eqref{XiExplicit} for $\Xi_{\mathbf n, t}^{L, A}$. The second part is a consequence of the first and \eqref{EqDefiXiHat}.
\end{proof}

We next provide a proposition establishing a relation between the behavior of $\widehat\Xi_{[\mathbf n], t}$ and $\Theta_{[\mathbf n], t}$. Notice that, if a subset $\mathcal A$ of $L^\infty_{\mathrm{loc}}(\mathbb R, \mathcal M_d(\mathbb C)^N)$ is shift-invariant, then $\mathcal A$ is uniformly locally bounded if and only if it is bounded.

\begin{proposition}
Let $\mathcal A$ be a bounded shift-invariant subset of $L^\infty(\mathbb R, \mathcal M_d(\mathbb C)^N)$, $L \in V_+(\Lambda)$, and $f: \mathbb R \to \mathbb R_+^\ast$ be a continuous function. Then the following assertions hold.

\begin{enumerate}[label={\bf \roman*.}, ref={(\roman*)}]
\item\label{ShiftImpl1} If $\abs{\Theta_{[\mathbf n], t}^{L, \Lambda, A}} \leq f(t)$ holds for every $A \in \mathcal A$, $\mathbf n \in \mathbb N^N$, and almost every $t \in \left(L \cdot \mathbf n - L_{\max}, L \cdot \mathbf n\right)$, then, for every $A \in \mathcal A$, $\mathbf n \in \mathbb N^N \setminus \{0\}$, and almost every $t \in \mathbb R$, one has $\abs{\widehat\Xi_{[\mathbf n], t}^{L, \Lambda, A}} \leq \max_{s \in [L \cdot \mathbf n - L_{\min}, L \cdot \mathbf n]} \allowbreak{}f(s)$.

\item\label{ShiftImpl2} If $\abs{\widehat\Xi_{[\mathbf n], t}^{L, \Lambda, A}} \leq f(L \cdot \mathbf n)$ holds for every $A \in \mathcal A$, $\mathbf n \in \mathbb N^N$, and almost every $t \in \mathbb R$, then there exists a constant $C > 0$ such that, for every $A \in \mathcal A$, $\mathbf n \in \mathbb N^N$, and almost every $t \in \left(L \cdot \mathbf n - L_{\max}, L \cdot \mathbf n\right)$, one has $\abs{\Theta_{[\mathbf n], t}^{L, \Lambda, A}} \leq C \max_{s \in [t - L_{\max}, t + L_{\max}]}f(s)$.
\end{enumerate}
\end{proposition}

\begin{proof}
We start by showing \ref{ShiftImpl1}. Let $A \in \mathcal A$ and $\mathbf n \in \mathbb N^N \setminus \{0\}$. For every $k \in \mathbb Z$, there exists a set $N_k \subset \left[L \cdot \mathbf n - L_{\min}, L \cdot \mathbf n\right)$ of measure zero such that, for every $t \in \left[L \cdot \mathbf n - L_{\min}, L \cdot \mathbf n\right) \setminus N_k$,
\[
\begin{split}
\abs{\widehat\Xi_{[\mathbf n], t}^{L, \Lambda, A(\cdot - k L_{\min})}} & = \abs{\sum_{[j] \in \mathcal J} \widehat\Xi_{[\mathbf n - e_j], t}^{L, \Lambda, A(\cdot - k L_{\min})} \widehat A_{[j]}^\Lambda(t - k L_{\min} - L \cdot \mathbf n + L_j)} \\
  & = \abs{\Theta_{[\mathbf n], t}^{L, \Lambda, A(\cdot - k L_{\min})}} \leq f(t),
\end{split}
\]
where we use Proposition \ref{PropXiHatRecurrence}, the fact that $L \cdot \mathbf n - L_j \leq L \cdot \mathbf n - L_{\min} \leq t$ for every $[j] \in \mathcal J$, and Equation \eqref{EqDefiTheta}.

Let $N = \bigcup_{k \in \mathbb Z} (N_k - k L_{\min})$, which is of measure zero. For $t \in \mathbb R \setminus N$, let $k \in \mathbb Z$ be such that $t \in \left[L \cdot \mathbf n - (k + 1) L_{\min}, L \cdot \mathbf n - k L_{\min}\right)$, so that $t + k L_{\min} \in \left[L \cdot \mathbf n - L_{\min}, L \cdot \mathbf n\right)$. Since $t \notin N$, we have $t + k L_{\min} \notin N_k$, and so, using Lemma \ref{LemmXiTimeTransl}, we obtain that
\[\abs{\widehat\Xi_{[\mathbf n], t}^{L, \Lambda, A}} = \abs{\widehat\Xi_{[\mathbf n], t + k L_{\min}}^{L, \Lambda, A(\cdot - k L_{\min})}} \leq f(t + k L_{\min}) \leq \max_{s \in [L \cdot \mathbf n - L_{\min}, L \cdot \mathbf n]} f(s).\]

Let us now show \ref{ShiftImpl2}. Without loss of generality, the norm $\abs{\cdot}$ is sub-multiplicative. Since $\mathcal A$ is bounded, there exists $M > 0$ such that, for every $A \in \mathcal A$, $j \in \llbracket 1, N\rrbracket$, and $t \in \mathbb R$, we have $\abs{A_j(t)} \leq M$. Let $A \in \mathcal A$. For every $\mathbf n \in \mathbb N^N$, let $N_{[\mathbf n]}$ be a set of measure zero such that $\abs{\widehat\Xi_{[\mathbf n], t}^{L, \Lambda, A}} \leq f(L \cdot \mathbf n)$ holds for every $t \in \mathbb R \setminus N_{[\mathbf n]}$. Let $N = \bigcup_{\mathbf n \in \mathbb N^N} N_{[\mathbf n]}$, which is of measure zero. If $\mathbf n \in \mathbb N^N$ and $t \in \left(L \cdot \mathbf n - L_{\max}, L \cdot \mathbf n\right) \setminus N$, then
\[
\begin{split}
\abs{\Theta_{[\mathbf n], t}^{L, \Lambda, A}} & \leq \sum_{\substack{[j] \in \mathcal J \\ L \cdot \mathbf n - L_j \leq t}} \abs{\widehat\Xi_{[\mathbf n - e_j], t}^{L, \Lambda, A}} \abs{\widehat A_{[j]}^{\Lambda}(t - L \cdot \mathbf n + L_j)} \\
 & \leq N M \sum_{[j] \in \mathcal J} f(L \cdot \mathbf n - L_j) \leq C \max_{s \in [t - L_{\max}, t + L_{\max}]} f(s),
\end{split}
\]
where $C = N^2 M$.
\end{proof}

As an immediate consequence of the previous proposition and Theorem \ref{TheoP0EquivP1}, we have the following theorem, which improves Theorem \ref{TheoP0EquivP1} by replacing $(\Theta, \Lambda)$-exponential type by $(\widehat\Xi, \Lambda)$-exponential type.

\begin{theorem}
\label{TheoStabShiftInv}
Let $\Lambda \in (\mathbb R_+^\ast)^N$ and $\mathcal A$ be a bounded shift-invariant subset of $L^\infty(\mathbb R, \mathcal M_d(\mathbb C)^N)$. For every $L \in V_+(\Lambda)$, $\Sigma_\delta(L, \mathcal A)$ is of $(\widehat\Xi, \Lambda)$-exponential type $\gamma$ if and only if it is of $(\Theta, \Lambda)$-ex\-po\-nen\-tial type $\gamma$.

As a consequence, for every $L \in V_+(\Lambda)$, if $\Sigma_\delta(L, \mathcal A)$ is of $(\widehat\Xi, \Lambda)$-exponential type $\gamma$ then, for every $p \in [1, +\infty]$, it is of exponential type $\gamma$ in $\mathsf X_p^\delta$. Conversely, for every $L \in W_+(\Lambda)$, if there exists $p \in [1, +\infty]$ such that $\Sigma_\delta(L, \mathcal A)$ is of exponential type $\gamma$ in $\mathsf X_p^\delta$, then it is of $(\widehat\Xi, \Lambda)$-exponential type $\gamma$. Finally, for every $L \in W_+(\Lambda)$ and $p \in [1, +\infty]$,
\begin{equation}
\label{LambdaLyapXi}
\lambda_p(L, \mathcal A) = \limsup_{L \cdot \mathbf n \to +\infty} \sup_{A \in \mathcal A} \esssup_{t \in \mathbb R} \frac{\ln\abs{\widehat\Xi_{[\mathbf n], t}^{L, \Lambda, A}}}{L \cdot \mathbf n}.
\end{equation}
\end{theorem}

\subsubsection{Arbitrary switching}
\label{SecSwitched}

We consider in this section $\mathcal A$ of the type $\mathcal A = L^\infty(\mathbb R, \mathfrak B)$ with $\mathfrak B$ a nonempty bounded subset of $\mathcal M_d(\mathbb C)^N$. In this case, $\Sigma_{\delta}(L, \mathcal A)$ corresponds to a switched system under arbitrary $\mathfrak B$-valued switching signals (for a general discussion on switched systems and their stability, see e.g. \cite{Liberzon2003Switching, Sun2011Stability} and references therein).

Motivated by formula \eqref{EqExplicitXiHat} for $\widehat\Xi_{[\mathbf n], t}$, we define below a new measure of the asymptotic behavior of $\Sigma_\delta(L, \mathcal A)$. For this, we introduce, for $\Lambda\in (\mathbb R_+^\ast)^N$ and $x \in \mathbb R_+$,
\begin{equation}
\label{LLambda}
\mathcal L(\Lambda) = \{\Lambda \cdot \mathbf n \;|\: \mathbf n \in \mathbb N^N\} \qquad \text{and} \qquad \mathcal L_x(\Lambda) = \mathcal L(\Lambda) \cap \left[0, x\right).
\end{equation}

\begin{definition}
We define
\begin{equation*}
\mu(\Lambda, \mathfrak B) = \limsup_{\substack{x \to +\infty \\ x \in \mathcal L(\Lambda)}} \sup_{\substack{B^r \in \mathfrak B \\ \text{for } r \in \mathcal L_x(\Lambda)}} \abs{\sum_{\substack{\mathbf n \in \mathbb N^N \\ \Lambda \cdot \mathbf n = x}} \sum_{v \in V_{\mathbf n}} \prod_{k=1}^{\abs{\mathbf n}_{1}} B_{v_k}^{\Lambda \cdot \mathbf p_v(k)}}^{\frac{1}{x}}.
\end{equation*}
\end{definition}

Note that $\mu(\Lambda, \mathfrak B)$ is independent of the choice of the norm $\abs{\cdot}$ and $\mu(\Lambda, \mathfrak B) = \mu(\Lambda, \overline{\mathfrak B})$. The main result of this section is the following.

\begin{theorem}
\label{TheoMuLyap}
Let $\Lambda \in (\mathbb R_+^\ast)^N$, $L \in V_+(\Lambda)$, $\mathfrak B$ be a nonempty bounded subset of $\mathcal M_d(\mathbb C)^N$, $\mathcal A = L^\infty(\mathbb R, \mathfrak B)$, and $p \in [1, +\infty]$. Set $m_1 = \min_{j \in \llbracket 1, N\rrbracket}\frac{\Lambda_j}{L_j}$ and $m_2 = \max_{j \in \llbracket 1, N\rrbracket}\frac{\Lambda_j}{L_j}$ if $\mu(\Lambda, \mathfrak B) < 1$, and $m_1 = \max_{j \in \llbracket 1, N\rrbracket}\frac{\Lambda_j}{L_j}$ and $m_2 = \min_{j \in \llbracket 1, N\rrbracket}\frac{\Lambda_j}{L_j}$ if $\mu(\Lambda, \mathfrak B) \geq 1$. Then the following assertions hold:
\begin{enumerate}[label={\bf \roman*.}, ref={(\roman*)}]
\item\label{MuLyap1} $\lambda_p(L, \mathcal A) \leq m_1 \ln \mu(\Lambda, \mathfrak B)$;

\item\label{MuLyap2} if $L \in W_+(\Lambda)$, then $m_2 \lambda_p(\Lambda, \mathcal A) \leq \lambda_p(L, \mathcal A) \leq m_1 \lambda_p(\Lambda, \mathcal A)$;

\item\label{MuLyap3} $\lambda_p(\Lambda, \mathcal A) = \ln \mu(\Lambda, \mathfrak B)$.
\end{enumerate}
\end{theorem}

\begin{proof}
Notice that \ref{MuLyap2} follows from \ref{MuLyap1} and \ref{MuLyap3} by exchanging the role of $L$ and $\Lambda$, since $\Lambda \in V_+(L)$ for every $L \in W_+(\Lambda)$.

Let us prove \ref{MuLyap1}. Since $\min_{j \in \llbracket 1, N\rrbracket} \frac{\Lambda_j}{L_j} \leq \frac{\Lambda \cdot \mathbf n}{L \cdot \mathbf n} \leq \max_{j \in \llbracket 1, N\rrbracket} \frac{\Lambda_j}{L_j}$ for every $\mathbf n \in \mathbb N^N \setminus \{0\}$, it suffices to show that, for every $\varepsilon > 0$, there exists $C > 0$ such that, for every $A \in \mathcal A$, $\mathbf n \in \mathbb N^N \setminus \{0\}$, and $t \in \mathbb R$, we have
\begin{equation}
\label{XiLeqRho}
\abs{\widehat\Xi_{[\mathbf n], t}^{L, \Lambda, A}} \leq C \left(\mu(\Lambda, \mathfrak B) + \varepsilon\right)^{\Lambda \cdot \mathbf n}.
\end{equation}

By definition of $\mu(\Lambda, \mathfrak B)$, there exists $X_0 \in \mathcal L(\Lambda)$ such that, for every $x \in \mathcal L(\Lambda)$ with $x \geq X_0$, we have
\[
\sup_{\substack{B^r \in \mathfrak B \\ \text{for } r \in \mathcal L_x(\Lambda)}} \abs{\sum_{\substack{\mathbf n \in \mathbb N^N \\ \Lambda \cdot \mathbf n = x}} \sum_{v \in V_{\mathbf n}} \prod_{k=1}^{\abs{\mathbf n}_{1}} B_{v_k}^{\Lambda \cdot \mathbf p_v(k)}} \leq \left(\mu(\Lambda, \mathfrak B) + \varepsilon\right)^{x}.
\]
Since $\mathfrak B$ is bounded, the quantity
\[
C^\prime = \max_{x \in \mathcal L_{X_0}(\Lambda)} \sup_{\substack{B^r \in \mathfrak B \\ \text{for } r \in \mathcal L_x(\Lambda)}} \abs{\sum_{\substack{\mathbf n \in \mathbb N^N \\ \Lambda \cdot \mathbf n = x}} \sum_{v \in V_{\mathbf n}} \prod_{k=1}^{\abs{\mathbf n}_{1}} B_{v_k}^{\Lambda \cdot \mathbf p_v(k)}}
\]
is finite. Setting $C = \max\{1, C^\prime, C^\prime (\mu(\Lambda, \mathfrak B) + \varepsilon)^{-X_0}\}$, we have, for every $x \in \mathcal L(\Lambda)$,
\begin{equation}
\label{BoundSup}
\sup_{\substack{B^r \in \mathfrak B \\ \text{for } r \in \mathcal L_x(\Lambda)}} \abs{\sum_{\substack{\mathbf n \in \mathbb N^N \\ \Lambda \cdot \mathbf n = r}} \sum_{v \in V_{\mathbf n}} \prod_{k=1}^{\abs{\mathbf n}_{1}} B_{v_k}^{\Lambda \cdot \mathbf p_v(k)}} \leq C \left(\mu(\Lambda, \mathfrak B) + \varepsilon\right)^{x}.
\end{equation}

Define $\varphi_L: \mathcal L(\Lambda) \to \mathcal L(L)$ by $\varphi_L(\Lambda \cdot \mathbf n) = L \cdot \mathbf n$. This is a well-defined function since $L \in V_+(\Lambda)$. Let $A \in \mathcal A$, $\mathbf n \in \mathbb N^N \setminus \{0\}$, and $t \in \mathbb R$. By Proposition \ref{PropXiHatRecurrence},
\begin{equation}
\label{ExplicitXiHat}
\widehat\Xi_{[\mathbf n], t}^{L, \Lambda, A} = \sum_{\mathbf n^\prime \in [\mathbf n] \cap \mathbb N^N} \sum_{v \in V_{\mathbf n^\prime}} \prod_{k=1}^{\abs{\mathbf n^\prime}_{1}} A_{v_k}\left(t - L \cdot \mathbf p_v(k)\right).
\end{equation}
For $r \in \mathcal L_{\Lambda \cdot \mathbf n}(\Lambda)$, we set $B^r = A(t - \varphi_L(r)) \in \mathfrak B$. Thus, for every $\mathbf n^\prime \in [\mathbf n] \cap \mathbb N^N$, $v \in V_{\mathbf n^\prime}$, and $k \in \llbracket 1, \abs{\mathbf n^\prime}_{1} \rrbracket$, we have, by definition of $\varphi_L$,
\begin{equation}
\label{BandA}
B_{v_k}^{\Lambda \cdot \mathbf p_v(k)} = A_{v_k}\left(t - \varphi_L \left(\Lambda \cdot \mathbf p_v(k)\right)\right) = A_{v_k}\left(t - L \cdot \mathbf p_v(k)\right).
\end{equation}
We thus obtain \eqref{XiLeqRho} by combining \eqref{BoundSup}, \eqref{ExplicitXiHat} and \eqref{BandA}.

In order to prove \ref{MuLyap3}, we are left to show the inequality $\ln \mu(\Lambda, \mathfrak B) \leq \lambda_p(\Lambda, \mathcal A)$. Let $x \in \mathcal L(\Lambda)$ and $A^0 \in \mathfrak B$. For $r \in \mathcal L_x(\Lambda)$, let $B^r \in \mathfrak B$. We define
\[
\zeta = \frac{1}{2} \min_{\substack{y_1, y_2 \in \mathcal L_x(\Lambda) \\ y_1 \not = y_2}} \abs{y_1 - y_2} > 0.
\]
Let $A = (A_1, \dotsc, A_N) \in \mathcal A$ be defined for $t \in \mathbb R$ by
\[
A(t) = 
\left\{
\begin{aligned}
& B^{\Lambda \cdot \mathbf m}, & & \begin{aligned}& \text{ if } \mathbf m \in \mathbb N^N \text{ is such that } \Lambda \cdot \mathbf m < x \\ & \text{ and } t \in (- \Lambda \cdot \mathbf m - \zeta, - \Lambda \cdot \mathbf m + \zeta),\end{aligned} \\
& A^0, & & \text{ otherwise}.
\end{aligned}
\right.
\]
The function $A$ is well-defined since the sets $(- \Lambda \cdot \mathbf m - \zeta, - \Lambda \cdot \mathbf m + \zeta)$ are disjoint for $\mathbf m \in \mathbb N^N$ with $\Lambda \cdot \mathbf m < x$. For every $\mathbf n \in \mathbb N^N$ with $\Lambda \cdot \mathbf n = x$, every $v \in V_{\mathbf n}$, $t \in (-\zeta, \zeta)$, and $k \in \llbracket 1, \abs{\mathbf n}_{1} \rrbracket$, we have
\[
A_{v_k}\left(t - \Lambda \cdot \mathbf p_v(k)\right) = B_{v_k}^{\Lambda \cdot \mathbf p_v(k)},
\]
and then, for every $\mathbf n^\prime \in \mathbb N^N$ with $\Lambda \cdot \mathbf n^\prime = x$, we have
\[
\sum_{\substack{\mathbf n \in \mathbb N^N \\ \Lambda \cdot \mathbf n = x}} \sum_{v \in V_{\mathbf n}} \prod_{k=1}^{\abs{\mathbf n}_{1}} B_{v_k}^{\Lambda \cdot \mathbf p_v(k)} = \sum_{\substack{\mathbf n \in \mathbb N^N \\ \Lambda \cdot \mathbf n = x}} \sum_{v \in V_{\mathbf n}} \prod_{k=1}^{\abs{\mathbf n}_{1}} A_{v_k}\left(t - \Lambda \cdot \mathbf p_v(k)\right) = \widehat\Xi_{[\mathbf n^\prime], t}^{\Lambda, \Lambda, A}.
\]
Hence, for every $\mathbf n^\prime \in \mathbb N^N$ with $\Lambda \cdot \mathbf n^\prime = x$, we have
\[
\abs{\sum_{\substack{\mathbf n \in \mathbb N^N \\ \Lambda \cdot \mathbf n = x}} \sum_{v \in V_{\mathbf n}} \prod_{k=1}^{\abs{\mathbf n}_{1}} B_{v_k}^{\Lambda \cdot \mathbf p_v(k)}}^{\frac{1}{x}} \leq \sup_{A \in \mathcal A} \esssup_{t \in \mathbb R} \abs{\widehat\Xi_{[\mathbf n^\prime], t}^{\Lambda, \Lambda, A}}^{\frac{1}{\Lambda \cdot \mathbf n^\prime}}.
\]
Since this holds for every choice of $B^r \in \mathfrak B$, $r \in \mathcal L_x(\Lambda)$, and $x \in \mathcal L(\Lambda)$, we deduce from \eqref{LambdaLyapXi} that $\ln \mu(\Lambda, \mathfrak B) \leq \lambda_p(\Lambda, \mathcal A)$.
\end{proof}

\begin{remark}
\label{RemkAClosed}
Since $\mu(\Lambda, \mathfrak B) = \mu(\Lambda, \overline{\mathfrak B})$, one has $\lambda_p(\Lambda, \mathcal A) = \lambda_p(\Lambda, L^\infty(\mathbb R, \overline{\mathfrak B}))$.
\end{remark}

As regards exponential stability of $\Sigma_\delta(L, \mathcal A)$, we deduce from the previous theorem and Remark \ref{RemkTransformation} the following corollary.

\begin{corollary}
\label{CoroSilkRho0}
Let $\Lambda \in (\mathbb R_+^\ast)^N$, $\mathfrak B$ be a nonempty bounded subset of $\mathcal M_d(\mathbb C)^N$, and $\mathcal A = L^\infty(\mathbb R, \mathfrak B)$. The following statements are equivalent:
\begin{enumerate}[label={\bf \roman*.}, ref={(\roman*)}]
\item $\mu(\Lambda, \mathfrak B) < 1$;
\item $\Sigma_\delta(\Lambda, \mathcal A)$ is exponentially stable in $\mathsf X_p^\delta$ for some $p \in [1, +\infty]$;
\item $\Sigma_\delta(L, \mathcal A)$ is exponentially stable in $\mathsf X_p^\delta$ for every $L \in V_+(\Lambda)$ and $p \in [1, +\infty]$.
\end{enumerate}
Moreover, for every $p \in [1, +\infty]$,
\begin{equation*}
\lambda_p(\Lambda, \mathcal A) = \inf\{\nu \in \mathbb R \;|\: \mu(\Lambda, \mathfrak B_{-\nu}) < 1\},
\end{equation*}
where $\mathfrak B_{-\nu} = \{(e^{-\nu \Lambda_1} B_1, \dotsc, e^{-\nu \Lambda_N} B_N) \;|\: (B_1, \dotsc, B_N) \in \mathfrak B\}$.
\end{corollary}

Corollary \ref{CoroSilkRho0} is reminiscent of the well-known characterization of stability in the autonomous case proved by Hale and Silkowski when $\Lambda$ has rationally independent components (see \cite[Theorem 5.2]{Avellar1980Zeros}) and in a more general setting by Michiels \emph{et al.} in \cite{Michiels2009Strong}. In such a characterization, $(1, \dotsc, 1)$ is assumed to be in $V(\Lambda)$ and $\mu(\Lambda, \mathfrak B)$ is replaced in the statement of Corollary \ref{CoroSilkRho0} by
\begin{equation*}
\rho_{\mathrm{HS}}(\Lambda, A) = \max_{(\theta_1, \dotsc, \theta_N) \in \widetilde V(\Lambda)} \spr\left(\sum_{j=1}^N A_j e^{i \theta_j}\right),
\end{equation*}
where $\widetilde V(\Lambda)$ is the image of $V(\Lambda)$ by the canonical projection from $\mathbb R^N$ onto the torus $(\mathbb R / 2 \pi \mathbb Z)^N$. (Notice that $\widetilde V(\Lambda)$ is compact since the matrix $B$ characterizing $V(\Lambda)$ in Proposition \ref{PropRatDep} has integer coefficients.)

We propose below a generalization of $\rho_{\mathrm{HS}}(\Lambda, A)$ to the non-autonomous case defined as follows.
\begin{definition}
For $\Lambda \in (\mathbb R_+^\ast)^N$, $\mathfrak B$ a nonempty bounded subset of $\mathcal M_d(\mathbb C)^N$, and $\mathcal L(\Lambda)$ given by \eqref{LLambda}, we set
\begin{equation*}
\mu_{\mathrm{HS}}(\Lambda, \mathfrak B) = \limsup_{n \to +\infty} \sup_{(\theta_1, \dotsc, \theta_N) \in \widetilde V(\Lambda)} \sup_{\substack{B^r \in \mathfrak B \\ \text{for } r \in \mathcal L_{n \Lambda_{\max}}(\Lambda)}} \abs{\sum_{v \in \llbracket 1, N\rrbracket^n} \prod_{k=1}^n B_{v_k}^{\Lambda \cdot \mathbf p_v(k)} e^{i \theta_{v_k}}}^{\frac{1}{n}}.
\end{equation*}
\end{definition}

Let us check in the next proposition that $\mu_{\mathrm{HS}}$ actually extends $\rho_{\mathrm{HS}}$.

\begin{proposition}
\label{PropMuRho}
Let $A = (A_1, \dotsc, A_N) \in \mathcal M_d(\mathbb C)^N$ and $\mathfrak B = \{A\}$. Then one has $\mu_{\mathrm{HS}}(\Lambda, \mathfrak B) = \rho_{\mathrm{HS}}(\Lambda, A)$.
\end{proposition}

\begin{proof}
One has
\[
\begin{split}
\max_{(\theta_1, \dotsc, \theta_N) \in \widetilde V(\Lambda)} \spr\left(\sum_{j=1}^N A_j e^{i \theta_j}\right) & = \max_{(\theta_1, \dotsc, \theta_N) \in \widetilde V(\Lambda)} \lim_{n \to +\infty} \abs{\left(\sum_{j=1}^N A_j e^{i \theta_j}\right)^n}^{\frac{1}{n}} \\
 & = \lim_{n \to +\infty} \sup_{(\theta_1, \dotsc, \theta_N) \in \widetilde V(\Lambda)} \abs{\left(\sum_{j=1}^N A_j e^{i \theta_j}\right)^n}^{\frac{1}{n}} \\
 & = \lim_{n \to +\infty} \sup_{(\theta_1, \dotsc, \theta_N) \in \widetilde V(\Lambda)} \abs{\sum_{v \in \llbracket 1, N\rrbracket^n} \prod_{k=1}^n A_{v_k} e^{i \theta_{v_k}}}^{\frac{1}{n}},
\end{split}
\]
where the second equality is obtained as consequence of the uniformity of the Gelfand limit on bounded subsets of $\mathcal M_d(\mathbb C)$ (see, for instance, \cite[Proposition 3.3.5]{Green1996Uniform}).
\end{proof}

In the sequel, we relate $\mu_{\mathrm{HS}}(\Lambda, \mathfrak B)$ to a modified version of the expression \eqref{LambdaLyapXi} of $\lambda_p(L, \mathcal A)$.

\begin{definition}
For $L \in V_+(\Lambda)$ and $\mathcal A$ a set of functions $A: \mathbb R \to \mathcal M_d(\mathbb C)^N$, we define
\[\lambda_{\mathrm{HS}}(L, \mathcal A) = \limsup_{\substack{\abs{\mathbf n}_1 \to +\infty \\ \mathbf n \in \mathbb N^N}} \sup_{A \in \mathcal A} \esssup_{t \in \mathbb R} \frac{\ln\abs{\widehat\Xi_{[\mathbf n], t}^{L, \Lambda, A}}}{\abs{\mathbf n}_1}.\]
\end{definition}

\begin{remark}
\label{RemkLyapHS}
Since $L_{\min} \abs{\mathbf n}_1 \leq L \cdot \mathbf n \leq L_{\max} \abs{\mathbf n}_1$ for every $L \in V_+(\Lambda)$ and $\mathbf n \in \mathbb N^N$, it follows immediately from \eqref{LambdaLyapXi} that, for every $p \in [1, +\infty]$,
\[
\begin{aligned}
L_{\min} \lambda_p(L, \mathcal A) \leq \lambda_{\mathrm{HS}}(L, \mathcal A) \leq L_{\max} \lambda_p(L, \mathcal A), & \qquad \text{ if } \lambda_p(L, \mathcal A) \geq 0, \\
L_{\max} \lambda_p(L, \mathcal A) \leq \lambda_{\mathrm{HS}}(L, \mathcal A) \leq L_{\min} \lambda_p(L, \mathcal A), & \qquad \text{ if } \lambda_p(L, \mathcal A) < 0.
\end{aligned}
\]
In particular, the signs of $\lambda_{\mathrm{HS}}(L, \mathcal A)$ and $\lambda_p(L, \mathcal A)$ being equal, they both characterize the exponential stability of $\Sigma_\delta(L, \mathcal A)$.
\end{remark}

\begin{theorem}
\label{TheoMuHSLyap}
Let $\Lambda \in (\mathbb R_+^\ast)^N$, $\mathfrak B$ be a nonempty bounded subset of $\mathcal M_d(\mathbb C)^N$, and $\mathcal A = L^\infty(\mathbb R, \mathfrak B)$. Set $m = \inf \left\{1, \frac{\abs{z_+}_1}{\abs{z_-}_1} \;\middle|\: z \in Z(\Lambda) \setminus \{0\}\right\}$ if $\mu_{\mathrm{HS}}(\Lambda, \mathfrak B) < 1$ and $m = \sup \left\{1, \frac{\abs{z_+}_1}{\abs{z_-}_1} \;\middle|\: z \in Z(\Lambda) \setminus \{0\}\right\}$ if $\mu_{\mathrm{HS}}(\Lambda, \mathfrak B) \geq 1$. Then the following assertions hold:
\begin{enumerate}[label={\bf \roman*.}, ref={(\roman*)}]
\item\label{MuHSLyap1} for every $L \in V_+(\Lambda)$, $\lambda_{\mathrm{HS}}(L, \mathcal A) \leq m \ln\mu_{\mathrm{HS}}(\Lambda, \mathfrak B)$;
\item\label{MuHSLyap3} if $(1, \dotsc, 1) \in V(\Lambda)$ and $L \in W_+(\Lambda)$, one has $\lambda_{\mathrm{HS}}(L, \mathcal A) = \ln\mu_{\mathrm{HS}}(\Lambda, \mathfrak B)$.
\end{enumerate}
\end{theorem}

\begin{proof}
We start by proving \ref{MuHSLyap1}. It is enough to show that, for every $\varepsilon > 0$ small enough, there exists $C > 0$ such that, for every $A \in \mathcal A$, $\mathbf n \in \mathbb N^N \setminus \{0\}$, and $t \in \mathbb R$, we have
\[
\abs{\widehat\Xi_{[\mathbf n], t}^{L, \Lambda, A}} \leq C (1 + \abs{\mathbf n}_1) \left(\mu_{\mathrm{HS}}(\Lambda, \mathfrak B) + \varepsilon\right)^{m \abs{\mathbf n}_1}.
\]

Let $L \in V_+(\Lambda)$ and $\varepsilon > 0$ be such that $\mu_{\mathrm{HS}}(\Lambda, \mathfrak B) + \varepsilon < 1$ if $\mu_{\mathrm{HS}}(\Lambda, \mathfrak B) < 1$. We can proceed as in the proof of Theorem \ref{TheoMuLyap} to obtain a finite constant $C_0 > 0$ such that, for every $n \in \mathbb N^\ast$,
\begin{equation}
\label{BoundSupRho1}
\sup_{(\theta_1, \dotsc, \theta_N) \in \widetilde V(\Lambda)} \sup_{\substack{B^r \in \mathfrak B \\ \text{for } r \in \mathcal L_{n\Lambda_{\max}}(\Lambda)}} \abs{\sum_{v \in \llbracket 1, N\rrbracket^n} \prod_{k=1}^{n} B_{v_k}^{\Lambda \cdot \mathbf p_v(k)} e^{i \theta_{v_k}}} \leq C_0 \left(\mu_{\mathrm{HS}}(\Lambda, \mathfrak B) + \varepsilon\right)^{n}.
\end{equation}

Let $A \in \mathcal A$, $t \in \mathbb R$, and $\varphi_L$ be as in the proof of Theorem \ref{TheoMuLyap}. For $r \in \mathcal L_{n \Lambda_{\max}}(\Lambda)$, we set $B^r = A(t - \varphi_L(r))$, and similarly to the proof of Theorem \ref{TheoMuLyap}, \eqref{BandA} holds for every $v \in \llbracket 1, N\rrbracket^n$ and $k \in \llbracket 1, n \rrbracket$. Thus \eqref{BoundSupRho1} implies that, for every $n \in \mathbb N^\ast$ and $\theta \in \widetilde V(\Lambda)$,
\[
\abs{\sum_{v \in \llbracket 1, N\rrbracket^n} \prod_{k=1}^{n} A_{v_k}\left(t - L \cdot \mathbf p_v(k)\right) e^{i \theta_{v_k}}} \leq C_0 \left(\mu_{\mathrm{HS}}(\Lambda, \mathfrak B) + \varepsilon\right)^{n}.
\]
Since
\[
\begin{split}
\sum_{v \in \llbracket 1, N\rrbracket^n} \prod_{k=1}^{n} A_{v_k}\left(t - L \cdot \mathbf p_v(k)\right) e^{i \theta_{v_k}} & = \sum_{\substack{\mathbf n \in \mathbb N^N \\ \abs{\mathbf n}_{1} = n}} \sum_{v \in V_{\mathbf n}} \prod_{k=1}^{\abs{\mathbf n}_{1}} A_{v_k}\left(t - L \cdot \mathbf p_v(k)\right) e^{i \theta_{v_k}} \\
 & = \sum_{\substack{\mathbf n \in \mathbb N^N \\ \abs{\mathbf n}_{1} = n}} e^{i \mathbf n \cdot \theta} \sum_{v \in V_{\mathbf n}} \prod_{k=1}^{\abs{\mathbf n}_{1}} A_{v_k}\left(t - L \cdot \mathbf p_v(k)\right) \\
 & = \sum_{\substack{\mathbf n \in \mathbb N^N \\ \abs{\mathbf n}_{1} = n}} e^{i \mathbf n \cdot \theta} \Xi_{\mathbf n, t}^{L, A},
\end{split}
\]
we obtain that, for every $n \in \mathbb N^\ast$ and $\theta \in \widetilde V(\Lambda)$,
\begin{equation}
\label{ConvExpoA}
\abs{\sum_{\substack{\mathbf n \in \mathbb N^N \\ \abs{\mathbf n}_{1} = n}} e^{i \mathbf n \cdot \theta} \Xi_{\mathbf n, t}^{L, A}} \leq C_0 \left(\mu_{\mathrm{HS}}(\Lambda, \mathfrak B) + \varepsilon\right)^{n}.
\end{equation}

Following Proposition \ref{PropRatDep}, fix $h \in \llbracket 1, N\rrbracket$ and $B \in \mathcal M_{N, h}(\mathbb Z)$ with $\rank(B) = h$ such that $\Lambda = B \ell_0$ for $\ell_0 \in (\mathbb R_+^\ast)^h$ with rationally independent components. Let $M \in \mathcal M_h(\mathbb R)$ be an invertible matrix such that $\ell_0 = M e_1$, where $e_1$ is the first vector of the canonical basis of $\mathbb R^h$, in such a way that $\Lambda = B M e_1$. For $n \in \mathbb N$, we define the function $f_n: \mathbb R^h \to \mathcal M_d(\mathbb C)$ by
\[
f_n(\nu) = \sum_{\substack{\mathbf n \in \mathbb N^N \\ \abs{\mathbf n}_{1} = n}} e^{i \mathbf n \cdot B M \nu} \Xi_{\mathbf n, t}^{L, A}.
\]
We claim that, for every $\mathbf n_0 \in \mathbb N^N$,
\begin{equation}
\label{FourierInverse}
\lim_{R \to +\infty} \frac{1}{(2R)^h} \int_{[-R, R]^h} f_n(\nu) e^{-i \mathbf n_0 \cdot B M \nu} d\nu = \sum_{\substack{\mathbf n \in [\mathbf n_0] \cap \mathbb N^N \\ \abs{\mathbf n}_{1} = n}} \Xi_{\mathbf n, t}^{L, A}.
\end{equation}
Indeed, we have
\[
\frac{1}{(2R)^h} \int_{[-R, R]^h} f_n(\nu) e^{-i \mathbf n_0 \cdot B M \nu} d\nu = \sum_{\substack{\mathbf n \in \mathbb N^N \\ \abs{\mathbf n}_{1} = n}} \Xi_{\mathbf n, t}^{L, A} \frac{1}{(2R)^h} \int_{[-R, R]^h} e^{i (\mathbf n - \mathbf n_0) \cdot B M \nu} d\nu.
\]
If $\mathbf n \in \mathbb N^N$ is such that $\Lambda \cdot \mathbf n = \Lambda \cdot \mathbf n_0$, then $\Lambda \cdot (\mathbf n - \mathbf n_0) = 0$, and therefore $\mathbf n - \mathbf n_0 \in Z(\Lambda) \subset V(\Lambda)^{\perp} = (\range B)^\perp$. One gets $(\mathbf n - \mathbf n_0) \cdot B M \nu = 0$ for every $\nu \in \mathbb R^h$, implying that
\begin{equation*}
\frac{1}{(2R)^h} \int_{[-R, R]^h} e^{i (\mathbf n - \mathbf n_0) \cdot B M \nu} d\nu = 1.
\end{equation*}
If now $\Lambda \cdot \mathbf n \not = \Lambda \cdot \mathbf n_0$, set $\xi = \Lambda \cdot (\mathbf n - \mathbf n_0)$, which is nonzero. Then
\[
\abs{\frac{1}{(2R)^h} \int_{[-R, R]^h} e^{i (\mathbf n - \mathbf n_0) \cdot B M \nu} d\nu} \leq \frac{1}{2R} \abs{\int_{-R}^R e^{i \xi \nu_1} d\nu_1} = \abs{\frac{\sin(\xi R)}{\xi R}} \xrightarrow[R \to +\infty]{} 0,
\]
which gives \eqref{FourierInverse}.

We can now combine \eqref{ConvExpoA} and \eqref{FourierInverse} to obtain that, for every $n \in \mathbb N^\ast$ and $\mathbf n_0 \in \mathbb N^N \setminus \{0\}$,
\begin{equation*}
\abs{\sum_{\substack{\mathbf n \in [\mathbf n_0] \cap \mathbb N^N \\ \abs{\mathbf n}_{1} = n}} \Xi_{\mathbf n, t}^{L, A}} \leq C_0 (\mu_{\mathrm{HS}}(\Lambda, \mathfrak B) + \varepsilon)^n.
\end{equation*}
Set $m_0 = \sup \left\{1, \frac{\abs{z_+}_1}{\abs{z_-}_1}\;\middle|\: z \in Z(\Lambda) \setminus \{0\}\right\}$ and notice that, since $Z(\Lambda) \allowbreak= -Z(\Lambda)$, one has $\frac{1}{m_0} = \inf \left\{1, \frac{\abs{z_+}_1}{\abs{z_-}_1}\;\middle|\: z \in Z(\Lambda) \setminus \{0\}\right\}$. We claim that, if $\mathbf n, \mathbf n_0 \in \mathbb N^N$ and $\Lambda \cdot \mathbf n = \Lambda \cdot \mathbf n_0$, then $\frac{1}{m_0} \abs{\mathbf n_0}_1 \leq \abs{\mathbf n}_1 \leq m_0 \abs{\mathbf n_0}_1$. Indeed, let $z = \mathbf n - \mathbf n_0 \in Z(\Lambda)$ and $\mathbf n_1 = \mathbf n_0 - z_- \in \mathbb N^N$. Then one has
\[
\frac{\abs{\mathbf n}_1}{\abs{\mathbf n_0}_1} = \frac{\abs{z_+}_1 + \abs{\mathbf n_1}_1}{\abs{z_-}_1 + \abs{\mathbf n_1}_1} \in \left[\frac{1}{m_0},\; m_0\right].
\]
Hence, for every $\mathbf n_0 \in \mathbb N^N \setminus \{0\}$,
\begin{equation*}
\widehat\Xi_{[\mathbf n_0], t}^{L, \Lambda, A} = \sum_{n=0}^{+\infty} \sum_{\substack{\mathbf n \in [\mathbf n_0] \cap \mathbb N^N \\ \abs{\mathbf n}_{1} = n}} \Xi_{\mathbf n, t}^{L, A} = \sum_{n \in \left\llbracket\frac{\abs{\mathbf n_0}_1}{m_0},\; m_0 \abs{\mathbf n_0}_1 \right\rrbracket} \sum_{\substack{\mathbf n \in [\mathbf n_0] \cap \mathbb N^N \\ \abs{\mathbf n}_{1} = n}} \Xi_{\mathbf n, t}^{L, A},
\end{equation*}
and we conclude that
\[
\abs{\widehat\Xi_{[\mathbf n_0], t}^{L, \Lambda, A}} \leq \sum_{n \in \left\llbracket\frac{\abs{\mathbf n_0}_1}{m_0},\; m_0 \abs{\mathbf n_0}_1 \right\rrbracket} C_0 (\mu_{\mathrm{HS}}(\Lambda, \mathfrak B) + \varepsilon)^n \leq C (1 + \abs{\mathbf n_0}_1) (\mu_{\mathrm{HS}}(\Lambda, \mathfrak B) + \varepsilon)^{m \abs{\mathbf n_0}_1},
\]
for some $C > 0$. This concludes the proof of \ref{MuHSLyap1}.

Suppose now that $(1, \dotsc, 1) \in V(\Lambda)$. Then $\abs{z_+}_1 = \abs{z_-}_1$ for every $z \in Z(\Lambda)$, and hence \ref{MuHSLyap1} yields $\lambda_{\mathrm{HS}}(L, \mathcal A) \leq \ln\mu_{\mathrm{HS}}(\Lambda, \mathfrak B)$ for every $L \in V_+(\Lambda)$. We claim that it is enough to prove \ref{MuHSLyap3} only for $L = \Lambda$. Indeed, assume that $\lambda_{\mathrm{HS}}(\Lambda, \mathcal A) = \ln \mu_{\mathrm{HS}}(\Lambda, \mathfrak B)$. In particular,
\begin{equation}
\label{LambdaLLeqLambdaLambda}
\lambda_{\mathrm{HS}}(L, \mathcal A) \leq \lambda_{\mathrm{HS}}(\Lambda, \mathcal A)
\end{equation}
for every $L \in V_+(\Lambda)$. Since $\Lambda \in V_+(L)$ if $L \in W_+(\Lambda)$, by exchanging the role of $L$ and $\Lambda$ in \eqref{LambdaLLeqLambdaLambda}, we deduce that $\lambda_{\mathrm{HS}}(L, \mathcal A) = \lambda_{\mathrm{HS}}(\Lambda, \mathcal A)$ for every $L \in W_+(\Lambda)$, and hence \ref{MuHSLyap3}.

Let $n \in \mathbb N^\ast$ and $B^r \in \mathfrak B$ for $r \in \mathcal L_{n \Lambda_{\max}}(\Lambda)$. As in the argument for \ref{MuLyap3} in Theorem \ref{TheoMuLyap}, there exist $\zeta > 0$ and a function $A: \mathbb R \to \mathcal M_d(\mathbb C)^N$ such that, for every $v \in \llbracket 1, N\rrbracket^n$, $t \in (-\zeta, \zeta)$, and $k \in \llbracket 1, n\rrbracket$, we have
\[
A_{v_k}\left(t - \Lambda \cdot \mathbf p_v(k)\right) = B_{v_k}^{\Lambda \cdot \mathbf p_v(k)} \;\; \text{ and } \;\; \sum_{v \in \llbracket 1, N\rrbracket^n} \prod_{k=1}^{n} B_{v_k}^{\Lambda \cdot \mathbf p_v(k)} e^{i \theta_{v_k}} = \sum_{\substack{\mathbf n \in \mathbb N^N \\ \abs{\mathbf n}_{1} = n}} e^{i \mathbf n \cdot \theta} \Xi_{\mathbf n, t}^{\Lambda, A}.
\]

Denote $\mathcal Z_+ = \{[\mathbf n] \in \mathcal Z \;|\: [\mathbf n] \cap \mathbb N^N \not = \emptyset\}$. Since $(1, \dotsc, 1) \in V(\Lambda)$, one deduces that, if $\mathbf n, \mathbf n^\prime \in \mathbb N^N$ are such that $\mathbf n \approx \mathbf n^\prime$, then $e^{i \mathbf n \cdot \theta} = e^{i \mathbf n^\prime \cdot \theta}$ for every $\theta \in \widetilde V(\Lambda)$ and $\abs{\mathbf n}_1 = \abs{\mathbf n^\prime}_1$. We set $\abs{[\mathbf n]}_1 = \abs{\mathbf n}_1$ for every $\mathbf n \in \mathbb N^N$. Then
\[
\sum_{\substack{\mathbf n \in \mathbb N^N \\ \abs{\mathbf n}_{1} = n}} e^{i \mathbf n \cdot \theta} \Xi_{\mathbf n, t}^{\Lambda, A} = \sum_{\substack{[\mathbf n] \in \mathcal Z_+ \\ \abs{[\mathbf n]}_1 = n}} \sum_{\mathbf n^\prime \in [\mathbf n] \cap \mathbb N^N} e^{i \mathbf n^\prime \cdot \theta} \Xi_{\mathbf n^\prime, t}^{\Lambda, A} = \sum_{\substack{[\mathbf n] \in \mathcal Z_+ \\ \abs{[\mathbf n]}_1 = n}} e^{i \mathbf n \cdot \theta} \widehat\Xi_{[\mathbf n], t}^{\Lambda, \Lambda, A}.
\]
We clearly have $\# \{[\mathbf n] \in \mathcal Z_+ \;|\: \abs{[\mathbf n]}_1 = n\} \leq \# \{\mathbf n \in \mathbb N^N \;|\: \abs{\mathbf n}_{1} = n\} = \binom{n + N - 1}{N - 1} \leq (n+1)^{N-1}$, and we get that, for every $\theta \in \widetilde V(\Lambda)$ and $\mathbf n \in \mathbb N^N$ with $\abs{\mathbf n}_1 = n$,
\[
\abs{\sum_{v \in \llbracket 1, N\rrbracket^n} \prod_{k=1}^{n} B_{v_k}^{\Lambda \cdot \mathbf p_v(k)} e^{i \theta_{v_k}}}^{\frac{1}{n}} \leq (n + 1)^{\frac{N - 1}{n}} \sup_{A \in \mathcal A} \esssup_{t \in \mathbb R} \abs{\widehat\Xi_{[\mathbf n], t}^{\Lambda, \Lambda, A}}^{\frac{1}{n}}.
\]
Since the above inequality holds for every choice of $B^r \in \mathfrak B$, $r \in \mathcal L_{n \Lambda_{\max}}(\Lambda)$, $n \in \mathbb N^\ast$, we deduce that $\ln \mu_{\mathrm{HS}}(\Lambda, \mathfrak B) \leq \lambda_{\mathrm{HS}}(\Lambda, \mathcal A)$. This concludes the proof of Theorem \ref{TheoMuHSLyap}.
\end{proof}

The next corollary, which follows directly from the above theorem and Remarks \ref{RemkTransformation} and \ref{RemkLyapHS}, generalizes the stability criterion in \cite{Michiels2009Strong, Avellar1980Zeros} to the nonautonomous case (see Proposition \ref{PropMuRho}).

\begin{corollary}
\label{CoroSilkRhoHS}
Let $\Lambda \in (\mathbb R_+^\ast)^N$, $\mathfrak B$ be a nonempty bounded subset of $\mathcal M_d(\mathbb C)^N$, and $\mathcal A = L^\infty(\mathbb R, \mathfrak B)$. Consider the following statements:
\begin{enumerate}[label={\bf \roman*.}, ref={(\roman*)}]
\item\label{Silk1} $\mu_{\mathrm{HS}}(\Lambda, \mathfrak B) < 1$;
\item\label{Silk2} $\Sigma_\delta(\Lambda, \mathcal A)$ is exponentially stable in $\mathsf X_p^\delta$ for some $p \in [1, +\infty]$;
\item\label{Silk3} $\Sigma_\delta(L, \mathcal A)$ is exponentially stable in $\mathsf X_p^\delta$ for every $L \in V_+(\Lambda)$ and $p \in [1, +\infty]$.
\end{enumerate}
Then \ref{Silk1} $\implies$ \ref{Silk3} $\implies$ \ref{Silk2}. If moreover $(1, \dotsc, 1) \in V(\Lambda)$, we also have \ref{Silk2} $\implies$ \ref{Silk1} and, for every $p \in [1, +\infty]$,
\begin{equation*}
\lambda_p(\Lambda, \mathcal A) = \inf\{\nu \in \mathbb R \;|\: \mu_{\mathrm{HS}}(\Lambda, \mathfrak B_{-\nu}) < 1\},
\end{equation*}
where $\mathfrak B_{-\nu} = \{(e^{-\nu \Lambda_1} B_1, \dotsc, e^{-\nu \Lambda_N} B_N) \;|\: (B_1, \dotsc, B_N) \in \mathfrak B\}$.
\end{corollary}


\section{Transport system}
\label{SecTransport}

For $L = (L_1, \dotsc, L_N) \in (\mathbb R_+^\ast)^N$ and $M = (m_{ij})_{i, j \in \llbracket 1, N\rrbracket}: \mathbb R \to \mathcal M_N(\mathbb C)$, we consider the system of transport equations
\begin{equation}
\label{MainSystTransport}
\Sigma_{\tau}(L, M): \;
\left\{
\begin{aligned}
 & \pad{u_i}{t}(t, x) + \pad{u_i}{x}(t, x) = 0, & \; & i \in \llbracket 1, N\rrbracket,\; t \in \left[0, +\infty\right),\; x \in [0, L_i], \\
 & u_i(t, 0) = \sum_{j=1}^{N} m_{ij}(t) u_j(t, L_j), & & i \in \llbracket 1, N\rrbracket,\; t \in \left[0, +\infty\right),
\end{aligned}
\right.
\end{equation}
where, for $i \in \llbracket 1, N\rrbracket$, $u_i(\cdot, \cdot)$ takes values in $\mathbb C$.

The time-varying matrix $M$ represents transmission conditions and in particular it may encode an underlying network for \eqref{MainSystTransport}, where the graph structure is determined by the non-zero coefficients of $M$. When no regularity assumptions are made on the function $M$, we may not have solutions for \eqref{MainSystTransport} in the classical sense in $C^1(\mathbb R_+ \times [0, L_i])$ nor in $C^0(\mathbb R_+, W^{1, p}([0, L_i], \mathbb C)) \cap C^1(\mathbb R_+, L^p([0, L_i], \mathbb C))$. We thus consider the following weaker definition of solution.

\begin{definition}
\label{DefiSolTransport}
Let $M: \mathbb R \to \mathcal M_N(\mathbb C)$ and $u_{i, 0}: [0, L_i] \to \mathbb C$ for $i \in \llbracket 1, N\rrbracket$. We say that $(u_i)_{i \in \llbracket 1, N\rrbracket}$ is a \emph{solution} of $\Sigma_\tau(L, M)$ with initial condition $(u_{i, 0})_{i \in \llbracket 1, N\rrbracket}$ if $u_i: \mathbb R_+ \times [0, L_i] \to \mathbb C$, $i \in \llbracket 1, N\rrbracket$, satisfy the second equation of \eqref{MainSystTransport}, and, for every $i \in \llbracket 1, N\rrbracket$, $t \geq 0$, $x \in [0, L_i]$, $s \in [-\min(x, t), L_i - x]$, one has $u_i(t + s, x + s) = u_i(t, x)$ and $u_i(0, x) = u_{i, 0}(x)$.
\end{definition}

\subsection{Equivalent difference equation}

For $i \in \llbracket 1, N\rrbracket$ and $M: \mathbb R \to \mathcal M_N(\mathbb C)$, define the orthogonal projection $P_i = e_i e_i^{\mathrm{T}}$ and set $A_i(\cdot) = M(\cdot) P_i$. Consider the system of difference equations
\begin{equation}
\label{EquivDelay}
v(t) = \sum_{j=1}^N A_j(t) v(t - L_j).
\end{equation}
This system is equivalent to \eqref{MainSystTransport} in the following sense.

\begin{proposition}
\label{PropTranspDelay}
Suppose that $(u_i)_{i \in \llbracket 1, N\rrbracket}$ is a solution of \eqref{MainSystTransport} with initial condition $(u_{i, 0})_{i \in \llbracket 1, N\rrbracket}$ and let $v: \left[-L_{\max}, +\infty\right) \to \mathbb C^N$ be given for $i \in \llbracket 1, N\rrbracket$ by
\begin{equation}
\label{vFuncOfU}
v_i(t) = 
\left\{
\begin{aligned}
& 0, & & \text{ if } t \in \left[-L_{\max}, -L_i\right), \\
& u_{i, 0}(-t), & & \text{ if } t \in \left[-L_i, 0\right), \\
& u_i(t, 0), & & \text{ if } t \geq 0.
\end{aligned}
\right.
\end{equation}
Then $v$ is a solution of \eqref{EquivDelay}.

Conversely, suppose that $v: \left[-L_{\max}, +\infty\right) \to \mathbb C^N$ is a solution of \eqref{EquivDelay} and let $(u_i)_{i \in \llbracket 1, N\rrbracket}$ be given for $i \in \llbracket 1, N\rrbracket$, $t \geq 0$ and $x \in [0, L_i]$ by $u_i(t, x) = v_i(t - x)$. Then $(u_i)_{i \in \llbracket 1, N\rrbracket}$ is a solution of \eqref{MainSystTransport}.
\end{proposition}

\begin{proof}
Let $(u_i)_{i \in \llbracket 1, N\rrbracket}$ be a solution of \eqref{MainSystTransport} with initial condition $(u_{i, 0})_{i \in \llbracket 1, N\rrbracket}$ and let $v: \left[-L_{\max}, +\infty\right) \allowbreak\to \mathbb C^N$ be given by \eqref{vFuncOfU}. Then, for $t \geq 0$,
\[
v_i(t) = u_i(t, 0) = \sum_{j=1}^N m_{ij}(t) u_j(t, L_j),
\]
and, by Definition \ref{DefiSolTransport}, $u_j(t, L_j) = v_j(t - L_j)$ since $u_j(t, L_j) = u_j(t - L_j, 0)$ if $t \geq L_j$ and $u_j(t, L_j) = u_{j, 0}(L_j - t)$ if $0 \leq t < L_j$. Hence $v_i(t) = \sum_{j=1}^N m_{ij}(t) v_j(t - L_j)$ and thus $v(t) = \sum_{j=1}^N A_j(t) v(t - L_j)$.

Conversely, suppose that $v: \left[-L_{\max}, +\infty\right) \to \mathbb C^N$ is a solution of \eqref{EquivDelay} with initial condition $v_0$ and let $(u_i)_{i \in \llbracket 1, N\rrbracket}$ be given for $i \in \llbracket 1, N\rrbracket$, $t \geq 0$ and $x \in [0, L_i]$ by $u_i(t, x) = v_i(t - x)$. It is then clear that $u_i(t + s, x + s) = u_i(t, x)$ for $s \in [-\min(x, t), L_i - x]$, and, since $v_i(t) = \sum_{j=1}^N m_{ij}(t) v_j(t - L_j)$,
\[
u_i(t, 0) = v_i(t) = \sum_{j=1}^N m_{ij}(t) v_j(t - L_j) = \sum_{j=1}^N m_{ij}(t) u_j(t, L_j),
\]
and so $(u_i)_{i \in \llbracket 1, N\rrbracket}$ is a solution of \eqref{MainSystTransport}.
\end{proof}

The following result follows immediately from Proposition \ref{PropExistUnique}.

\begin{proposition}
\label{PropTranspWellPosed}
Let $u_{i, 0}: [0, L_i] \to \mathbb C$ for $i \in \llbracket 1, N\rrbracket$ and $M: \mathbb R \to \mathcal M_N(\mathbb C)$. Then $\Sigma_{\tau}(L, M)$ admits a unique solution $(u_i)_{i \in \llbracket 1, N\rrbracket}$, $u_i: \mathbb R_+ \times [0, L_i] \to \mathbb C$ for $i \in \llbracket 1, N\rrbracket$, with initial condition $(u_{i, 0})_{i \in \llbracket 1, N\rrbracket}$.
\end{proposition}

\subsection{Invariant subspaces}

For $p \in [1, +\infty]$, consider \eqref{MainSystTransport} in the Banach space
\[\mathsf X_p^\tau = \prod_{i=1}^N L^p([0, L_i], \mathbb C)\]
endowed with the norm
\[
\norm{u}_{p} =
\left\{
\begin{aligned}
& \left(\sum_{i=1}^N \norm{u_i}_{L^p([0, L_i], \mathbb C)}^p\right)^{1/p}, & & \text{ if } p \in \left[1, +\infty\right), \\
& \max_{i \in \llbracket 1, N\rrbracket} \norm{u_i}_{L^\infty([0, L_i], \mathbb C)}, & & \text{ if } p = +\infty.
\end{aligned}
\right.
\]
It follows from Proposition \ref{PropTranspDelay} and Remark \ref{RemkRegular} that, if $M \in L^\infty_{\mathrm{loc}}(\mathbb R, \mathcal M_N(\mathbb C))$ and $u_0 \in \mathsf X_p^\tau$, then the solution $t \mapsto u(t)$ of $\Sigma_\tau(L, M)$ with initial condition $u_0$ takes values in $\mathsf X_p^\tau$ for every $t \geq 0$.


In Section \ref{SecWave}, we study wave propagation on networks using transport equations via the d'Alembert decomposition. For that purpose, we need to study transport equations in the range of the d'Alembert decomposition operator, which happens to take the following form (see Proposition \ref{PropTOnto}). For $r \in \mathbb N$ and $R \in \mathcal M_{r, N}(\mathbb C)$ with coefficients $\rho_{ij}$, $i \in \llbracket 1, r\rrbracket$, $j \in \llbracket 1, N\rrbracket$, let
\begin{equation*}
\mathsf{Y}_p(R) = \left\{u = (u_1, \dotsc, u_N) \in \mathsf X_p^\tau \;\middle|\: \forall i \in \llbracket 1, r\rrbracket,\; \sum_{j = 1}^N \rho_{ij} \int_{0}^{L_j} u_j(x) dx = 0\right\}.
\end{equation*}
This is a closed subspace of $\mathsf X_p^\tau$, which is thus itself a Banach space.

\begin{remark}
Let $r \in \mathbb N$, $R \in \mathcal M_{r, N}(\mathbb C)$, and $M \in L^\infty_{\mathrm{loc}}(\mathbb R, \mathcal M_N(\mathbb C))$. Note that, if $1 \leq p \leq q \leq +\infty$, $\mathsf Y_q(R)$ is a dense subset of $\mathsf Y_p(R)$ since $\mathsf X_q^\tau$ is a dense subset of $\mathsf X_p^\tau$. As a consequence, by a density argument, Propositions \ref{PropSolExpliciteAdpBis} and \ref{PropTranspDelay}, one obtains that, if $\mathsf Y_p(R)$ is invariant under the flow of $\Sigma_\tau(L, M)$ for some $p \in [1, +\infty]$, then $\mathsf Y_q(R)$ is invariant for every $q \in [1, +\infty]$.
\end{remark}

The following proposition provides a necessary and sufficient condition for $\mathsf Y_p(R)$ to be invariant under the flow of \eqref{MainSystTransport}.

\begin{proposition}
\label{PropInvariantIff}
Let $r \in \mathbb N$, $R \in \mathcal M_{r, N}(\mathbb C)$, $(u_{i, 0})_{i \in \llbracket 1, N\rrbracket} \in \mathsf Y_p(R)$, and $M \in L^\infty_{\mathrm{loc}}(\mathbb R, \mathcal M_N(\mathbb C))$. Then the solution $u = (u_i)_{i \in \llbracket 1, N\rrbracket}$ of $\Sigma_\tau(L, M)$ with initial condition $(u_{i, 0})_{i \in \llbracket 1, N\rrbracket}$ belongs to $\mathsf Y_p(R)$ for every $t \geq 0$ if and only if
\[
R (M(t) - \id_N) w(t) = 0
\]
for almost every $t \geq 0$, where $w = (w_i)_{i \in \llbracket 1, N\rrbracket}$ and $w_i(t) = u_i(t, L_i)$.
\end{proposition}

\begin{proof}
Let $v: \left[-L_{\max}, +\infty\right) \to \mathbb C^N$ be the solution of \eqref{EquivDelay} corresponding to $u$, given by \eqref{vFuncOfU}, and let $w = (w_i)_{i \in \llbracket 1, N\rrbracket}$ be defined by $w_i(t) = v_i(t - L_i) = u_i(t, L_i)$. Let $\lambda = (\lambda_i)_{i \in \llbracket 1, r\rrbracket}$ be given for $i \in \llbracket 1, r\rrbracket$ by $\lambda_i(t) = \sum_{j = 1}^N \rho_{ij} \int_0^{L_j} u_j(t, x) dx$. Since $\lambda_i(0) = 0$, we have
\begin{align*}
\lambda_i(t) & = \sum_{j = 1}^N \rho_{ij} \left[\int_0^{L_j} u_j(t, x) dx - \int_0^{L_j} u_{j, 0}(x) dx\right] \displaybreak[0] \\
 & = \sum_{j = 1}^N \rho_{ij} \left[\int_0^{L_j} v_j(t - x) dx - \int_0^{L_j} v_j(-x) dx\right] \displaybreak[0] \\
 & = \sum_{j = 1}^N \rho_{ij} \left[\int_{t - L_j}^{t} v_j(s) ds - \int_0^{L_j} v_j(s - L_j) ds\right] \displaybreak[0] \\
 & = \sum_{j = 1}^N \rho_{ij} \int_{0}^{t} \left(v_j(s) - v_j(s - L_j)\right) ds \displaybreak[0] \\
 & = \sum_{j = 1}^N \rho_{ij} \int_{0}^{t} \left(\sum_{k=1}^N m_{jk}(s) v_k(s - L_k) - v_j(s - L_j)\right) ds \displaybreak[0] \\
 & = \sum_{j = 1}^N \rho_{ij} \int_{0}^{t} \sum_{k=1}^N \left(m_{jk}(s) - \delta_{jk}\right) v_k(s - L_k) ds,
\end{align*}
so that $\lambda(t) = \int_0^t R (M(s) - \id_N) w(s) ds$. The conclusion of the proposition follows immediately.
\end{proof}

\begin{definition}
Let $L \in (\mathbb R_+^\ast)^N$ and $\mathcal M$ be a subset of $L^\infty_{\mathrm{loc}}(\mathbb R, \mathcal M_N(\mathbb C))$. We denote by $\mathrm{Inv}(\mathcal M)$ the set
\[
\begin{split}
\mathrm{Inv}(\mathcal M) = \{ & R \in \mathcal M_{r, N}(\mathbb C) \;|\: r \in \mathbb N,\; \mathsf Y_p(R) \text{ is invariant under} \\
 & \text{the flow of } \Sigma_\tau(L, M),\; \forall M \in \mathcal M,\; \forall p \in [1, +\infty]\}.
\end{split}
\]
\end{definition}

\subsection{Stability of solutions on invariant subspaces}

We next provide a definition for exponential stability of \eqref{MainSystTransport}.

\begin{definition}
Let $p \in [1, +\infty]$, $L \in (\mathbb R_+^\ast)^N$, $\mathcal M$ be a uniformly locally bounded subset of $L^\infty_{\mathrm{loc}}(\mathbb R, \allowbreak\mathcal M_N(\mathbb C))$, and $R \in \mathrm{Inv}(\mathcal M)$. Let $\Sigma_\tau(L, \mathcal M)$ denote the family of systems $\Sigma_\tau(L, M)$ for $M \in \mathcal M$. We say that $\Sigma_\tau(L, \mathcal M)$ is of \emph{exponential type} $\gamma$ in $\mathsf Y_p(R)$ if, for every $\varepsilon > 0$, there exists $K > 0$ such that, for every $M \in \mathcal M$ and $u_0 \in \mathsf Y_p(R)$, the corresponding solution $u$ of $\Sigma_\tau(L, M)$ satisfies, for every $t \geq 0$,
\[
\norm{u(t)}_{p} \leq K e^{(\gamma + \varepsilon) t} \norm{u_0}_{p}.
\]
We say that $\Sigma_\tau(L, \mathcal M)$ is \emph{exponentially stable} in $\mathsf Y_p(R)$ if it is of negative exponential type.
\end{definition}

The next corollaries translate Propositions \ref{PropConvCoeffConvSolGeneral} and \ref{PropNConvCoeffNConvSolGeneral} into the framework of transport equations.

\begin{corollary}
Let $\Lambda \in (\mathbb R_+^\ast)^N$, $L \in V_+(\Lambda)$, and $\mathcal M$ be a uniformly locally bounded subset of $L^\infty_{\mathrm{loc}}(\mathbb R, \allowbreak\mathcal M_N(\mathbb C))$. Suppose that there exists a continuous function $f: \mathbb R \to \mathbb R_+^\ast$ such that, for every $M \in \mathcal M$, $\mathbf n \in \mathbb N^N$, and almost every $t \in (L \cdot \mathbf n - L_{\max}, L \cdot \mathbf n)$, \eqref{HypoConvCoeff} holds with $A = (A_1, \dotsc, A_N)$ given by $A_i = M P_i$. Then there exists a constant $C > 0$ such that, for every $M \in \mathcal M$, $p \in \left[1, +\infty\right]$, and $u_0 \in \mathsf X_p^\tau$, the corresponding solution $u$ of $\Sigma_\tau(L, M)$ satisfies
\begin{equation*}
\norm{u(t)}_{p} \leq C (t + 1)^{N-1} \max_{s \in [t - L_{\max}, t]} f(s) \norm{u_0}_{p}, \qquad \forall t \geq 0.
\end{equation*}
\end{corollary}

\begin{proof}
Let $C > 0$ be as in the Proposition \ref{PropConvCoeffConvSolGeneral}. Let $M \in \mathcal M$, $p \in \left[1, +\infty\right]$, $u_0 \in \mathsf X_p^\tau$, and $u$ be the solution of $\Sigma_\tau(L, M)$ with initial condition $u_0$. Let $v$ be the corresponding solution of \eqref{EquivDelay}, given by \eqref{vFuncOfU}, with initial condition $v_0$. Notice that $\norm{u_0}_{p} = \norm{v_0}_{p}$ and, for every $t \geq 0$, $\norm{u(t)}_{p} \leq \norm{v_t}_{p}$. By Proposition \ref{PropConvCoeffConvSolGeneral}, we have, for every $t \geq 0$,
\[
\begin{split}
\norm{u(t)}_{p} & \leq \norm{v_t}_{p} \leq C (t + 1)^{N-1} \max_{s \in [t - L_{\max}, t]} f(s) \norm{v_0}_{p} \\
 & = C (t + 1)^{N-1} \max_{s \in [t - L_{\max}, t]} f(s) \norm{u_0}_{p},
\end{split}
\]
which is the desired result.
\end{proof}

\begin{corollary}
Let $\Lambda \in (\mathbb R_+^\ast)^N$, $L \in W_+(\Lambda)$, $\mathcal M$ be a uniformly locally bounded subset of $L^\infty_{\mathrm{loc}}(\mathbb R, \allowbreak\mathcal M_N(\mathbb C))$, and $f: \mathbb R \to \mathbb R_+^\ast$ be a continuous function. Suppose that there exist $M \in \mathcal M$, $\mathbf n_0 \in \mathbb N^N$, and a set of positive measure $S \subset \left(L \cdot \mathbf n_0 - L_{\max}, L \cdot \mathbf n_0\right)$ such that, for every $t \in S$, \eqref{HypoNConvCoeff} is satisfied with $A = (A_1, \dotsc,\allowbreak A_N)$ given by $A_i = M P_i$. Then there exist a constant $C > 0$ independent of $f$, an initial condition $u_0 \in \mathsf X_\infty^\tau$, and $t > 0$ such that, for every $p \in [1, +\infty]$ and $R \in \mathrm{Inv}(\mathcal M)$, the solution $u$ of $\Sigma_\tau(L, M)$ with initial condition $u_0$ satisfies $u(s) \in \mathsf Y_p(R)$ for every $s \geq 0$ and
\begin{equation*}
\norm{u(t)}_{p} > C \min_{s \in [t - L_{\max}, t]} f(s) \norm{u_0}_{p}.
\end{equation*}
\end{corollary}

\begin{proof}
As in Proposition \ref{PropNConvCoeffNConvSolGeneral}, since $L \in W_+(\Lambda)$, we can assume for the rest of the argument that $\Lambda = L$.

Let $C > 0$ be as in Proposition \ref{PropNConvCoeffNConvSolGeneral}. We construct an initial condition $v_0 \in \mathsf X_p^\delta$ as follows: choose $t_0$ and $j_0$ as in Proposition \ref{PropNConvCoeffNConvSolGeneral} and verifying in addition $t_0 \not = L \cdot \mathbf n_0 - L_{j_0}$. Then pick $\delta > 0$ as in Proposition \ref{PropNConvCoeffNConvSolGeneral} and satisfying in addition $\delta < \abs{t_0 - L \cdot \mathbf n_0 + L_{j_0}}$ and $\delta < L_{\min}/2$. Next, take $\mu \in L^\infty(\mathbb R, \mathbb R)$ as in Proposition \ref{PropNConvCoeffNConvSolGeneral} and satisfying in addition $\int_{-\delta}^\delta \mu(s) ds = 0$. Finally, consider the initial condition $v_0(s) = \mu(s - t_0 + L \cdot \mathbf n_0) e_{j_0}$. As in \eqref{EstimNConv}, we still obtain that the solution $v$ of \eqref{EquivDelay} with initial condition $v_0$ satisfies, for $p \in [1, +\infty]$,
\begin{equation}
\label{DivV}
\norm{v_{t_0 + \delta}}_{p} \geq \norm{v_{t_0}}_{L^p([-\delta, \delta], \mathbb C^N)} > C \min_{s \in [t_0 + \delta - L_{\max}, t_0 + \delta]} f(s) \norm{v_0}_{p}.
\end{equation}

Let $u$ be the solution of \eqref{MainSystTransport} corresponding to $v$, in the sense of Proposition \ref{PropTranspDelay}, and denote by $u_0 = (u_{i, 0})_{i \in \llbracket 1, N\rrbracket}$ its initial condition. Since $u_{i, 0}(x) = v_i(-x)$, we have $u_0 \in \prod_{i=1}^N L^\infty([0, L_i], \mathbb C)$. Furthermore, $u_{i, 0} = 0$ for $i \not = j_0$ and $u_{j_0, 0}(x) = v_{j_0}(-x) = \mu(L \cdot \mathbf n_0 - t_0 - x)$. By definition of $\delta$, we must have either $(L \cdot \mathbf n_0 - t_0 - \delta, L \cdot \mathbf n_0 - t_0 + \delta) \subset [0, L_{j_0}]$ or $(L \cdot \mathbf n_0 - t_0 - \delta, L \cdot \mathbf n_0 - t_0 + \delta) \cap [0, L_{j_0}] = \emptyset$, but the latter case is impossible since we would then have $u_{j_0, 0} = 0$, and thus $v(s) = 0$ for every $s \geq -L_{\max}$, which contradicts \eqref{DivV}. Hence $(L \cdot \mathbf n_0 - t_0 - \delta, L \cdot \mathbf n_0 - t_0 + \delta) \subset [0, L_{j_0}]$ and
\[
\int_{0}^{L_{j_0}} u_{j_0, 0}(x) dx = \int_{-\delta}^{\delta} \mu(x) dx = 0.
\]
We thus have clearly $u_0 \in \mathsf Y_\infty(R)$, and in particular $u(s) \in \mathsf Y_p(R)$ for every $s \geq 0$ and $p \in [1, +\infty]$. Furthermore, $\norm{v_0}_{p} = \norm{u_0}_{p}$ and, for $p \in \left[1, +\infty\right)$,
\[
\begin{split}
\norm{v_{t_0}}_{L^p([-\delta, \delta], \mathbb C^N)}^p & = \int_{-\delta}^{\delta} \abs{v(t_0 + s)}_p^p ds = \int_{-\delta}^\delta \sum_{i=1}^N \abs{u_i(t_0 + s, 0)}^p ds \\
 & = \int_0^{2\delta} \sum_{i=1}^N \abs{u_i(t_0 + \delta, s)}^p ds \leq \sum_{i=1}^N \int_0^{L_i} \abs{u_i(t_0 + \delta, s)}^p ds \\
 & = \norm{u(t_0 + \delta)}_{p}^p,
\end{split}
\]
with a similar estimate for $p = +\infty$. Hence, it follows from \eqref{DivV} that, for every $p \in [1, +\infty]$,
\[\norm{u(t)}_{p} > C \min_{s \in [t - L_{\max}, t]} f(s) \norm{u_0}_{p}\]
with $t = t_0 + \delta$.
\end{proof}

As a consequence of the previous analysis, we have the following result.

\begin{theorem}
\label{TheoTranspTheta}
Let $\mathcal M$ be a uniformly locally bounded subset of $L^\infty_{\mathrm{loc}}(\mathbb R, \mathcal M_N(\mathbb C))$, $\Lambda \in (\mathbb R_+^\ast)^N$, and $\mathcal A = \{A = (A_1, \dotsc, A_N): \mathbb R \to \mathcal M_N(\mathbb C)^N \;|\: A_i = M P_i, M \in \mathcal M\}$. For every $L \in V_+(\Lambda)$, if $\Sigma_\delta(L, \mathcal A)$ is of $(\Theta, \Lambda)$-exponential type $\gamma$ then, for every $p \in [1, +\infty]$ and $R \in \mathrm{Inv}(\mathcal M)$, $\Sigma_\tau(L, \mathcal M)$ is of exponential type $\gamma$ in $\mathsf Y_p(R)$. Conversely, for every $L \in W_+(\Lambda)$, if there exist $p \in [1, +\infty]$ and $R \in \mathrm{Inv}(\mathcal M)$ such that $\Sigma_\tau(L, \mathcal M)$ is of exponential type $\gamma$ in $\mathsf Y_p(R)$, then $\Sigma_\delta(L, \mathcal A)$ is of $(\Theta, \Lambda)$-exponential type $\gamma$.
\end{theorem}

It follows from Theorem \ref{TheoTranspTheta} that the exponential type $\gamma$ for $\Sigma_\tau(L, \mathcal M)$ in $\mathsf Y_p(R)$ is independent of $p \in [1, +\infty]$ and $R \in \mathrm{Inv}(\mathcal M)$. When $\mathcal M$ is shift-invariant, thanks to Theorem \ref{TheoStabShiftInv}, one can replace $(\Theta, \Lambda)$-exponential type by $(\widehat\Xi, \Lambda)$-exponential type for $\Sigma_\delta(L, \mathcal A)$ in Theorem \ref{TheoTranspTheta}.

Assume now that $\mathcal M = L^\infty(\mathbb R, \mathfrak B)$, where $\mathfrak B$ is a bounded subset of $\mathcal M_N(\mathbb C)$. Let $\mathcal A = \{A = (A_1, \dotsc, A_N): \mathbb R \to \mathcal M_N(\mathbb C)^N \;|\: A_i = M P_i,\; M \in \mathcal M\}$, i.e., $\mathcal A = L^\infty(\mathbb R, \mathfrak A)$ where $\mathfrak A = \{A = (A_1, \dotsc, A_N) \in \mathcal M_N(\mathbb C)^N \;|\: A_i = M P_i,\; M \in \mathfrak B\}$. We can thus transpose the results from Section \ref{SecSwitched}, and in particular Corollary \ref{CoroSilkRho0}, to the transport framework.

\begin{corollary}
\label{CoroSilkTransport}
Let $\Lambda \in (\mathbb R_+^\ast)^N$, $\mathfrak B$ be a nonempty bounded subset of $\mathcal M_N(\mathbb C)$, $\mathcal M = L^\infty(\mathbb R, \mathfrak B)$. The following statements are equivalent.
\begin{enumerate}[label={\bf \roman*.}, ref={(\roman*)}]
\item $\Sigma_\tau(\Lambda, \mathcal M)$ is exponentially stable in $\mathsf Y_p(R)$ for some $p \in [1, +\infty]$ and $R \in \mathrm{Inv}(\mathcal M)$.
\item $\Sigma_\tau(L, \mathcal M)$ is exponentially stable in $\mathsf Y_p(R)$ for every $L \in V_+(\Lambda)$, $p \in [1, +\infty]$, and $R \in \mathrm{Inv}(\mathcal M)$.
\end{enumerate}
\end{corollary}

\begin{remark}
\label{RemkBClosed}
In accordance with Remark \ref{RemkAClosed}, the exponential stability of the system $\Sigma_\tau(\Lambda, \mathcal M)$ is equivalent to that of $\Sigma_\tau(\Lambda, L^\infty(\mathbb R, \overline{\mathfrak B}))$.
\end{remark}


\section{Wave propagation on networks}
\label{SecWave}

We consider here the problem of wave propagation on a finite network of elastic strings. The notations we use here come from \cite{Dager2006Wave}.

A \emph{graph} $\mathcal G$ is a pair $(\mathcal V, \mathcal E)$, where $\mathcal V$ is a set, whose elements are called \emph{vertices}, and
\[\mathcal E \subset \{\{q, p\} \;|\: q, p \in \mathcal V,\; q \not = p\}.\]
The elements of $\mathcal E$ are called \emph{edges}, and, for $e = \{q, p\} \in \mathcal E$, the vertices $q, p$ are called the \emph{endpoints} of $\mathcal E$. An \emph{orientation} on $\mathcal G$ is defined by two maps $\alpha, \omega: \mathcal E \to \mathcal V$ such that, for every $e \in \mathcal E$, $e = \{\alpha(e), \omega(e)\}$. Given $q, p \in \mathcal V$, a \emph{path} from $q$ to $p$ is a $n$-tuple $(q = q_1, \dotsc, q_n = p) \in \mathcal V^n$ where, for every $j \in \llbracket 1, n-1\rrbracket$, $\{q_j, q_{j+1}\} \in \mathcal E$. The positive integer $n$ is called the \emph{length} of the path. A path of length $n$ in $\mathcal G$ is said to be \emph{closed} if $q_1 = q_n$; \emph{simple} if all the edges $\{q_j, q_{j+1}\}$, $j \in \llbracket 1, n-1\rrbracket$, are different; and \emph{elementary} if the vertices $q_1, \dotsc, q_n$ are pairwise different, except possibly for the pair $(q_1, q_n)$. An elementary closed path is called a \emph{cycle}. A graph which does not admit cycles is called a \emph{tree}. We say that a graph $\mathcal G$ is \emph{connected} if, for every $q, p \in \mathcal V$, there exists a path from $q$ to $p$. We say that $\mathcal G$ is \emph{finite} if $\mathcal V$ is a finite set. For every $q \in \mathcal V$, we denote by $\mathcal E_q$ the set of edges for which $q$ is an endpoint, that is,
\[
\mathcal E_q = \{e \in \mathcal E \;|\: q \in e\}.
\]
The cardinality of $\mathcal E_q$ is denoted by $n_q$. We say that $q \in \mathcal V$ is \emph{exterior} if $\mathcal E_q$ contains at most one element and \emph{interior} otherwise. We denote by $\mathcal V_{\mathrm{ext}}$ and $\mathcal V_{\mathrm{int}}$ the set of exterior and interior vertices, respectively. We suppose in the sequel that $\mathcal V_{\mathrm{ext}}$ contains at least two elements, and we fix a nonempty subset $\mathcal V_{\mathrm d}$ of $\mathcal V_{\mathrm{ext}}$ such that $\mathcal V_{\mathrm u} = \mathcal V_{\mathrm{ext}} \setminus \mathcal V_{\mathrm d}$ is nonempty. The vertices of $\mathcal V_{\mathrm d}$ are said to be \emph{damped}, and the vertices of $\mathcal V_{\mathrm u}$ are said to be \emph{undamped}. Note that $\mathcal V$ is the disjoint union $\mathcal V = \mathcal V_{\mathrm{int}} \cup \mathcal V_{\mathrm u} \cup \mathcal V_{\mathrm d}$.

A \emph{network} is a pair $(\mathcal G, L)$ where $\mathcal G = (\mathcal V, \mathcal E)$ is an oriented graph and $L = (L_e)_{e \in \mathcal E}$ is a vector of positive real numbers, where each $L_e$ is called the \emph{length} of the edge $e$. We say that a network is \emph{finite} (respectively, \emph{connected}) if its underlying graph $\mathcal G$ is finite (respectively, connected). If $e \in \mathcal E$ and $u: [0, L_e] \to \mathbb C$ is a function, we write $u(\alpha(e)) = u(0)$ and $u(\omega(e)) = u(L_e)$. For every elementary path $(q_1, \dotsc, q_n)$, we define its \emph{signature} $s: \mathcal E \to \{-1, 0, 1\}$ by
\[
s(e) = 
\left\{
\begin{aligned}
& 1, & & \text{ if } e = \{q_i, q_{i+1}\} \text{ for some } i \in \llbracket 1, n-1\rrbracket \text{ and } \alpha(e) = q_i, \\
& -1, & & \text{ if } e = \{q_i, q_{i+1}\} \text{ for some } i \in \llbracket 1, n-1\rrbracket \text{ and } \alpha(e) = q_{i+1}, \\
& 0, & & \text{ otherwise}.
\end{aligned}
\right.
\]
The \emph{normal derivatives} of $u$ at $\alpha(e)$ and $\omega(e)$ are defined by $\frac{d u}{d n_e}(\alpha(e)) = - \frac{d u}{d x}(0)$ and $\frac{d u}{d n_e}(\omega(e)) = \frac{d u}{d x}(L_e)$.

In what follows, we consider only finite connected networks. In order to simplify the notations, we identify $\mathcal E$ with the finite set $\llbracket 1, N\rrbracket$, where $N = \# \mathcal E$. We model wave propagation along the edges of a finite connected network $(\mathcal G, L)$ by $N$ displacement functions $u_j: \left[0, +\infty\right) \times [0, L_j] \to \mathbb C$, $j \in \llbracket 1, N\rrbracket$, satisfying
\begin{multline}
\label{MainSystWave}
\Sigma_{\omega}(\mathcal G, L, \eta): \\
\left\{
\begin{aligned}
\frac{\partial^2 u_j}{\partial t^2}(t, x) & = \frac{\partial^2 u_j}{\partial x^2}(t, x), & & j \in \llbracket 1, N\rrbracket,\; t \in \left[0, +\infty\right),\; x \in [0, L_j], \\
u_j(t, q) & = u_k(t, q), & & q \in \mathcal V,\; j, k \in \mathcal E_q,\; t \in \left[0, +\infty\right), \\
\sum_{j \in \mathcal E_q} \frac{\partial u_j}{\partial n_j}(t, q) & = 0, & & q \in \mathcal V_{\mathrm{int}},\; t \in \left[0, +\infty\right), \\
\frac{\partial u_j}{\partial t}(t, q) & = - \eta_q(t) \frac{\partial u_j}{\partial n_j}(t, q), & & q \in \mathcal V_{\mathrm{d}},\; j \in \mathcal E_q,\; t \in \left[0, +\infty\right), \\
u_j(t, q) & = 0, & & q \in \mathcal V_{\mathrm{u}},\; j \in \mathcal E_q,\; t \in \left[0, +\infty\right).
\end{aligned}
\right.
\end{multline}
Each function $\eta_q$ is assumed to be nonnegative and determines the damping at the vertex $q \in \mathcal V_{\mathrm d}$. We denote by $\eta$ the vector-valued function $\eta = (\eta_q)_{q \in \mathcal V_{\mathrm d}}$.

\begin{remark}
Let $(\mathcal G, L)$ be a finite connected network with $\mathcal E = \llbracket 1, N\rrbracket$ and $(\alpha_1, \omega_1)$, $(\alpha_2, \omega_2)$ be two orientations of $\mathcal G$. If $(u_j)_{j \in \llbracket 1, N\rrbracket}$ satisfies \eqref{MainSystWave} with orientation $(\alpha_1, \omega_1)$ and $(v_j)_{j \in \llbracket 1, N\rrbracket}$ is given by $v_j = u_j$ if $\alpha_1(j) = \alpha_2(j)$ and $v_j(x) = u_j(L_j - x)$ otherwise, we can easily verify that $(v_j)_{j \in \llbracket 1, N\rrbracket}$ satisfies \eqref{MainSystWave} with orientation $(\alpha_2, \omega_2)$. Hence the dynamical properties of \eqref{MainSystWave} do not depend on the orientation of $\mathcal G$.
\end{remark}

For $p \in [1, +\infty]$, consider the Banach spaces $L^p(\mathcal G, L) = \prod_{j=1}^N L^p([0, L_j], \mathbb C)$ and
\begin{multline*}
W^{1, p}_0(\mathcal G, L) = \left\{(u_1, \dotsc, u_N) \in \prod_{j=1}^N W^{1, p}([0, L_j], \mathbb C) \;\middle|\: \right. \\ \left. u_j(q) = u_k(q),\; \forall q \in \mathcal V,\; \forall j, k \in \mathcal E_q; u_j(q) = 0,\; \forall q \in \mathcal V_{\mathrm u},\; \forall j \in \mathcal E_q\right\},
\end{multline*}
endowed with the usual norms
\[
\begin{aligned}
\norm{u}_{L^p(\mathcal G, L)} & =
\left\{
\begin{aligned}
& \left(\sum_{i=1}^N \norm{u_i}_{L^p([0, L_i], \mathbb C)}^p\right)^{\frac{1}{p}}, & & \text{ if } p \in \left[1, +\infty\right), \\
& \max_{i \in \llbracket 1, N\rrbracket} \norm{u_i}_{L^\infty([0, L_i], \mathbb C)}, & & \text{ if } p = +\infty,
\end{aligned}
\right. \\
\norm{u}_{W^{1, p}_0(\mathcal G, L)} & =
\left\{
\begin{aligned}
& \left(\sum_{i=1}^N \norm{u_i^\prime}_{L^p([0, L_i], \mathbb C)}^p\right)^{\frac{1}{p}}, & & \text{ if } p \in \left[1, +\infty\right), \\
& \max_{i \in \llbracket 1, N\rrbracket} \norm{u_i^\prime}_{L^\infty([0, L_i], \mathbb C)}, & & \text{ if } p = +\infty.
\end{aligned}
\right.
\end{aligned}
\]
We will omit $(\mathcal G, L)$ from the notations when it is clear from the context.

Let $\mathsf X_p^\omega = W^{1, p}_0 \times L^p$, endowed with the norm $\norm{\cdot}_{p}$ defined by
\[
\norm{(u, v)}_{p} = 
\left\{
\begin{aligned}
& \left(\norm{u}_{W^{1, p}_0(\mathcal G, L)}^p + \norm{v}_{L^p(\mathcal G, L)}^p\right)^{\frac{1}{p}}, & & \text{ if } p \in \left[1, +\infty\right), \\
& \max\left(\norm{u}_{W^{1, \infty}_0(\mathcal G, L)}, \norm{v}_{L^\infty(\mathcal G, L)}\right), & & \text{ if } p = +\infty,
\end{aligned}
\right.
\]
and, for every $t \in \mathbb R$, define the operator $A(t)$ by
\[
\begin{split}
D(A(t)) = & \left\{(u, v) \in \left(W^{1, p}_0 \cap \prod_{j=1}^N W^{2, p}([0, L_j], \mathbb C)\right) \times W^{1, p}_0 \;\middle|\: \right. \\
 & \hphantom{ \{ } \left. v_j(q) = - \eta_q(t) \frac{d u_j}{d n_j}(q),\; \forall q \in \mathcal V_{\mathrm d},\; \forall j \in \mathcal E_q;\; \sum_{j \in \mathcal E_q} \frac{d u_j}{d n_j}(q) = 0,\; \forall q \in \mathcal V_{\mathrm{int}}\right\},
\end{split}
\]
\[
A(t)
\begin{pmatrix}
u \\
v \\
\end{pmatrix} = 
\begin{pmatrix}
v \\
u^{\prime\prime}
\end{pmatrix}.
\]
One can then write \eqref{MainSystWave} as an evolution equation in $\mathsf X_p^\omega$ as
\begin{equation}
\label{WaveFunctional}
\dot U(t) = A(t) U(t)
\end{equation}
where $U = \left(u, \pad{u}{t}\right)$.

\subsection{Equivalence with a system of transport equations}

In order to make a connection with transport systems, we consider, for $p \in [1, +\infty]$, the Banach space
\[
\mathsf X_p^\tau = \prod_{j=1}^{2N} L^p([0, L_j^\tau], \mathbb C),
\]
where $L_{2j-1}^\tau = L_{2j}^\tau = L_j$ for $j \in \llbracket 1, N\rrbracket$.

\begin{definition}[D'Alembert decomposition operator]
Let $T: \mathsf X_p^\omega \to \mathsf X_p^\tau$ be the operator given by $T(u, v) = f$, where, for $j \in \llbracket 1, N\rrbracket$, $x \in [0, L_j]$,
\begin{equation}
\label{DefT}
f_{2j - 1}(x) = u_j^\prime(L_j - x) + v_j(L_j - x), \qquad f_{2j}(x) = u_j^\prime(x) - v_j(x).
\end{equation}
\end{definition}

In order to describe the range of $T$, we introduce the following notations. Let $r \in \mathbb N$ be the number of elementary paths $(q_1, \dotsc, q_n)$ in $\mathcal G$ with $q_1 = q_n$ or $q_1, q_n \in \mathcal V_u$. The set of such paths will be indexed by $\llbracket 1, r\rrbracket$. We denote by $s_i$ the signature of the path corresponding to the index $i \in \llbracket 1, r\rrbracket$. We define $R \in \mathcal M_{r, 2N}(\mathbb C)$ by its coefficients $\rho_{ij}$ given by
\[\rho_{i, 2j - 1} = \rho_{i, 2j} = s_i(j) \text{ for }i \in \llbracket 1, r\rrbracket,\; j \in \llbracket 1, N\rrbracket.\]
We then have the following proposition.

\begin{proposition}
\label{PropTOnto}
The operator $T$ is a bijection from $\mathsf X_p^\omega$ to $\mathsf Y_p(R)$. Moreover, $T$ and $T^{-1}$ are continuous.
\end{proposition}

\begin{proof}
Let $(u, v) \in \mathsf X_p^\omega$ and let $f = T (u, v) \in \mathsf X_p^\tau$. Let $(q_1, \dotsc, q_n)$ be an elementary path in $\mathcal G$ with $q_1 = q_n$ or $q_1, q_n \in \mathcal V_{\mathrm u}$ and let $s$ be its signature. For $i \in \llbracket 1, n-1\rrbracket$, let $j_i$ be the index corresponding to the edge $\{q_i, q_{i+1}\}$. We have
\[
\begin{split}
 & \sum_{j=1}^N s(j) \int_0^{L_j} \left(f_{2j-1}(x) + f_{2j}(x)\right) dx \\
{} = {} & 2 \sum_{j=1}^{N} s(j) \int_0^{L_{j}} u_{j}^\prime(x) dx = 2 \sum_{i=1}^{N} s(j) \left(u_{j}(L_{j}) - u_{j}(0)\right) \\
{} = {} & 2 \sum_{i=1}^{n-1} \left(u_{j_i}(q_{i+1}) - u_{j_i}(q_i)\right) = 2 \left(u_{j_{n-1}}(q_n) - u_{j_1}(q_1)\right) = 0,
\end{split}
\]
and thus $f \in \mathsf Y_p(R)$.

Conversely, take $f \in \mathsf Y_p(R)$. For $j \in \llbracket 1, N\rrbracket$, define $v_j: [0, L_j] \to \mathbb C$ by
\begin{equation}
\label{DefVj}
v_j(x) = \frac{f_{2j - 1}(L_j - x) - f_{2j}(x)}{2}.
\end{equation}
One clearly has $v_j \in L^p([0, L_j], \mathbb C)$. We define $u_j$ as follows: let $e \in \mathcal E$ be the edge corresponding to the index $j$. Let $(q_1, \dotsc, q_n)$ be any elementary path with $q_1 \in \mathcal V_{\mathrm u}$ and $q_n = \alpha(e)$. Let $s: \mathcal E \to \{-1, 0, 1\}$ be the signature of that path and, for $i \in \llbracket 1, n-1\rrbracket$, let $j_i$ be the index associated with the edge $\{q_{i}, q_{i+1}\}$. For $x \in [0, L_j]$, set
\begin{equation}
\label{DefUj}
u_j(x) = \sum_{i=1}^{n-1} s(j_i) \int_0^{L_{j_i}} \frac{f_{2j_i - 1}(\xi) + f_{2j_i}(\xi)}{2} d\xi + \int_0^x \frac{f_{2j - 1}(L_j - \xi) + f_{2j}(\xi)}{2} d\xi.
\end{equation}
This definition does not depend on the choice of the path $(q_1, \dotsc, q_n)$ with $q_1 \in \mathcal V_{\mathrm u}$ and $q_n = \alpha(e)$ thanks to the definition of the matrix $R$. It is an immediate consequence of \eqref{DefUj} that $(u, v) \in \mathsf X_p^\omega$. The map $f \mapsto (u, v)$ defines an operator $S: \mathsf Y_p(R) \to \mathsf X_p^\omega$. We clearly have $T \circ S = \id_{\mathsf Y_p(R)}$ and $S \circ T = \id_{\mathsf X_p^\omega}$, and thus $T$ is bijective. The continuity of $T$ and $S$ follows immediately from \eqref{DefT}, \eqref{DefVj}, and \eqref{DefUj}.
\end{proof}

\begin{remark}
\label{RemkTUnit}
When $p = 2$, one easily checks that $\frac{1}{\sqrt{2}} T: \mathsf X_2^\omega \to \mathsf Y_2(R)$ is unitary.
\end{remark}

\begin{remark}
The operator $T$ corresponds to the d'Alembert decomposition of the solutions of the one-dimensional wave equation into a pair of traveling waves moving in opposite directions. For every $j \in \llbracket 1, N\rrbracket$, $f_{2j-1}$ and $f_{2j}$ correspond to the waves moving from $\omega(j)$ to $\alpha(j)$ and from $\alpha(j)$ to $\omega(j)$, respectively (see Figure \ref{uniqueFigureDuPapier}).

\begin{figure}[ht]
\centering

\setlength{\unitlength}{1cm}

\begin{tikzpicture}
\draw[thick] (1, 1) -- (5, 1);
\draw[fill] (1, 1) circle [radius=0.1];
\draw[fill] (5, 1) circle [radius=0.1];
\node[left] at (1, 1) {$\alpha(j)$};
\node[right] at (5, 1) {$\omega(j)$};
\draw[thick, domain=1.5:4.5] plot (\x, {1 + 0.75 - 0.75/4*(\x - 3)^2});
\fill (3, 1 + 0.75) -- ++(-0.125, 0) -- ++(0.25, 0.125) -- ++(0, -0.25) -- ++(-0.25, 0.125);
\draw[thick, domain=1.5:4.5] plot (\x, {1 - 0.75 + 0.75/4*(\x - 3)^2});
\fill (3, 1 - 0.75) -- ++(0.125, 0) -- ++(-0.25, 0.125) -- ++(0, -0.25) -- ++(0.25, 0.125);
\node[above] at (3, 1 + 0.75) {$f_{2j-1}$};
\node[below] at (3, 1 - 0.75) {$f_{2j}$};
\end{tikzpicture}

\caption{D'Alembert decomposition of the wave equation on the edge $j \in \llbracket 1, N\rrbracket$.}
\label{uniqueFigureDuPapier}
\end{figure}
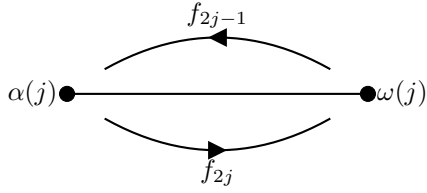
\end{remark}

Let us consider the operator $B(t)$ in $\mathsf Y_p(R)$ defined by conjugation as
\[
\begin{gathered}
D(B(t)) = \{f \in \mathsf Y_p(R) \;|\: T^{-1}f \in D(A(t))\}, \quad
B(t) f = T A(t) T^{-1} f.
\end{gathered}
\]
In order to give a more explicit formula for $B(t)$, we introduce the following notations.

\begin{definition}[Inward and outward decompositions]
The \emph{inward and outward decompositions} of $\mathbb C^{2N}$ are defined respectively as the direct sums
\[
\mathbb C^{2N} = \bigoplus_{q \in \mathcal V} W^q_{\mathrm{in}}, \qquad \mathbb C^{2N} = \bigoplus_{q \in \mathcal V} W^q_{\mathrm{out}},
\]
where, for every $q \in \mathcal V$, we set
\[
\begin{aligned}
W^{q}_{\mathrm{in}} & = \Span\left(\{e_{2j} \;|\: \omega(j) = q\} \cup \{e_{2j-1} \;|\: \alpha(j) = q\}\right), \\
W^{q}_{\mathrm{out}} & = \Span\left(\{e_{2j} \;|\: \alpha(j) = q\} \cup \{e_{2j-1} \;|\: \omega(j) = q\}\right).
\end{aligned}
\]
For every $q \in \mathcal V$, we denote by $\Pi^{q}_{\mathrm{in}}$ and $\Pi^q_{\mathrm{out}}$ the canonical projections of $\mathbb C^{2N}$ onto $W^q_{\mathrm{in}}$ and $W^q_{\mathrm{out}}$, respectively, which we identify with matrices in $\mathcal M_{n_q, 2N}(\mathbb C)$.
\end{definition}

For $n \in \mathbb N$, let $J_n$ denote the $n \times n$ matrix with all elements equal to $1$. Set $D = \diag((-1)^j)_{j \in \llbracket 1, 2N\rrbracket}$. For $q \in \mathcal V$ and $t \in \mathbb R$, we set
\[
M^q(t) = 
\left\{
\begin{aligned}
& \left(\Pi^q_{\mathrm{out}}\right)^{\mathrm{T}} \left(\id_{n_q} - \frac{2}{n_q} J_{n_q}\right) \Pi^q_{\mathrm{in}}, & & \text{ if } q \in \mathcal V_{\mathrm{int}}, \\
& \left(\Pi^q_{\mathrm{out}}\right)^{\mathrm{T}} \Pi^q_{\mathrm{in}}, & & \text{ if } q \in \mathcal V_{\mathrm u}, \\
& \frac{1 - \eta_q(t)}{1 + \eta_q(t)}\left(\Pi^q_{\mathrm{out}}\right)^{\mathrm{T}} \Pi^q_{\mathrm{in}}, & & \text{ if } q \in \mathcal V_{\mathrm d}.
\end{aligned}
\right.
\]
We define the time-dependent matrix $M = (m_{ij})_{i, j \in \llbracket 1, 2N\rrbracket}$ by
\begin{equation}
\label{DefM}
M = - D \left(\sum_{q \in \mathcal V} M^{q}\right) D.
\end{equation}

\begin{remark}
If the components of $\eta$ are nonnegative measurable functions, then $M$ is measurable and its components take values in $[-1, 1]$.
\end{remark}

\begin{remark}
\label{RemkMatM}
Notice that $W^{q_1}_{\mathrm{in}}$ and $W^{q_2}_{\mathrm{in}}$ are orthogonal whenever $q_1 \not = q_2$, and similarly for the outward decomposition. Moreover, for each $q \in \mathcal V$, the spaces $W^q_{\mathrm{in}}$ and $W^q_{\mathrm{out}}$ are invariant under $D$. We finally notice that the image of $M^q(t)$ is contained in $W^q_{\mathrm{out}}$. From these observations, we deduce that, for every $q \in \mathcal V$ and $t \in \mathbb R$,
\[\Pi^q_{\mathrm{out}} D M(t) = -\Pi^q_{\mathrm{out}} M^q(t) D.\]
\end{remark}

We finally obtain the following expression for $B(t)$.

\begin{proposition}
For $t \in \mathbb R$ and $p \in [1, +\infty]$, the operator $B(t)$ is given by
\begin{multline}
\label{DomBtF}
D(B(t)) = \left\{f \in \mathsf Y_p(R) \cap \prod_{i=1}^{2N} W^{1, p}([0, L_i^\tau], \mathbb C) \;\middle|\:\right. \\ \left. f_i(0) = \sum_{j=1}^{2N} m_{ij}(t) f_j(L_j^\tau),\; \forall i \in \llbracket 1, 2N\rrbracket\right\},
\end{multline}
\begin{equation}
\label{BtFeqFPrime}
B(t) f = -f^\prime.
\end{equation}
\end{proposition}

\begin{proof}
Let $f \in \mathsf Y_p(R)$ and $(u, v) = T^{-1} f \in \mathsf X_p^\omega$ and notice that
\begin{equation}
\label{ExplicitTminus1}
u_j^\prime(x) = \frac{f_{2j-1}(L_j - x) + f_{2j}(x)}{2}, \qquad v_j(x) = \frac{f_{2j-1}(L_j - x) - f_{2j}(x)}{2}.
\end{equation}
It follows from \eqref{DefT} and \eqref{ExplicitTminus1} that $f_i \in W^{1, p}([0, L_i^\tau], \mathbb C)$ for every $i \in \llbracket 1, 2N\rrbracket$ if and only if $u_i \in W^{2, p}([0, L_i], \mathbb C)$ and $v_i \in W^{1, p}([0, L_i], \mathbb C)$ for every $i \in \llbracket 1, N\rrbracket$.

We suppose from now on that $f_i \in W^{1, p}([0, L_i^\tau], \mathbb C)$ for every $i \in \llbracket 1, 2N\rrbracket$. Let $F_0 = (f_i(0))_{i \in \llbracket 1, 2N\rrbracket}$ and $F_L = (f_i(L_i^\tau))_{i \in \llbracket 1, 2N\rrbracket}$. The condition
\begin{equation}
\label{ConditionF}
f_i(0) = \sum_{j=1}^{2N} m_{ij}(t) f_j(L_j^\tau), \qquad \forall i \in \llbracket 1, 2N\rrbracket
\end{equation}
can be written as $F_0 = M(t) F_L$. Thanks to the outward decomposition of $\mathbb C^{2N}$, this is equivalent to $\Pi^q_{\mathrm{out}} D F_0 = \Pi^q_{\mathrm{out}} D M(t) F_L$ for every $q \in \mathcal V$. By Remark \ref{RemkMatM}, we have $\Pi^q_{\mathrm{out}} D M(t) = -\Pi^q_{\mathrm{out}} M^q(t) D$, and thus \eqref{ConditionF} is equivalent to
\begin{equation}
\label{EqProjQ}
\Pi^q_{\mathrm{out}} D F_0 + \Pi^q_{\mathrm{out}} M^q(t) D F_L = 0, \qquad \forall q \in \mathcal V.
\end{equation}

If $q \in \mathcal V_{\mathrm d}$, let $j$ be the index corresponding to the unique edge in $\mathcal E_q$. To simplify the notations, we consider here the case $\alpha(j) = q$, the other case being analogous. Then
\[
\begin{split}
 & \Pi^q_{\mathrm{out}} D F_0 + \Pi^q_{\mathrm{out}} M^q(t) D F_L = \Pi^q_{\mathrm{out}} D F_0 + \frac{1 - \eta_q(t)}{1 + \eta_q(t)} \Pi^q_{\mathrm{in}} D F_L \\
 = & f_{2j}(0) - \frac{1 - \eta_q(t)}{1 + \eta_q(t)} f_{2j-1}(L_j) = u^\prime_j(0) - v_j(0) - \frac{1 - \eta_q(t)}{1 + \eta_q(t)} \left(u_j^\prime(0) + v_j(0)\right) \\
 = & \frac{2}{1 + \eta_q(t)} \left(\eta_q(t) u_j^\prime(0) - v_j(0)\right),
\end{split}
\]
which shows that the left-hand side is equal to zero if and only if one has $v_j(q) = - \eta_q(t) \frac{d u_j}{d n_j}(q)$. If $q \in \mathcal V_{\mathrm u}$, the same argument shows that the left-hand side is equal to zero if and only if $v_j(q) = 0$.

Finally, if $q \in \mathcal V_{\mathrm{int}}$, one easily obtains that
\[
\Pi_{\mathrm{in}}^q D F_L = 
\begin{pmatrix}
\frac{d u_j}{d n_j}(q) - v_j(q)
\end{pmatrix}_{j \in \mathcal E_q}, \qquad
\Pi_{\mathrm{out}}^q D F_0 = 
\begin{pmatrix}
-\frac{d u_j}{d n_j}(q) - v_j(q)
\end{pmatrix}_{j \in \mathcal E_q}.
\]
Since $\Pi^q_{\mathrm{out}} \left(\Pi^q_{\mathrm{out}}\right)^{\mathrm{T}} = \id_{W_{\mathrm{out}}^q}$, one has
\[
\begin{split}
& \Pi^q_{\mathrm{out}} D F_0 + \Pi^q_{\mathrm{out}} M^q(t) D F_L \\
 = & \begin{pmatrix}
-\frac{d u_j}{d n_j}(q) - v_j(q)
\end{pmatrix}_{j \in \mathcal E_q} + \left(\id_{n_q} - \frac{2}{n_q} J_{n_q}\right)
\begin{pmatrix}
\frac{d u_j}{d n_j}(q) - v_j(q)
\end{pmatrix}_{j \in \mathcal E_q} \\
 = & 
\begin{pmatrix}
- 2 v_j(q) - \frac{2}{n_q} \sum_{k \in \mathcal E_q} \left(\frac{d u_k}{d n_k}(q) - v_k(q)\right)
\end{pmatrix}_{j \in \mathcal E_q}.
\end{split}
\]
The right-hand side is equal to zero if and only if $v_j(q) = v_k(q)$ for every $j, k \in \mathcal E_q$ and $\sum_{k \in \mathcal E_q} \frac{d u_k}{d n_k}(q) = 0$.

Collecting all the equivalences corresponding to the identities in \eqref{EqProjQ}, we conclude that \eqref{DomBtF} holds.

Let now $f \in D(B(t))$ and denote $(u, v) = T^{-1} f \in D(A(t))$, $g = B(t) f$. Then
\[
g = T A(t) T^{-1} f = T A(t) (u, v) = T (v, u^{\prime\prime}),
\]
and so, by \eqref{DefT}, for every $j \in \llbracket 1, 2N\rrbracket$,
\begin{align*}
g_{2j-1}(x) & = v_j^\prime(L_j - x) + u_j^{\prime\prime}(L_j - x) \\
 & = - \frac{d}{dx} \left(v_j(L_j - x) + u_j^{\prime}(L_j - x)\right) = - f_{2j-1}^\prime(x), \\
g_{2j}(x) & = v_j^\prime(x) - u_j^{\prime\prime}(x) = \frac{d}{dx} \left(v_j(x) - u_j^\prime(x)\right) = -f_{2j}^\prime(x),
\end{align*}
which shows that \eqref{BtFeqFPrime} holds.
\end{proof}

The operator $T: \mathsf X_p^\omega \to \mathsf Y_p(R)$ transforms \eqref{WaveFunctional} into
\[
\dot F(t) = B(t) F(t).
\]
This evolution equation corresponds to the system of transport equations
\begin{equation}
\label{EquivSystTransport}
\left\{
\begin{aligned}
 & \pad{f_i}{t}(t, x) + \pad{f_i}{x}(t, x) = 0, & \qquad & i \in \llbracket 1, 2N\rrbracket,\; t \in \left[0, +\infty\right),\; x \in [0, L_i^\tau], \\
 & f_i(t, 0) = \sum_{j=1}^{2N} m_{ij}(t) f_j(t, L_j^\tau), & & i \in \llbracket 1, 2N\rrbracket,\; t \in \left[0, +\infty\right),
\end{aligned}
\right.
\end{equation}
where $F(t) = (f_i(t))_{i \in \llbracket 1, 2N\rrbracket}$. The following property of the matrix $M(t)$ will be useful in the sequel.

\begin{lemma}
\label{LemmMAlmostOrth}
For every $t \in \mathbb R$,
\[M(t)^{\mathrm{T}} M(t) = \id_{2N} - \sum_{q \in \mathcal V_{\mathrm d}} \frac{4 \eta_q(t)}{(1 + \eta_q(t))^2} (\Pi_\mathrm{in}^q)^{\mathrm{T}} \Pi_{\mathrm{in}}^q.\]
\end{lemma}

\begin{proof}
Notice that, for every $q \in \mathcal V$, $M^q(t)$ can be written as
\[
M^q(t) = (\Pi_{\mathrm{out}}^q)^{\mathrm{T}} \left(\lambda_q(t) \id_{n_q} - \frac{2}{n_q} \delta_q J_{n_q}\right) \Pi_{\mathrm{in}}^q,
\]
where $\lambda_q(t) = \frac{1 - \eta_q(t)}{1 + \eta_q(t)}$ if $q \in \mathcal V_{\mathrm d}$ and $\lambda_q(t) = 1$ otherwise, while $\delta_q = 1$ if $q \in \mathcal V_{\mathrm{int}}$ and $\delta_q = 0$ otherwise. By a straightforward computation, one verifies that, for every $q \in \mathcal V$,
\[\left(\lambda_q(t) \id_{n_q} - \frac{2}{n_q} \delta_q J_{n_q}\right)^{\mathrm{T}} \left(\lambda_q(t) \id_{n_q} - \frac{2}{n_q} \delta_q J_{n_q}\right) = \lambda_q(t)^2 \id_{n_q}.\]
Noticing furthermore that, for every $q_1, q_2 \in \mathcal V$, $\Pi_{\mathrm{out}}^{q_1} (\Pi_{\mathrm{out}}^{q_2})^{\mathrm{T}} = \delta_{q_1 q_2} \id_{W_{\mathrm{out}}^{q_1}}$, one deduces that
\[M(t)^{\mathrm{T}} M(t) = D \left[\sum_{q \in \mathcal V} \lambda_q(t)^2 (\Pi_{\mathrm{in}}^{q})^{\mathrm{T}} \Pi_{\mathrm{in}}^{q}\right] D.\]
Since the term between brackets in the above equation is diagonal and $\lambda_q(t)^2 = 1 - \frac{4 \eta_q(t)}{(1 + \eta_q(t))^2}$ for $q \in \mathcal V_{\mathrm d}$, the conclusion follows.
\end{proof}

\subsection{Existence of solutions}

Thanks to the operator $T: \mathsf X_p^\omega \to \mathsf Y_p(R)$, one can give the following definition for solutions of \eqref{MainSystWave}.

\begin{definition}
Let $U_0 \in \mathsf X_p^\omega$ and $\eta = (\eta_q)_{q \in \mathcal V_{\mathrm d}}$ be a measurable function with nonnegative components. We say that $U: \mathbb R_+ \to \mathsf X_p^\omega$ is a \emph{solution} of $\Sigma_\omega(\mathcal G, L, \eta)$ with initial condition $U_0$ if $T^{-1}U: \mathbb R_+ \to \mathsf Y_p(R)$ is a solution of \eqref{EquivSystTransport} with initial condition $T^{-1} U_0 \in \mathsf Y_p(R)$.
\end{definition}

For every $F_0 \in \mathsf Y_p(R)$, it follows from Proposition \ref{PropTranspWellPosed} that \eqref{EquivSystTransport} admits a unique solution $F: \mathbb R_+ \to \mathsf X_p^\tau$. In order to show that this solution remains in $\mathsf Y_p(R)$ for every $t \geq 0$, one needs to show that $\mathsf Y_p(R)$ is invariant under the flow of \eqref{EquivSystTransport}.

\begin{proposition}
\label{PropRMeqR}
For every $t \in \mathbb R$, $R M(t) = R$.
\end{proposition}

\begin{proof}
Thanks to the inward decomposition of $\mathbb C^{2N}$, we prove the proposition by showing that for every $q \in \mathcal V$ and $t \in \mathbb R$,
\begin{equation}
\label{RDMeqRDProj}
- R D (\Pi_{\mathrm{out}}^q)^{\mathrm{T}} \left[\lambda_q(t) \id_{n_q} - \frac{2}{n_q} \delta_q J_{n_q}\right] = R D (\Pi_{\mathrm{in}}^q)^{\mathrm{T}},
\end{equation}
where $\lambda_q(t)$ and $\delta_q$ are defined as in the proof of Lemma \ref{LemmMAlmostOrth}. Without loss of generality, it is enough to consider the case where $R$ is a line matrix, i.e., we consider a single elementary path $(q_1, \dotsc, q_n)$ in $\mathcal G$ with $q_1 = q_n$ or $q_1, q_n \in \mathcal V_{\mathrm u}$, with signature $s$. Then $R = (\rho_{j})_{j \in \llbracket 1, 2N\rrbracket}$ is given by $\rho_{2j-1} = \rho_{2j} = s(j)$ for $j \in \llbracket 1, N\rrbracket$. For $i \in \llbracket 1, n-1\rrbracket$, denote by $j_i$ the edge corresponding to $\{q_i, q_{i+1}\}$. Let us write $R = \sum_{i=1}^{n-1} s(j_i) (e_{2j_i - 1} + e_{2j_i})^{\mathrm{T}}$ and notice that
\[R D = \sum_{i=1}^{n-1} s(j_i) (-e_{2j_i - 1} + e_{2j_i})^{\mathrm{T}}.\]
By definition of the signature $s$, one has, for $i \in \llbracket 1, n-1\rrbracket$,
\[
\begin{aligned}
-s(j_i) e_{2j_i - 1}^{\mathrm{T}} & = e_{2j_i - 1}^{\mathrm{T}} \left[(\Pi_{\mathrm{in}}^{q_{i+1}})^{\mathrm{T}}\Pi_{\mathrm{in}}^{q_{i+1}} - (\Pi_{\mathrm{in}}^{q_{i}})^{\mathrm{T}}\Pi_{\mathrm{in}}^{q_{i}}\right], \\
s(j_i) e_{2j_i}^{\mathrm{T}} & = e_{2j_i}^{\mathrm{T}} \left[(\Pi_{\mathrm{in}}^{q_{i+1}})^{\mathrm{T}}\Pi_{\mathrm{in}}^{q_{i+1}} - (\Pi_{\mathrm{in}}^{q_{i}})^{\mathrm{T}}\Pi_{\mathrm{in}}^{q_{i}}\right],
\end{aligned}
\]
and
\[
\begin{aligned}
-s(j_i) e_{2j_i - 1}^{\mathrm{T}} & = e_{2j_i - 1}^{\mathrm{T}} \left[(\Pi_{\mathrm{out}}^{q_{i}})^{\mathrm{T}}\Pi_{\mathrm{out}}^{q_{i}} - (\Pi_{\mathrm{out}}^{q_{i+1}})^{\mathrm{T}}\Pi_{\mathrm{out}}^{q_{i+1}}\right], \\
s(j_i) e_{2j_i}^{\mathrm{T}} & = e_{2j_i}^{\mathrm{T}} \left[(\Pi_{\mathrm{out}}^{q_{i}})^{\mathrm{T}}\Pi_{\mathrm{out}}^{q_{i}} - (\Pi_{\mathrm{out}}^{q_{i+1}})^{\mathrm{T}}\Pi_{\mathrm{out}}^{q_{i+1}}\right].
\end{aligned}
\]
One deduces that
\[
\begin{split}
R D & = \sum_{i=1}^{n-1} (e_{2j_i - 1} + e_{2j_i})^{\mathrm{T}} \left[(\Pi_{\mathrm{in}}^{q_{i+1}})^{\mathrm{T}}\Pi_{\mathrm{in}}^{q_{i+1}} - (\Pi_{\mathrm{in}}^{q_{i}})^{\mathrm{T}}\Pi_{\mathrm{in}}^{q_{i}}\right] \\
 & = \sum_{i=1}^{n-1} (e_{2j_i - 1} + e_{2j_i})^{\mathrm{T}} \left[(\Pi_{\mathrm{out}}^{q_{i}})^{\mathrm{T}}\Pi_{\mathrm{out}}^{q_{i}} - (\Pi_{\mathrm{out}}^{q_{i+1}})^{\mathrm{T}}\Pi_{\mathrm{out}}^{q_{i+1}}\right].
\end{split}
\]

By using the above relations, Equation \eqref{RDMeqRDProj} can be rewritten as
\begin{multline}
\label{RMeqMReduced}
\left[\lambda_q(t) \id_{n_q} - \frac{2}{n_q} \delta_q J_{n_q}\right] \Pi_{\mathrm{out}}^{q} \sum_{i=1}^{n-1} \left(\delta_{q q_{i + 1}} - \delta_{q q_{i}}\right)(e_{2j_i - 1} + e_{2j_i}) \\ = \Pi_{\mathrm{in}}^{q} \sum_{i=1}^{n-1} \left(\delta_{q q_{i+1}} - \delta_{q q_i}\right) (e_{2j_i - 1} + e_{2j_i}).
\end{multline}
Such an identity is trivially satisfied if $q \notin \{q_1, \dotsc, q_n\}$. Assume now that either $q = q_i$ for some $i \in \llbracket 2, n-1\rrbracket$ or $q = q_1 = q_n$ (and in the latter case set $i = n$ and define $j_{n+1} = j_1$). In particular, $q \in \mathcal V_{\mathrm{int}}$ and $\lambda_q(t) = \delta_q = 1$. We therefore must prove that
\begin{multline}
\label{RMeqMCaseInt}
\left[\id_{n_{q_i}} - \frac{2}{n_{q_i}} J_{n_{q_i}}\right] \Pi_{\mathrm{out}}^{q_i} (e_{2j_{i-1}-1} + e_{2j_{i-1}} - e_{2j_{i}-1} - e_{2j_{i}}) \\ = \Pi_{\mathrm{in}}^{q_i} (e_{2j_{i-1}-1} + e_{2j_{i-1}} - e_{2j_{i}-1} - e_{2j_{i}}).
\end{multline}
By definition of $\Pi_{\mathrm{in}}^{q_i}$ and $\Pi_{\mathrm{out}}^{q_i}$, one has that
\[\Pi_{\mathrm{out}}^{q_i} (e_{2j_{i-1}-1} + e_{2j_{i-1}} - e_{2j_{i}-1} - e_{2j_{i}}) = \Pi_{\mathrm{in}}^{q_i} (e_{2j_{i-1}-1} + e_{2j_{i-1}} - e_{2j_{i}-1} - e_{2j_{i}}) = w,\]
where $w \in \mathbb C^{n_{q_i}}$ has all its coordinates equal to zero, except two of them, one equal to $1$ and the other one equal to $-1$. Hence $J_{n_{q_i}} w = 0$ and \eqref{RMeqMCaseInt} holds true.

It remains to treat the case $q \in \{q_1, q_n\} \subset \mathcal V_{\mathrm u}$. In this case, $\lambda_q(t) = 1$ and $\delta_q = 0$, and we furthermore assume, with no loss of generality, that $q = q_1$. We can rewrite \eqref{RMeqMReduced} as
\[
\Pi_{\mathrm{out}}^{q_1} (e_{2j_1 - 1} + e_{2j_1}) \\ = \Pi_{\mathrm{in}}^{q_1} (e_{2j_1 - 1} + e_{2j_1}),
\]
which holds true by definition of $\Pi_{\mathrm{in}}^{q_1}$ and $\Pi_{\mathrm{out}}^{q_1}$. This concludes the proof of the proposition.
\end{proof}

The main result of the section, given next, follows immediately from Propositions \ref{PropInvariantIff} and \ref{PropRMeqR}.

\begin{proposition}
Let $(\mathcal G, L)$ be a network, $p \in [1, +\infty]$, and $\eta = (\eta_q)_{q \in \mathcal V_{\mathrm d}}$ be a measurable function with nonnegative components. Then, for every $U_0 \in \mathsf X_p^\omega$, the system $\Sigma_\omega(\mathcal G, L, \eta)$ defined in \eqref{MainSystWave} admits a unique solution $U: \mathbb R_+ \to \mathsf X_p^\omega$.
\end{proposition}

\subsection{Stability of solutions}
\label{SecStabWave}

We next provide an appropriate definition of exponential type for \eqref{MainSystWave}.

\begin{definition}
Let $(\mathcal G, L)$ be a network, $p \in [1, +\infty]$, and $\mathcal D$ be a subset of the space of measurable functions $\eta = (\eta_q)_{q \in \mathcal V_{\mathrm d}}$ with nonnegative components. Denote by $\Sigma_\omega(\mathcal G, L, \mathcal D)$ the family of systems $\Sigma_\omega(\mathcal G, L, \eta)$ for $\eta \in \mathcal D$. We say that $\Sigma_\omega(\mathcal G, L, \mathcal D)$ is of \emph{exponential type} $\gamma$ in $\mathsf X_p^\omega$ if, for every $\varepsilon > 0$, there exists $K > 0$ such that, for every $\eta \in \mathcal D$ and $u_0 \in \mathsf X_p^\omega$, the corresponding solution $u$ of $\Sigma_\omega(\mathcal G, L, \eta)$ satisfies, for every $t \geq 0$,
\[
\norm{u(t)}_{p} \leq K e^{(\gamma + \varepsilon) t} \norm{u_0}_{p}.
\]
We say that $\Sigma_\omega(\mathcal G, L, \mathcal D)$ is \emph{exponentially stable} in $\mathsf X_p^\omega$ if it is of negative exponential type.
\end{definition}

Given $\mathcal D$ as in the above definition, we define
\[\mathcal M = \{M: \mathbb R \to \mathcal M_{2N}(\mathbb R) \;|\: M \text{ is given by \eqref{DefM} for some } \eta \in \mathcal D\}.\]
Thanks to the continuity of $T$ and $T^{-1}$ established in Proposition \ref{PropTOnto}, we remark that $\Sigma_\omega(\mathcal G, L, \mathcal D)$ is of exponential type $\gamma$ in $\mathsf X_p^\omega$ if and only if $\Sigma_\tau(L, \mathcal M)$ is of exponential type $\gamma$ in $\mathsf Y_p(R)$. As a consequence of Corollary \ref{CoroSilkTransport}, we have the following result in the case of arbitrarily switching dampings $\eta_q$, $q \in \mathcal V_{\mathrm d}$.

\begin{corollary}
\label{CoroWaveEquivLambdaL}
Let $(\mathcal G, \Lambda)$ be a network, $d = \#\mathcal V_{\mathrm{d}}$, $\mathfrak D$ a subset of $(\mathbb R_+)^d$, and $\mathcal D = L^\infty(\mathbb R, \mathfrak D)$. The following statements are equivalent.
\begin{enumerate}[label={\bf \roman*.}, ref={(\roman*)}]
\item $\Sigma_\omega(\mathcal G, \Lambda, \mathcal D)$ is exponentially stable in $\mathsf X_p^\omega$ for some $p \in [1, +\infty]$.
\item $\Sigma_\omega(\mathcal G, L, \mathcal D)$ is exponentially stable in $\mathsf X_p^\omega$ for every $L \in V_+(\Lambda)$ and $p \in [1, +\infty]$.
\end{enumerate}
\end{corollary}

We can now provide a necessary and sufficient condition on $\mathcal G$ and $\mathfrak D$ for the exponential stability of $\Sigma_\omega(\mathcal G, \Lambda, L^\infty(\mathbb R, \mathfrak D))$.

\begin{theorem}
\label{TheoWaveStabIffTopo}
Let $(\mathcal G, \Lambda)$ be a network, $d = \#\mathcal V_{\mathrm{d}}$, $\mathfrak D$ a bounded subset of $(\mathbb R_+)^d$, and $\mathcal D = L^\infty(\mathbb R, \mathfrak D)$. Then $\Sigma_\omega(\mathcal G, \Lambda, \mathcal D)$ is exponentially stable in $\mathsf X_p^\omega$ for some $p \in [1, +\infty]$ if and only if $\mathcal G$ is a tree, $\mathcal V_{\mathrm u}$ contains only one vertex, and $\overline{\mathfrak D} \subset (\mathbb R_+^\ast)^d$.
\end{theorem}

\begin{proof}
Similarly to Remark \ref{RemkBClosed}, the exponential stability of $\Sigma_\omega(\mathcal G, \Lambda, \mathcal D)$ is equivalent to that of $\Sigma_\omega(\mathcal G, \Lambda, L^\infty(\mathbb R, \overline{\mathfrak D}))$. We therefore assume with no loss of generality that $\mathfrak D$ is compact.

Suppose that either $\mathcal G$ is not a tree, $\mathcal V_{\mathrm u}$ contains more than one vertex, or $\mathfrak D$ contains a point $\overline\eta$ with $\overline\eta_{\overline q} = 0$ for some $\overline q \in \mathcal V_{\mathrm d}$. Let $(q_1, \dotsc, q_n)$ be an elementary path in $\mathcal G$ with $q_1 = q_n$, $q_1, q_n \in \mathcal V_{\mathrm u}$, or $q_1 \in \mathcal V_{\mathrm u}$ and $q_n = \overline q$. Let $s$ be its signature and, for $i \in \llbracket 1, n-1\rrbracket$, let $j_i$ be the index corresponding to the edge $\{q_i, q_{i+1}\}$. Take $L \in V_+(\Lambda) \cap \mathbb N^N$, which is possible thanks to Proposition \ref{PropRatDep}. For $j \in \llbracket 1, N\rrbracket$, we define
\[
u_j(t, x) = 
\left\{
\begin{aligned}
& s(j_i) \sin(2 \pi t) \sin(2 \pi x), & & \text{ if } j = j_i \text{ for a certain } i \in \llbracket 1, n-1\rrbracket, \\
& 0, & & \text{ otherwise}.
\end{aligned}
\right.
\]
One easily checks that $(u_j)_{j \in \llbracket 1, N\rrbracket}$ is a solution of $\Sigma_\omega(\mathcal G, L, \eta)$ for every $\eta \in \mathcal D$. Since it is periodic and nonzero, $\Sigma_\omega(\mathcal G, L, \mathcal D)$ is not exponentially stable in $\mathsf X_p^\omega$ for any $p \in [1, +\infty]$, and so, by Corollary \ref{CoroWaveEquivLambdaL}, $\Sigma_\omega(\mathcal G, \Lambda, \mathcal D)$ is not exponentially stable in $\mathsf X_p^\omega$ for any $p \in [1, +\infty]$.

Suppose now that $\mathcal G$ is a tree, $\mathcal V_{\mathrm u}$ contains only one vertex, and $\mathfrak D = \overline{\mathfrak D} \subset (\mathbb R_+^\ast)^d$. Up to changing the orientation of $\mathcal G$, we assume that $\alpha(j) = q$ for every $q \in \mathcal V_{\mathrm u}$ and $j \in \mathcal E_q$. Let $\eta_{\min} = \min_{\eta \in \mathfrak D} \min_{q \in \mathcal V_{\mathrm d}} \eta_q \allowbreak> 0$. Let $U = (u, v)$ be a solution of $\Sigma_\omega(\mathcal G, \Lambda, \mathcal D)$ in $\mathsf X_2^\omega$ and $f = \frac{1}{\sqrt{2}} T U$. Notice that $\norm{f(t)}_{\mathsf Y_2(R)} = \norm{U(t)}_{2}$ thanks to Remark \ref{RemkTUnit}. For $t \geq 0$, denote $F_0(t) = (f_i(t, 0))_{i \in \llbracket 1, 2N\rrbracket}$ and $F_\Lambda(0) = (f_i(t, \allowbreak\Lambda_i^\tau))_{i \in \llbracket 1, 2N\rrbracket}$, so that $F_0(t) = M(t) F_\Lambda(t)$. For $t \geq 0$ and $s \in [0, \Lambda_{\min}]$, we have, by Lemma \ref{LemmMAlmostOrth},
\begin{align*}
\norm{U(t + s)}_{2}^2 & = \sum_{i=1}^{2N} \int_0^{\Lambda_i^\tau} \abs{f_i(t + s, x)}^2 dx \displaybreak[0] \\
& = \sum_{i=1}^{2N} \int_s^{\Lambda_i^\tau} \abs{f_i(t, x - s)}^2 dx + \int_0^s \abs{F_0(t + s - x)}_2^2 dx \displaybreak[0] \\
& = \sum_{i=1}^{2N} \int_s^{\Lambda_i^\tau} \abs{f_i(t, x - s)}^2 dx + \int_0^s \abs{F_\Lambda(t + s - x)}_2^2 dx \\
& \hphantom{=} {} - \int_0^{s} \sum_{q \in \mathcal V_{\mathrm d}} \sum_{i \in \mathcal E_q} \frac{4 \eta_q(t + s - x)}{(1 + \eta_q(t + s - x))^2} \abs{f_{2i-1}(t + s - x, \Lambda_i)}^2 dx \displaybreak[0] \\
& = \norm{U(t)}_{2}^2 - \sum_{q \in \mathcal V_{\mathrm d}} \sum_{i \in \mathcal E_q} \int_t^{t + s} \frac{4 \eta_q(\tau)}{(1 + \eta_q(\tau))^2} \abs{f_{2i-1}(\tau, \Lambda_i)}^2 d\tau,
\end{align*}
and, since
\[f_{2i-1}(\tau, \Lambda_i) = \frac{\pad{u_i}{x}(\tau, 0) + v_i(\tau, 0)}{\sqrt{2}} = \frac{1 + \eta_q(\tau)}{\sqrt{2}} \pad{u_i}{x}(\tau, 0), \qquad \forall q \in \mathcal V_{\mathrm d},\; \forall i \in \mathcal E_q,\]
we conclude that
\begin{equation}
\label{DerivEnergy}
\norm{U(t + s)}_{2}^2 = \norm{U(t)}_{2}^2 - \sum_{q \in \mathcal V_{\mathrm d}} \sum_{i \in \mathcal E_q} \int_t^{t + s} 2 \eta_q(\tau) \abs{\pad{u_i}{x}(\tau, 0)}^2 d\tau.
\end{equation}
Since \eqref{DerivEnergy} holds for every $t \geq 0$ and every $s \in [0, \Lambda_{\min}]$, one can easily obtain by an inductive argument that it holds for every $t \geq 0$ and $s \geq 0$. Hence, for every $t \geq 0$ and $s \geq 0$,
\[
\norm{U(t + s)}_{2}^2 - \norm{U(t)}_{2}^2 \leq - 2 \eta_{\min} \sum_{q \in \mathcal V_{\mathrm d}} \sum_{i \in \mathcal E_q} \int_t^{t + s} \abs{\pad{u_i}{x}(\tau, 0)}^2 d\tau.
\]
We can thus proceed as in \cite[Chapter 4, Section 4.1]{Dager2006Wave} (see also \cite{Schmidt1992Modelling}) to obtain the following observability inequality: there exist $c > 0$ and $\ell > 0$ such that, for every $t \geq 0$,
\[
\sum_{q \in \mathcal V_{\mathrm d}} \sum_{i \in \mathcal E_q} \int_t^{t + \ell} \abs{\pad{u_i}{x}(\tau, 0)}^2 d\tau \geq c \norm{U(t + \ell)}_{2}^2.
\]
This yields the desired exponential convergence in $\mathsf X_2^\omega$, and hence in $\mathsf X_p^\omega$ for every $p \in [1, +\infty]$ thanks to Corollary \ref{CoroWaveEquivLambdaL}.
\end{proof}



\makeatletter
\renewcommand{\@tocwrite}[2]{}
\makeatother

\end{document}